\documentclass[11pt]{article}
\usepackage{amsmath}
\usepackage{amsthm}
\usepackage[xetex]{graphicx}
\usepackage{epstopdf}
\usepackage{amssymb,enumerate}
\usepackage{mathrsfs}
\usepackage{calligra}
 \usepackage{graphicx}
 \usepackage{latexsym}
 \usepackage{xcolor}
 \usepackage{geometry}
 \usepackage{booktabs}

\newtheorem{thm}{Theorem}[section]

\newtheorem{lem}[thm]{Lemma}
\newtheorem{prop}[thm]{Proposition}
\newtheorem{Corollary}{Corollary}
\newtheorem{defn}[thm]{Definition}

\makeatletter

\newcommand{\Rmnum}[1]
{expandafter\@slowromancap\romannumeral#1@}
\makeatother
\numberwithin{equation}{section}
\numberwithin{figure}{section}
\allowdisplaybreaks[4]
\begin{document}

\title{$\mathcal{D}$-Pullback Measure Attractors of McKean-Vlasov Stochastic Lattice Wave Equations}
\author{Huixia Wen$^{1,2}$,\,\,\,Ji Shu$^{1}$\thanks{Corresponding author:
		shuji@sicnu.edu.cn(J.Shu),\  wenhuixia@cafuc.edu.cn(H.Wen).},\,\,\,Gang He$^{1}$\small\\
	\textit{1.School of Mathematical Sciences, Sichuan Normal University},\\ \textit{Chengdu 610068, PR China}	\small\\
	\textit{2.Faculty of Science, Civil Aviation Flight University of China},\\ \textit{Chengdu 618307, PR China}}
\date{}
\maketitle
 {\begin{center}
\begin{minipage}{6.0in}
{\textbf{Abstract:} This paper deals with the pullback measure attractor for a non-autonomous distribution-dependent stochastic lattice wave equation with nonlinear noise on $\mathbb{Z}$. We establish well-posedness in the energy space $\ell^{2}\times\ell^{2}$ and show that the solution generates a continuous non-autonomous dynamical system on $\mathcal{P}_{4}(\ell^{2}\times\ell^{2})$, which possesses a unique pullback measure attractor 	$\mathcal{A}$. Furthermore, we prove the convergence of the attractor $\mathcal{A}^{\varepsilon}$ as $\varepsilon\to 0^{+}$ to the attractor of the corresponding distribution-independent limit system.}
\end{minipage}
\end{center}}

 \section{Introduction}
 \indent
 \par
In this paper, we study the pullback measure attractors and their limiting behaviour as a perturbation parameter tends to zero, for the following non-autonomous, distribution-dependent stochastic lattice wave equation with nonlinear noise defined on the integer set $\mathbb{Z}$: for $\tau \in \mathbb{R}$ and $t>\tau$,

\begin{equation}\label{eq:1.1}
	\left\{
	\begin{aligned}
		&d\dot{u}_{i}(t)+\alpha \dot{u}_{i}(t)\,dt-\nu\big(u_{i-1}(t)-2u_{i}(t)+u_{i+1}(t)\big)\,dt+\lambda u_{i}(t)\,dt\\
		&\quad=\big(f_{i}(t,u_{i}(t),\mathcal{L}_{u_{i}(t)})+g_{i}(t)\big)\,dt+\sigma_{i}(t,u_{i}(t),\mathcal{L}_{u_{i}(t)})\,dW_{i}(t),\quad t>\tau,\\
		&u_{i}(\tau)=u_{i,\tau},\ \ \dot{u}_{i}(\tau)=\dot{u}_{i,\tau},
	\end{aligned}
	\right.
\end{equation}
where $\dot{u}_{i}$ denotes the first-order time-derivative of $u_{i}$, $u=(u_{i})_{i\in\mathbb{Z}}$ is the unknown sequence, the initial sequences $u_{\tau}=(u_{i,\tau})_{i\in\mathbb{Z}}$ and $\dot{u}_{\tau}=(\dot{u}_{i,\tau})_{i\in\mathbb{Z}}$ are given, and $\alpha,\nu,\lambda$ are positive constants. The sequence $g=(g_{i})_{i\in\mathbb{Z}}$ is a time-dependent random sequence; for every $i\in\mathbb{Z}$, $f_{i}$ and $\sigma_{i}$ are nonlinear mappings; $\mathcal{L}_{u_{i}}$ denotes the law (distribution) of the random variable $u_{i}$, and $(W_{i}(t))_{i\in\mathbb{Z}}$ is a sequence of independent standard two-sided real-valued Wiener processes defined on a complete filtered probability space $(\Omega,\mathcal{F},\{\mathcal{F}_{t}\}_{t\in\mathbb{R}},\mathbb{P})$ satisfying the usual conditions. The natural phase space of \eqref{eq:1.1} is the energy space $\ell^{2}\times\ell^{2}$, where the second component accounts for the velocity $\dot{u}$.\par 

Lattice dynamical systems have broad applications in physics, biology, and engineering \cite{1,2,3}; the long-term dynamics of deterministic and stochastic lattice systems have been extensively studied \cite{4,5,6,7}. 
Stochastic lattice equations with additive or linear multiplicative noise can be transformed into pathwise deterministic systems, thereby allowing the study of pathwise random attractors \cite{8,9,16}.  
For nonlinear noise such a transformation is unavailable, which led to the development of mean-square random dynamical systems, weak pullback mean random attractors(see \cite{53,10,11,12,13} for applications), and pullback measure attractors (see \cite{41,42,43,44,45,46} for applications).  
Recently, pullback measure attractors and evolution systems of measures have been established for a complex-valued $p$-Laplacian Ginzburg--Landau lattice system driven by superlinear L\'evy noise \cite{14}.  
The nonlinear noise case has been studied extensively for first-order SPDEs, but remains less developed for second-order equations; we refer to \cite{11,12} and references therein.\par

Beyond these contributions, the statistical dynamics of stochastic lattice systems, including the existence, ergodicity, and mixing of invariant measures, have become a focus of research in recent years; see, for example, \cite{15,18} for results on invariant measures in unbounded domains and on limiting behaviour of invariant measures, and \cite{19} for weak pullback mean random attractors and limiting dynamics of stochastic lattice systems. For related large deviation principles, we refer the reader to \cite{55,56} and the references therein.\par

Unlike previous studies, this work focuses on the non-autonomous distribution-dependent (also known as McKean--Vlasov) stochastic lattice wave equation with nonlinear noise. McKean--Vlasov stochastic differential equations (MVSDEs), first introduced in \cite{20,21}, often arise from interacting particle systems \cite{22}. A distinctive feature of these equations is their dependence not only on the states of the solutions but also on their distributions; consequently the Markov operators associated with MVSDEs do not form semigroups (see, e.g., \cite{31}), which makes the conventional methods for distribution-independent stochastic differential equations inapplicable.\par

The fundamental theory of MVSDEs which include existence, uniqueness, Bismut--Elworthy--Li formula, large deviations, and existence/ergodicity of invariant measures, has been extensively developed \cite{20,23,24,25,26,27,28,29,30}.  
For infinite-dimensional cases, pullback measure attractors were established in \cite{31,32,17}.  
Most recently, existence and upper semi-continuity of pullback measure attractors and invariant measures, Wasserstein convergence rates, and finite-dimensional approximation have been obtained for a superlinear McKean--Vlasov lattice system driven by L\'evy noise \cite{33}.\par

The concept of measure attractor was introduced by Schmalfuss \cite{34} for stochastic Navier--Stokes equations with additive noise and was further developed in various settings \cite{35,36,37,38,39}.  
It has been extended to non-autonomous frameworks and applied to a wide range of stochastic lattice and PDE systems \cite{41,42,43,44,45,46,47,48,49,50}.  
Recently, the construction of an enlarged measure attractor on the Banach space $\ell^{p}$ ($p>2$) for a superlinear stochastic fractional discrete $p$-Laplace equation \cite{40}.  
Meanwhile, a minimal pullback measure attractor (unique and convex) was obtained for fractional stochastic FitzHugh--Nagumo equations with superlinear space--time noise on $\mathbb{R}^{N}$ \cite{54}.  
The present work uses the $\mathcal{D}$-pullback measure attractor framework of \cite{31,32,17}.\par

In this article, we obtain the existence of pullback measure attractors for the non-autonomous McKean--Vlasov stochastic lattice wave equation \eqref{eq:1.1}. Compared with the related works \cite{14,31,32,33,40,41,42,43,44,45,46,47,48,49,50}, the present paper addresses a setting that has not been treated in the literature, namely the combination of (i) a second-order (wave) structure with a damping term $\alpha\dot{u}_{i}$ in the energy space $\ell^{2}\times\ell^{2}$, (ii) distribution dependence in both the drift $f_{i}$ and the diffusion $\sigma_{i}$, and (iii) infinite-dimensional nonlinear noise in the sense that $\sigma_{i}$ depends on the state and the spatial index $i$ ranges over the whole lattice $\mathbb{Z}$. The first goal of this paper is the existence and uniqueness of the pullback measure attractor $\mathcal{A}=\{\mathcal{A}(\tau)\}_{\tau\in\mathbb{R}}$ for \eqref{eq:1.1}.\par

To establish the existence of pullback measure attractors for \eqref{eq:1.1}, we first establish the well-posedness of \eqref{eq:1.1} in the energy space $\ell^{2}\times\ell^{2}$, namely the existence, uniqueness, and continuous dependence on initial data of solutions. The main technical challenge at this step is to handle simultaneously the second-order time-derivative, the distribution-dependent coefficients, and the infinite-dimensional nonlinear noise. We address the well-posedness by constructing a sequence of successive approximations in time, combined with the It\^o formula applied to the energy functional $E(u)=\frac{1}{2}\big(\|u\|^{2}+\|\dot{u}\|^{2}\big)$ (see, e.g., \cite{51}) , along with the BDG inequality and Gronwall's lemma. This yields the existence and uniqueness of strong solutions in $L^2(\Omega, C([\tau, T], \ell^2\times\ell^2))$. The weak uniqueness is then derived by the Yamada–Watanabe theorem.\par

Next, we introduce, on the space $\mathcal{P}_{4}(\ell^{2}\times\ell^{2})$ of probability measures with finite fourth moment, the family of mappings $\{P^{*}_{\tau,t}\}_{t\geq\tau}$ associated with the solutions of \eqref{eq:1.1}, and we show that it forms a continuous non-autonomous dynamical system on $\mathcal{P}_{4}(\ell^{2}\times\ell^{2})$ by means of the abstract criterion of \cite{52}. The continuity of $\{P^{*}_{\tau,t}\}_{t\geq\tau}$ is the second main technical challenge: in the distribution-dependent setting, $P^{*}_{\tau,t}$ is in general no longer the adjoint of the Markov operator $P_{\tau,t}$ in the sense of measures, that is, in general $\int\varphi\,\mathrm{d} P^{*}_{\tau,t}\mu \neq \int(\varphi\circ P_{\tau,t})\,\mathrm{d}\mu$, so the standard Feller-property argument for continuity on the whole measure space does not apply. Following the strategy of \cite{31}, we apply Vitali's theorem to establish the continuity of $\{P^{*}_{\tau,t}\}_{t\geq\tau}$ on a bounded closed subset $\mathcal{B}_{\mathcal{P}_{4}(\ell^{2}\times\ell^{2})}(r)$ of the measure space, rather than on the entire space, by exploiting the fourth-moment bound on the solutions. This requires an $\ell^{4}$ estimate on the solutions, which is also needed for the construction of a $\mathcal{D}$-pullback absorbing set in $\mathcal{P}_{4}(\ell^{2}\times\ell^{2})$.\par

A further difficulty is to establish the $\mathcal{D}$-pullback asymptotic compactness of $\{P^{*}_{\tau,t}\}_{t\geq\tau}$ on $\mathcal{P}_{4}(\ell^{2}\times\ell^{2})$. Since the equation is defined on the infinite-dimensional lattice $\ell^{2}$ and Sobolev-type compactness arguments do not apply, the pullback asymptotic compactness does not follow from the uniform $\ell^{2}\times\ell^{2}$ estimates alone. We overcome this by establishing uniform tail estimates on the solutions of \eqref{eq:1.1}, that is, we show that the tails of the solutions in $L^{2}(\Omega;\ell^{2}\times\ell^{2})$ are uniformly small with respect to the initial data and the time variable once the spatial index is sufficiently large, and we then apply the Arzel\`a--Ascoli theorem to prove compactness. The uniform tail estimates are the third main technical challenge: due to the simultaneous presence of distribution dependence in both the drift and the diffusion and of infinite-dimensional nonlinear noise, the derivation of these tail estimates is more delicate than in the cases treated in \cite{10,16,44}.\par

The second goal of this paper is to investigate the limiting behaviour of the pullback measure attractors constructed above. To this end, we introduce the following family of equations parametrized by $\varepsilon\in(0,1)$, obtained from \eqref{eq:1.1} by replacing the distribution-dependent coefficients $f_{i},\sigma_{i}$ with perturbations $f_{i}^{\varepsilon},\sigma_{i}^{\varepsilon}$ that still depend on the law $\mathcal{L}_{u_{i}^{\varepsilon}}$ but converge to $f_{i},\sigma_{i}$ in a suitable sense as $\varepsilon\to 0^{+}$: for $t>\tau$,
\begin{equation}\label{eq:1.2}
	\left\{
	\begin{aligned}
		&d\dot{u}_{i}^{\varepsilon}(t)+\alpha \dot{u}_{i}^{\varepsilon}(t)\,dt-\nu\big(u_{i-1}^{\varepsilon}(t)-2u_{i}^{\varepsilon}(t)+u_{i+1}^{\varepsilon}(t)\big)\,dt+\lambda u_{i}^{\varepsilon}(t)\,dt\\
		&\quad=\big(f_{i}^{\varepsilon}(t,u_{i}^{\varepsilon}(t),\mathcal{L}_{u_{i}^{\varepsilon}(t)})+g_{i}(t)\big)\,dt+\sigma_{i}^{\varepsilon}(t,u_{i}^{\varepsilon}(t),\mathcal{L}_{u_{i}^{\varepsilon}(t)})\,dW_{i}(t),\quad t>\tau,\\
		&u_{i}^{\varepsilon}(\tau)=u_{i,\tau},\ \ \dot{u}_{i}^{\varepsilon}(\tau)=\dot{u}_{i,\tau},
	\end{aligned}
	\right.
\end{equation}
and the corresponding distribution-independent limit equation, which is obtained from \eqref{eq:1.1} by dropping the dependence on $\mathcal{L}$:
\begin{equation}\label{eq:1.3}
	\left\{
	\begin{aligned}
		&d\dot{u}_{i}^{0}(t)+\alpha \dot{u}_{i}^{0}(t)\,dt-\nu\big(u_{i-1}^{0}(t)-2u_{i}^{0}(t)+u_{i+1}^{0}(t)\big)\,dt+\lambda u_{i}^{0}(t)\,dt\\
		&\quad=\big(f_{i}(t,u_{i}^{0}(t))+g_{i}(t)\big)\,dt+\sigma_{i}(t,u_{i}^{0}(t))\,dW_{i}(t),\quad t>\tau,\\
		&u_{i}^{0}(\tau)=u_{i,\tau},\ \ \dot{u}_{i}^{0}(\tau)=\dot{u}_{i,\tau},
	\end{aligned}
	\right.
\end{equation}
\par
We denote by $\mathcal{A}^{\varepsilon}=\{\mathcal{A}^{\varepsilon}(\tau)\}_{\tau\in\mathbb{R}}$ and $\mathcal{A}^{0}=\{\mathcal{A}^{0}(\tau)\}_{\tau\in\mathbb{R}}$ the pullback measure attractors of \eqref{eq:1.2} and \eqref{eq:1.3}, respectively, in $\mathcal{P}_{4}(\ell^{2}\times\ell^{2})$. Under the uniform assumptions on $(f_{i}^{\varepsilon},\sigma_{i}^{\varepsilon})$ and suitable convergence as $\varepsilon\to 0^{+}$, we prove that $\mathcal{A}^{\varepsilon}$ is upper semi-continuous with respect to $\mathcal{A}^{0}$, i.e., for every $\tau\in\mathbb{R}$,
\begin{equation}\label{eq:1.4}
	\lim_{\varepsilon\to 0^{+}} d_{\mathcal{P}_{4}(\ell^{2}\times\ell^{2})}\!\big(\mathcal{A}^{\varepsilon}(\tau),\mathcal{A}^{0}(\tau)\big)=0.
\end{equation}\par
This upper semi-continuity result is a natural robustness question for measure attractors and is closely related to the convergence of invariant measures between the perturbed and the limiting systems; we refer to \cite{7,14,32,33,40,41,44,48} for related conclusions in different settings.
\par
To make the contribution of the present paper relative to the most closely related works immediately clear, we summarize the comparison in the following table.
\begin{table}[h]
	\centering
	\small
	\caption{Comparison of the present paper with the most closely related recent works on pullback measure attractors. ``\checkmark'' indicates the feature is present in the cited work; ``--'' indicates it is not.}
	\label{tab:comparison}
	\begin{tabular}{lccccccc}
		\toprule
		& \textbf{This paper} & \cite{14} & \cite{31} & \cite{32} & \cite{33} & \cite{40} & \cite{54} \\
		\midrule
		McKean--Vlasov (distribution) dependence & \checkmark & -- & \checkmark & \checkmark & \checkmark & -- & -- \\
		Nonlinear noise & \checkmark & \checkmark & -- & -- & \checkmark & \checkmark & \checkmark \\
		Second-order (wave) structure & \checkmark & -- & -- & -- & -- & -- & -- \\
		Lattice setting (discrete) & \checkmark & \checkmark & -- & \checkmark & \checkmark & \checkmark & -- \\
		Banach phase space & -- & -- & -- & -- & -- & \checkmark & \checkmark \\
		Upper semi-continuity ($\varepsilon\to 0$) & \checkmark & \checkmark & -- & \checkmark & \checkmark & \checkmark & -- \\
		Two-stages method & -- & -- & -- & -- & -- & \checkmark & \checkmark \\
		Minimal \& convex attractor & -- & -- & -- & -- & -- & -- & \checkmark \\
		\bottomrule
	\end{tabular}
\end{table}
\par
The most distinctive features of the present paper are: (a) it is the first to establish pullback measure attractors for a McKean--Vlasov stochastic lattice wave equation, combining for the first time the three ingredients of M-V dependence, second-order wave structure, and infinite-dimensional nonlinear noise; (b) it establishes the upper semi-continuity of the attractor as the M-V perturbation parameter tends to zero, recovering the corresponding state-dependent limit attractor; and (c) it is set on the Hilbert space $\ell^{2}\times\ell^{2}$, leaving the Banach-space extension to \cite{40,54} and the minimal/convex refinement to \cite{54}.
\par
The outline of this paper is as follows. In Section 2 we collect some basic concepts and introduce the necessary assumptions. Section 3 is devoted to establishing the well-posedness of solutions for the distribution-dependent stochastic lattice wave equation \eqref{eq:1.1}--\eqref{eq:1.3}. In Section 4 we derive various uniform estimates of the solutions, including the uniform tail estimates. In Section 5, the existence and uniqueness of the pullback measure attractor $\mathcal{A}$ for equation \eqref{eq:1.1} is established. In Section 6, we prove the upper semi-continuity $\mathcal{A}^{\varepsilon}\to\mathcal{A}^{0}$ as $\varepsilon\to 0^{+}$.

Throughout this paper, $\|\cdot\|$ and $(\cdot,\cdot)$ denote the norm and the inner product of $\ell^{2}$, respectively, and $\|\cdot\|_{p}$ denotes the norm in $\ell^{p}$.

\section{Preliminaries}
\indent
\par
In this section, we recall some results regarding pullback measure attractors for non-autonomous dynamical systems on the space of probability measures$(see, e.g.,\cite{41},\cite{32})$.\par 
In what follows, we denote by $X$ a separable Banach space with norm $\|\cdot\|_{X}$. Let $C_{b}(X)$ be the space of bounded continuous function $\varphi: X\to \mathbb{R}$ endowed with the norm
\begin{equation*}
\|\varphi\|_{\infty}=\sup_{x\in X}\left|\varphi(x)\right|.
\end{equation*}
Donote by $L_{b}(X)$ the space of bounded Lipschitz functions on $X$ which consists of all functions $\varphi\in C_{b}(X)$ such that
\begin{equation*}
Lip(\varphi):=\sup_{x_{1},x_{2}\in X,x_{1}\neq x_{2}}\dfrac{\left|\varphi(x_{1})-\varphi(x_{2})\right|}{\|x_{1}-x_{2}\|_{X}}<\infty.
\end{equation*}
The space $L_{b}(X)$ is endowed with the norm
\begin{equation*}
\|\varphi\|_{L_b}=\|\varphi\|_{\infty}+Lip(\varphi).
\end{equation*}
Let $\mathcal{P}(X)$ be the set of probability measures on $(X,\mathcal{B}(X))$, where $\mathcal{B}(X)$ is the Borel $\sigma-$algebra of $X$. Given $\varphi\in C_{b}(X)$ and $\mu\in\mathcal{P}(X)$, we write
\begin{equation*}
	(\varphi,\mu)=\int_{X}\varphi(x)\mu(dx).
\end{equation*}\par
Recall that a sequence $\{\mu_{n}\}_{n=1}^{\infty}\subseteq\mathcal{P}(X)$ is weakly convergent to $\mu\in\mathcal{P}(X)$ if for every $\varphi\in C_{b}(X)$,
\begin{equation*}
	\lim_{n\rightarrow \infty}(\varphi,\mu_{n})=(\varphi,\mu).
\end{equation*}
Define a metric on $\mathcal{P}(X)$ by
\begin{equation*}
	d_{\mathcal{P}(X)}(\mu_{1},\mu_{2})=\sup_{\varphi\in L_{b}(X),\  \|\varphi\|_{L_{b}}\leq1}\left|(\varphi,\mu_{1})-(\varphi,\mu_{2})\right|,\ \ \forall\mu_{1},\mu_{2}\in\mathcal{P}(X).
\end{equation*}
Then $(\mathcal{P}(X),d_{\mathcal{P}(X)})$ is a polish space. Moreover, a sequence $\{\mu_n\}^{\infty}_{n=1}\subseteq\mathcal{P}(X)$ converges to $\mu$ in $(\mathcal{P}(X),d_{\mathcal{P}(X)})$ if and only if $\{\mu_n\}^{\infty}_{n=1}$ converges to $\mu$ weakly.
\par
For every$p\geq 1$, let $\mathcal{P}_{p}(X),\mathbb{W}_p^X$ denote a Polish space defined by
\begin{equation*}
	\mathcal{P}_{p}(X)=\left\{\mu\in\mathcal{P}(X): \mu(\|\cdot\|_X^p):=\int_{X}\|x\|^{p}_{X}\mu(dx)<\infty\right\}
\end{equation*}
and
\begin{equation*}
	\mathbb{W}_p^X(\mu,\nu)=\inf_{\pi\in\Pi(\mu,\nu)}\left(\int_{X\times X}\|x-y\|_X^p\pi(dx,dy)\right)^{\frac{1}{p}},\ \ \forall \mu,\nu\in\mathcal{P}_{p}(X),
\end{equation*}
where $\Pi(\mu,\nu)$ is the set of couplings of $\mu$ and $\nu$. The metric $\mathbb{W}_p^X$ is called the Wasserstein distance.
Given $r>0$, denote by
\begin{equation*}
	B_{\mathcal{P}_{p}(X)}(r)=\left\{\mu\in\mathcal{P}_{p}(X): \left(\int_{X}\|x\|^{p}_{X}\mu(dx)\right)^{\frac{1}{p}}<r\right\}.
\end{equation*}
A subset $V\subseteq\mathcal{P}_{p}(X)$ is bounded if there is $r>0$ such that $V\subseteq B_{\mathcal{P}_{p}(X)}(r)$. If $S$ is bounded in $\mathcal{P}_{p}(X)$, then we set
\begin{equation*}
	\|V\|_{\mathcal{P}_{p}(X)}(r)=\sup_{\mu\in V} \left(\int_{X}\|x\|^{p}_{X}\mu(dx)\right)^{\frac{1}{p}}.
\end{equation*}
\par
Note that $(\mathcal{P}_p(X), \mathbb{W}_p^X)$ is a Polish space, whereas $(\mathcal{P}_p(X), d_{\mathcal{P}(X)})$ is not complete. However, for every $r > 0$, the set $B_{\mathcal{P}_p(X)}(r)$ is a closed subset of $\mathcal{P}(X)$ with respect to the metric $d_{\mathcal{P}(X)}$. Consequently, the space $(B_{\mathcal{P}_p(X)}(r), d_{\mathcal{P}(X)})$ is complete for every $r > 0$.
\par 
Next, we recall the definition of non-autonomous dynamical systems in the metric space \((\mathcal{P}_p(X), d_{\mathcal{P}(X)})\) along with some related concepts.
\begin{defn}
\label{Def:2.1}
A family $S=\{S(t,\tau): t\in\mathbb{R}^{+}, \tau\in\mathbb{R}\}$ of mappings from $(\mathcal{P}_p(X), d_{\mathcal{P}(X)})$ to $(\mathcal{P}_p(X), d_{\mathcal{P}(X)})$ is called a continuous non-autonomous dynamical system on $(\mathcal{P}_p(X), d_{\mathcal{P}(X)})$, if for all $\tau\in\mathbb{R}$ and $t,s\in\mathbb{R}^{+}$, the following conditions are satisfied:\par
$(\mathbf{i})\ S(0,\tau)=I_{\mathcal{P}_{p}(X)}$, where $I_{\mathcal{P}_{p}(X)}$ is the identity operator on $(\mathcal{P}_p(X), d_{\mathcal{P}(X)})$;\par
$(\mathbf{ii})\ S(t+s,\tau)=S(t,s+\tau)\circ S(s,\tau)$.\par
\par
\end{defn}\par
If S is a non-autonomous dynamical system on $(\mathcal{P}_p(X), d_{\mathcal{P}(X)})$, and, for every bounded subset $S \subset \mathcal{P}_p(X)$, S is continuous from $(V, d_{\mathcal{P}(X)})$ to $(\mathcal{P}_p(X), d_{\mathcal{P}(X)})$, then we say that S is continuous on bounded subsets of $\mathcal{P}_p(X)$.\par 
Let $\mathcal{D}$ be a collection of families of nonempty bounded subsets $D(\tau) \subseteq \mathcal{P}_p(X) $, parametrized by  
$\tau \in \mathbb{R}$. Such a collection $\mathcal{D}$ is called inclusion-closed if $D = \{D(\tau) : \tau \in \mathbb{R}\} \in \mathcal{D} $ implies that every  
family $\tilde{D} = \{\tilde{D}(\tau) : \emptyset \neq \tilde{D}(\tau) \subseteq D(\tau), \, \forall \tau \in \mathbb{R}\}$ also belongs to $\mathcal{D}$.
\begin{defn}
	\label{Def:2.2}
	A family $K=\{K(\tau):\tau\in\mathbb{R}\}\in\mathcal{D}$ is called a $\mathcal{D}$-pullback absorbing set for $S$ if for each $\tau\in\mathbb{R}$ and every $D\in\mathcal{D}$, there exists $T=T(\tau,D)>0$ such that
		\begin{align*}
		S(t,\tau-t)D(\tau-t)\subseteq K(\tau),\ \forall\ t\geq T.
	\end{align*}
\end{defn}\par

\begin{defn}
	\label{Def:2.3}
The non-autonomous dynamical system $S$ is said to be $\mathcal{D}$-pullback asymptotically compact in $\mathcal{P}_{p}(X)$ if for each $\tau\in\mathbb{R},\{S(t_{n},\tau-t_{n})\mu_{n}\}^{\infty}_{n=1}$ has a convergent subsequence in $\mathcal{P}_{p}(X)$ wherever $t_{n}\rightarrow +\infty$ and $\mu_{n}\in D(\tau-t_{n})$ with $D\in\mathcal{D}$.
\end{defn}\par

\begin{defn}
	\label{Def:2.4}
	A family $\mathcal{A}=\{\mathcal{A}(\tau):\tau\in\mathbb{R}\}\in\mathcal{D}$ is called a $\mathcal{D}$-pullback measure attractor for $S$ if the following conditions are satisfied:\par
	$(\mathbf{i})\ \mathcal{A}(\tau)$ is compact in $\mathcal{P}_{p}(X)$ for each $\tau\in\mathbb{R}$;\par
	$(\mathbf{ii})\ \mathcal{A}$ is invariant, that is, $S(t,\tau)\mathcal{A}(\tau)=\mathcal{A}(\tau+t)$, for all $\tau\in\mathbb{R}$ and $t\in\mathbb{R}^+$;\par
	$(\mathbf{iii})\ \mathcal{A}$ attracts every set in $\mathcal{D}$, that is, for each $D=\{D(\tau):\tau\in\mathbb{R}\}\in\mathcal{D}$,
	\begin{align*}
		\lim_{t\to\infty}d_{\mathcal{P}_{p}(X)}(S(t,\tau-t)D(\tau-t),\mathcal{A}(\tau))=0,
	\end{align*}
	where $d_{\mathcal{P}_{p}(X)}(\cdot,\cdot)$ denotes the Hausdorff semi-metric between any two subsets of $\mathcal{P}_{p}(X)$.
	\end{defn}\par

Based on the above notation, from Theorem 2.25 and Proposition 3.6 in $\cite{52}$, we have the following criterion for the existence, uniqueness of $\mathcal{D}$-pullback measure attractors.

\begin{prop}
\label{pro:2.1}
Let $\mathcal{D}$ be an inclusion-closed collection of families of nonempty bounded subsets of $\mathcal{P}_{p}(X)$, and $S$ be a continuous non-autonomous dynamical system on $(\mathcal{P}_p(X), d_{\mathcal{P}(X)})$. Then $S$ has a unique $\mathcal{D}$-pullback measure attractor $\mathcal{A}\in\mathcal{D}$ in $(\mathcal{P}_p(X), d_{\mathcal{P}(X)})$ if and only if $S$ has a closed $\mathcal{D}$-pullback absorbing set $K\in\mathcal{D}$ and $S$ is $\mathcal{D}$-pullback asymptotically compact in $(\mathcal{P}_p(X), d_{\mathcal{P}(X)})$. 
\end{prop}
\par 
We now present some assumptions that will be used in subsequent estimates. Throughout this paper, we set $\delta_0$ for the Dirac probability measure at 0.\par  
$\textbf{(A1).\ Assumption on nonlinear term}\ f_i$. Suppose $f_i: \mathbb{R} \times \mathbb{R} \times \mathbb{R} \times \mathcal{P}_2(\mathbb{R}) \to \mathbb{R}$ is continuous, there exists a $\psi(\cdot) = (\psi_i(\cdot))_{i \in \mathbb{Z}} \in L^\infty_{loc}(\mathbb{R}, \ell^2)$ such that for any $t, u, u_1, u_2 \in \mathbb{R}$ and $\mu, \mu_1, \mu_2 \in \mathcal{P}_2(\mathbb{R})$,  
\begin{align}
	\label{eq:2.1}
	\theta(\cdot)=(\theta_i(\cdot))_{i \in \mathbb{Z}}:=(f_i(t,0,\delta_0))_{i \in \mathbb{Z}}\in L^\infty_{loc}(\mathbb{R}, \ell^2),
\end{align}
\begin{align}
\label{eq:2.2}
|f_i(t, u_1, \mu_1) - f_i(t, u_2, \mu_2)| \leq \psi_i(t)(|u_1 - u_2| + \mathbb{W}_2^{\mathbb{R}}(\mu_1, \mu_2)),
\end{align} 
\begin{align}
	\label{eq:2.3}
	\frac{\partial}{\partial u}f_i(t, u, \mu) \leq \eta_i(t),
\end{align} 
where $p\geq2,\ \eta(\cdot)=(\eta(\cdot))_{i \in \mathbb{Z}}\in L^\infty_{loc}(\mathbb{R}, \ell^2)$.\par 
By \eqref{eq:2.1}-\eqref{eq:2.3} we can get for all $t,u\in\mathbb{R},\ \mu\in\mathcal{P}_2(\mathbb{R})$, there exists a $\gamma(\cdot)=(\gamma_i(\cdot))_{i \in \mathbb{Z}}\in L^\infty_{loc}(\mathbb{R}, \ell^2)$,
\begin{align}
	\label{eq:2.4}
	\left|f_i(t, u, \mu)\right| \leq \gamma_i(t)(1+\left|u\right|+\mathbb{W}_2^{\mathbb{R}}(\mu, \delta_0)).
\end{align}\par 
$\textbf{(A2).\ Assumption on nonlinear term}\ g_i$. $g(\cdot)=(g_i(\cdot))_{i \in \mathbb{Z}}\in L^\infty_{loc}(\mathbb{R}, \ell^2)$.\par 
$\textbf{(A3).\ Assumption on diffusion term}\ \sigma_i$. Suppose $\sigma_i: \mathbb{R} \times \mathbb{R} \times \mathbb{R} \times \mathcal{P}_2(\mathbb{R}) \to \mathbb{R}$ is continuous, there exists a $\kappa(\cdot) = (\kappa_i(\cdot))_{i \in \mathbb{Z}} \in L^\infty(\mathbb{R}, \ell^2)$ such that for any $t, u, u_1, u_2 \in \mathbb{R}$ and $\mu, \mu_1, \mu_2 \in \mathcal{P}_2(\mathbb{R})$,  
\begin{align}
	\label{eq:2.5}
	\chi(\cdot)=(\chi_i(\cdot))_{i \in \mathbb{Z}}:=(\sigma_i(t,0,\delta_0))_{i \in \mathbb{Z}}\in L^\infty(\mathbb{R}, \ell^2),
\end{align}
\begin{align}
	\label{eq:2.6}
	|\sigma_i(t, u_1, \mu_1) - \sigma_i(t, u_2, \mu_2)| \leq \kappa_i(t)(|u_1 - u_2| + \mathbb{W}_2^{\mathbb{R}}(\mu_1, \mu_2)),
\end{align}
then from \eqref{eq:2.5} and \eqref{eq:2.6} we have,
\begin{align}
	\label{eq:2.7}
	\left|\sigma_i(t, u, \mu)\right| \leq \kappa_i(t)(\left|u\right|+\mathbb{W}_2^{\mathbb{R}}(\mu, \delta_0))+\chi_i(\cdot).
\end{align}\par 
Next, we introduce the space
\begin{align*}
\mathcal{L}^2 := \left\{ (\mu_i)_{i \in \mathbb{Z}} : \mu_i \in \mathcal{P}_2(\mathbb{R}) \text{ for each } i \in \mathbb{Z} \text{ and } \sum_{i \in \mathbb{Z}} W_2^{\mathbb{R}}(\mu_i, \delta_0)^2 < \infty \right\},
\end{align*}
which is equipped with the metric
\begin{align*}
	\rho(\mu, \omega) := \left\{ \sum_{i \in \mathbb{Z}} \mathbb{W}_2^{\mathbb{R}}(\mu_i, \omega_i)^2 \right\}^{\frac{1}{2}},
\end{align*}
for $\mu = (\mu_i)_{i \in \mathbb{Z}}, \omega = (\omega_i)_{i \in \mathbb{Z}} \in \mathcal{L}^2$. From Lemma 2.1 in \cite{32}, we can get that $(\mathcal{L}^2, \rho)$ is a complete metric space. Then for $u = (u_i)_{i \in \mathbb{Z}} \in L^2(\Omega, \ell^2)$, we denote $\mathcal{L}_u := (\mathcal{L}_{u_i})_{i \in \mathbb{Z}}$. Clearly, $\mathcal{L}_u \in \mathcal{L}^2$. Moreover, for any  
$u = (u_i)_{i \in \mathbb{Z}}, v = (v_i)_{i \in \mathbb{Z}} \in L^2(\Omega, \ell^2)$, the following property holds  
	\begin{align}
	\label{eq:2.8}
	\rho(\mathcal{L}_u, \mathcal{L}_v)^2 = \sum_{i \in \mathbb{Z}} \mathbb{W}_2^{\mathbb{R}}(\mathcal{L}_{u_i}, \mathcal{L}_{v_i})^2 \leq \sum_{i \in \mathbb{Z}} \mathbb{E} \|u_i - v_i\|^2 = \mathbb{E} \|u - v\|^2.
\end{align} \par
To express the lattice system \eqref{eq:1.1} as a differential equation in $\ell^2$, we define the following linear operators for every $u = (u_i)_{i \in \mathbb{Z}} \in \ell^2$:
\begin{align*}
(Bu)_i := u_{i+1} - u_i, \quad (B^*u)_i := u_{i-1} - u_i
\end{align*}
and
\begin{align*}
(Au)_i := -u_{i-1} + 2u_i - u_{i+1}, \quad \forall i \in \mathbb{Z}.
\end{align*}
Note that 
\begin{align}\label{eq:2.9}
	A = BB^* = B^*B\ \ \text{and}\ \ 
(B^*u, v) = (u, Bv),
\end{align}
for all $u, v \in \ell^2$. Moreover, it can be seen that $A$ is a bounded linear operator from $\ell^2$ to $\ell^2$ with
\begin{align}
		\label{eq:2.10}
\|Bu\|\leq 2 \|u\|,\ \ \text{and}\ \ \|Au\|\leq 4 \|u\|, \quad \forall u \in \ell^2.
\end{align}\par
For every $t \in \mathbb{R}, u = (u_i)_{i \in \mathbb{Z}}\in \ell^2$ and $\mu = (\mu_i)_{i \in \mathbb{Z}} \in \mathcal{L}^2$, we denote
\begin{align*}
   f(t, u, \mu) = (f_i(t, u_i, \mu_i))_{i \in \mathbb{Z}},
\end{align*}
and define the operator $\tilde{\sigma} : \mathbb{R} \times \ell^2 \times \ell^2 \times \mathcal{L}^2 \to L(\ell^2, \ell^2)$ by
\begin{align*}
\tilde{\sigma}(t, u, \mu)w = (\sigma_i(t, u_i, \mu_i)w_i)_{i \in \mathbb{Z}},
\end{align*}
for every $u = (u_i)_{i \in \mathbb{Z}}, w = (w_i)_{i \in \mathbb{Z}} \in \ell^2$ and $\mu = (\mu_i)_{i \in \mathbb{Z}} \in \mathcal{L}^2$. Let $e_k \in \ell^2$ represent the element having $1$ at position $k$ and $0$ in all the other components. Then, the set $\{e_k\}_{k \in \mathbb{Z}}$ forms a complete orthogonal basis of $\ell^2$.\par
We denote by $L_2(\ell^2, \ell^2)$ the space of Hilbert Schmidt operators from $\ell^2$ to $\ell^2$ which is equipped with norm
\begin{align*}
		\|C\|_{L_2(\ell^2, \ell^2)} := \left( \sum_{k \in \mathbb{Z}} \|Ce_k\|^2 \right)^{\frac{1}{2}}, \quad \forall C \in L_2(\ell^2, \ell^2).
\end{align*}
For any $t \in \mathbb{R}, u_1, u_2\in \ell^2$ and $\mu_1, \mu_2 \in \mathcal{L}^2$, by \eqref{eq:2.2} and \eqref{eq:2.6}, we have that
\begin{align}
\label{eq:2.11}
\|f(t, u_1, \mu_1) - f(t, u_2, \mu_2)\|^2 \leq 2\|\psi(t)\|^2 (\|u_1 - u_2\|^2 +\rho(\mu_1, \mu_2)^2),
\end{align}
and
\begin{align}
	\begin{aligned}
	\label{eq:2.12}
			\|\tilde{\sigma}(t, u_1,\mu_1) - \tilde{\sigma}(t, u_2, \mu_2)\|_{L_2(\ell^2, \ell^2)}^2 
			=& \sum_{k \in \mathbb{Z}} \|[\tilde{\sigma}(t, u_1,\mu_1) - \tilde{\sigma}(t, u_2, \mu_2)]e_k\|^2 \\
			\leq &2\|\kappa(t)\|^2 (\|u_1 - u_2\|^2 + \rho(\mu_1, \mu_2)^2).
	\end{aligned}
\end{align}
For all $t \in \mathbb{R}, u, v \in \ell^2$ and $\mu \in \mathcal{L}^2$, using \eqref{eq:2.4} and \eqref{eq:2.7}, we have that
\begin{align}
	\label{eq:2.13}
		\|f(t, u,\mu)\|^2 \leq 3\|\gamma(t)\|^2 (1 + \|u\|^2 + \rho(\mu, \hat{\delta}_0)^2),
\end{align}
and
\begin{align}
		\begin{aligned}
				\label{eq:2.14}
			\|\tilde{\sigma}(t, u, \mu)\|_{L_2(\ell^2, \ell^2)}^2 &= \sum_{k \in \mathbb{Z}} \|\sigma_k(t, u_k,\mu_k)\|^2 \\
			&\leq 4\|\kappa(t)\|^2 (\|u\|^2 +\rho(\mu, \hat{\delta}_0)^2) + 2\|\chi(t)\|^2,
		\end{aligned}
	\end{align}
where $\hat{\delta}_0=(\delta_0)_{i\in\mathbb{Z}}$.\par

Let $\delta>0$ be a constant, and we denote
\begin{align}
	\label{eq:2.15}
	\Phi(t):=\binom{u(t)}{v(t)}\ \ \ \text{with}\ \ \ v(t):=\dot{u}(t)+\delta u(t),
\end{align}
where $u(t)=(u_i(t))_{i\in\mathbb{Z}}$. Then, using the above notation, we rewrite equation \eqref{eq:1.1} as the following first-order system in $\ell^2\times\ell^2$ for $t>\tau$ with $\tau\in\mathbb{R}$:
\begin{equation}\label{eq:2.16}
	\left\{
	\begin{aligned}
		&d\Phi(t)=F(\Phi(t))dt+G(\Phi(t))dW(t)\\		
		&\Phi(\tau ) = \Phi_{\tau}=\binom{u_\tau}{v_\tau},
	\end{aligned}
	\right.
\end{equation}
where $u_\tau=(u_{i,\tau})_{i\in\mathbb{Z}},\ v_\tau=(\dot{u}_{i,\tau}+\delta u_{i,\tau})_{i\in\mathbb{Z}}$,
\begin{align*}
	F(\Phi(t))=\binom{v(t)-\delta u(t)}{-(\lambda+\delta^2-\alpha \delta)u(t)-(\alpha-\delta)v(t)-\nu Au(t)+f(t,u(t),\mathcal{L}_{u(t)})+g(t)},
\end{align*} 
and
\begin{align*}
	G(\Phi(t))=\binom{0}{\tilde{\sigma}(t, u(t),\mathcal{L}_{u(t)})dW(t)}.
\end{align*} 
Besides, $W$ is a cylindrical $Q$-Wiener process with $Q = I$ in $\ell^2$ on a probability space $(\Omega, \mathcal{F}, \mathbb{P})$. Let $\mathcal{N}$ denote the collection of null sets in the completion of $\mathcal{F}$. We choose the normal filtration
\begin{align*}
	\mathcal{F}_t := \sigma\{W(u) : \tau \leq u \leq t\} \lor \mathcal{N}, \quad \forall t \geq \tau.
\end{align*}
\par

For a Banach space $X$ and $\tau\in \mathbb R$, we use $L^2_{\mathcal{F}_{\tau}}(X)$ to denote the space of all $\mathcal{F}_{\tau}$-measurable, $X$-valued random variables
$\phi$ with $\mathbb{E}\|\phi\|^2_{X}<\infty$, where $\mathbb{E}$ means the mathematical expectation.\par
The strong solution of the system \eqref{eq:2.16} is understood in the following sense.
\begin{defn}
	\label{Def:2.6}
	Suppose $\tau\in\mathbb{R}$, and $\Phi_{\tau}=\binom{u_{\tau}}{v_{\tau}}\in L^2_{\mathcal{F}_{\tau}}(\Omega,\ell^2\times\ell^2)$ is $\mathcal{F}_{\tau}$-measurable. Then, a continuous $\ell^2\times\ell^2$-valued stochastic process $\Phi(t)=\binom{u(t)}{v(t)}$ with $t\in [\tau, \infty)$ is called a strong solution of \eqref{eq:2.16} with initial data $\Phi_{\tau}$ if $(\Phi_t)_{t\geq \tau}$ is $\mathcal{F}_{t}$-adapted, $\Phi(t)\in L^2(\Omega,C([\tau,T],\ell^2\times\ell^2)$ for all $T>\tau$ and for each $\tau\leq t\leq T$, $\Phi(t)=\binom{u(t)}{v(t)}$ satisfies
\begin{equation}\label{eq:2.17}
	u(t)=u_{\tau}+\int_{\tau}^{t}(v(r)-\delta u(r))dr,
\end{equation}	
\begin{equation}\label{eq:2.18}
	\begin{aligned}
			v(t)=&v_{\tau}+\int_{\tau}^{t}\left(-(\lambda+\delta^2-\alpha \delta)u(r)-(\alpha-\delta)v(r)-\nu Au(r)+f(r,u(r),\mathcal{L}_{u(r)})+g(r)\right)dr\\&+\int_{\tau}^{t}(\tilde{\sigma}(r, u(r),\mathcal{L}_{u(r)}))dW(r).
	\end{aligned}
\end{equation}		
\end{defn}\par
The weak solution of the system \eqref{eq:2.16} is understood in the following sense.
\begin{defn}
	\label{Def:2.7}
	For any $\mu_\tau \in \mathcal{P}_2(\mathbb{R})$, $(\tilde{\Phi}(t))_{t \ge \tau}, (\tilde{W}(t))_{t \ge \tau}$ is called a weak solution to \eqref{eq:2.16} starting at $\mu_\tau$, if $(\tilde{W}(t))_{t \ge \tau}$ is a cylindrical $Q$-Wiener process with $Q = I$ under a stochastic basis $(\tilde{\Omega}, \tilde{\mathcal{F}}, \{\tilde{\mathcal{F}}_t\}_{t \ge \tau}, \tilde{\mathbb{P}})$ (i.e. a complete probability space with normal filtration), and $(\tilde{\Phi}(t))_{t \ge \tau}$ is a $\mathcal{F}_t$-adapted $\ell^2\times\ell^2$-valued stochastic process with $\mathcal{L}_{\tilde{u}_\tau}|_{\tilde{\mathbb{P}}} = \mu_\tau$, $\tilde{\Phi}=\binom{\tilde{u}}{\tilde{v}} \in L^2(\Omega, C([\tau, T], \ell^2\times\ell^2))$ for any $T > \tau$ and for each $\tau\leq t\leq T$, $\tilde{\Phi}(t)=\binom{\tilde{u}(t)}{\tilde{v}(t)}$ satisfies
	\begin{equation}\label{eq:2.19}
		\tilde{u}(t)=\tilde{u}_{\tau}+\int_{\tau}^{t}(\tilde{v}(r)-\delta \tilde{u}(r))dr,
	\end{equation}	
	\begin{equation}\label{eq:2.20}
		\begin{aligned}
			\tilde{v}(t)=&\tilde{v}_{\tau}+\int_{\tau}^{t}\left(-(\lambda+\delta^2-\alpha \delta)\tilde{u}(r)-(\alpha-\delta)\tilde{v}(r)-\nu A\tilde{u}(r)+f(r,\tilde{u}(r),\mathcal{L}_{\tilde{u}(r)})+g(r)\right)dr\\&+\int_{\tau}^{t}(\tilde{\sigma}(r, \tilde{u}(r),\mathcal{L}_{\tilde{u}(r)}))d\tilde{W}(r).
		\end{aligned}
	\end{equation}		
\end{defn}\par

From now on, we assume that the coefficient $\delta$ in \eqref{eq:2.15} is a small enough positive constant such that
\begin{align}
	\label{eq:2.21}
	\alpha-\delta>0\ \ and\ \ 	\lambda+\delta^2-\alpha \delta>0,
\end{align}
and let
\begin{align}
	\label{eq:2.22}
	\beta=\min\{\delta,\alpha-\delta\},
\end{align}
\begin{align}
	\label{eq:2.23}
	\delta(\lambda+\delta^2-\alpha \delta)-(\dfrac{12}{\alpha-\delta}\|\gamma\|^2_{L^\infty([\tau,T],\ell^2)}+8\|\kappa\|^2_{L^\infty})>0,
\end{align}
and
\begin{align}
	\label{eq:2.24}
	2\delta>3\varepsilon_1+\varepsilon_2[(\lambda+\delta^2-\alpha \delta)^4+256\nu^4+54\|\gamma\|_{L^\infty([\tau,T],\ell^2)}^4]+576\varepsilon_3^{-1}\|\kappa\|^4_{L^\infty}.
\end{align}
\par 
Denote $D_1 = \{D(\tau) : \tau \in \mathbb{R}, D(\tau) \text{ is a bounded nonempty subset of } \mathcal{P}_2(\ell^2\times\ell^2)\}$, and $D_2 = \{D(\tau) : \tau \in \mathbb{R}, D(\tau) \text{ is a bounded nonempty subset of } \mathcal{P}_4(\ell^2\times\ell^2)\}$. Denote by $\mathcal{D}_0$ the collection of all such families $D_1$ which further satisfy:
\begin{align*}
	\lim_{\tau \to -\infty} e^{\beta \tau} \|D_1(\tau)\|_{\mathcal{P}_2(\ell^2\times\ell^2)}^2 = 0,
\end{align*}
where $\beta$ is the same number as in \eqref{eq:2.22}. Similarly, denote by $\mathcal{D}$ the collection of all such families $D_2$ which further satisfy:
\begin{align*}
	\lim_{\tau \to -\infty} e^{2\beta \tau} \|D_2(\tau)\|_{\mathcal{P}_4(\ell^2\times\ell^2)}^4 = 0.
\end{align*}
It is evident $\mathcal{D} \subseteq \mathcal{D}_0$.\par 
Throughout this article, we assume
\begin{align}
	\label{eq:2.25}
	\int_{-\infty}^{\tau}e^{\beta(t-\tau)}\|g(t)\|^2dt<\infty
\end{align}
and
\begin{align}
	\label{eq:2.26}
\int_{-\infty}^{\tau}e^{2\beta(t-\tau)}\|g(t)\|^4dr<\infty,\ \ \forall \tau\in\mathbb{R}.
\end{align}

\section{Existence and uniqueness of solutions}
\indent
\par
In this section, we will establish the well-posedness of solutions (including both strong and weak solutions) to the  distribution-dependent second-order stochastic lattice system \eqref{eq:2.16}.\par
The following proposition give the existence and uniqueness of strong solutions for the second-order stochastic
lattice system \eqref{eq:2.16}.

\begin{prop}\label{pro:3.1}
Suppose $\tau\in\mathbb{R}$, and $\Phi_{\tau}=\binom{u_{\tau}}{v_{\tau}}\in L^2_{\mathcal{F}_{\tau}}(\Omega,\ell^2\times\ell^2)$ is $\mathcal{F}_{\tau}$-measurable. If $(A1)-(A3)$ hold, then
system \eqref{eq:2.16} has a unique stong solution  $\Phi=\binom{u}{v}\in L^2(\Omega,C([\tau,T]),\ell^2\times\ell^2)$ for every $T>\tau$ with the property $\mathcal{L}_{\Phi_{(t)}}\in\mathcal{P}^2(\ell^2\times\ell^2)$ for all $\tau\leq t\leq T$. In addition, in this context,uniqueness means that any other solution $\Psi$ of system \eqref{eq:2.16} is indistinguishable from $\Phi$, i.e.,:
\begin{align*}
\mathbb{P}\left(\left\{\Phi(t)=\Psi(t),\forall t\in [\tau,T]\right\}\right)=1. 
\end{align*}

\end{prop}
\begin{proof}
The proof is carried out in three steps by means of the Picard iteration method.\par
\textbf{Step1:} First,we prove the existence. For all $\tau\leq t\leq T$, let $\Phi^{(0)}(t)=\binom{u^{(0)}(t)}{v^{(0)}(t)}=\binom{u_{\tau}}{v_{\tau}}$, $\mathcal{L}_{u^{(0)}(t)}=\mathcal{L}_{u_{\tau}}$. For any $n\in\mathbb{N}^+$, let $\Phi^{(n)}(t)=\binom{u^{(n)}(t)}{v^{(n)}(t)}$ solve the following integral equation	
\begin{equation}\label{eq:3.1}
	\left\{
	\begin{aligned}
		u^{(n)}(t)=&u_{\tau}+\int_{\tau}^{t}v^{(n-1)}(s)ds-\delta \int_{\tau}^{t}u^{(n-1)}(s)ds,\\		
		v^{(n)}(t)=&v_{\tau}-(\lambda+\delta^2-\alpha \delta)\int_{\tau}^{t}u^{(n-1)}(s)ds-(\alpha-\delta)\int_{\tau}^{t}v^{(n-1)}(s)ds-\nu \int_{\tau}^{t}Au^{(n-1)}(s)ds\\&+\int_{\tau}^{t}f(s,u^{(n-1)}(s),\mathcal{L}_{u^{(n-1)}(s)})ds+\int_{\tau}^{t}g(s)ds+\int_{\tau}^{t}\tilde{\sigma}(s, u^{(n-1)}(s),\mathcal{L}_{u^{(n-1)}(s)})dW(s)
	\end{aligned}
	\right.
\end{equation}
with initial condition $\binom{u^{(n)}_{\tau}}{v^{(n)}_{\tau}}=\binom{u_{\tau}}{v_{\tau}}$, so $\binom{u^{(n)}}{v^{(n)}}\in L^2(\Omega,C([\tau,\tau+T]),\ell^2\times\ell^2)$ foe each $n\in\mathbb{N}$.\par
Next, we will show that for all $n\in\mathbb{N}^+$ and $\tau\leq t\leq T$,
\begin{align}\label{eq:3.2}
	\mathbb{E}\{\sup_{\tau\leq\varrho\leq t}\|\Phi^{(n)}(\varrho)\|^2\}\leq\left[\bigstar_1+\bigstar_2(T-\tau)\mathbb{E}\|\Phi_{\tau}\|^2\right]e^{\bigstar_2T},
\end{align}
where $\bigstar_1$ and $\bigstar_2$ are positive constants independent of $n$ and $t$.\\
By \eqref{eq:3.1}, we have for any $n\in\mathbb{N}^+$ and $\tau\leq t\leq T$,
\begin{equation*}
	\left\{
	\begin{aligned}
		\mathbb{E}(\sup_{\tau\leq\varrho\leq t}\|u^{(n)}(\varrho)\|^2)=&3\|u_{\tau}\|^2+3\mathbb{E}(\sup_{\tau\leq\varrho\leq t}\|\int_{\tau}^{\varrho}v^{(n-1)}(s)ds\|^2)+3\delta^2 \mathbb{E}(\sup_{\tau\leq\varrho\leq t}\|\int_{\tau}^{\rho}u^{(n-1)}(s)ds\|^2),\\		
		\mathbb{E}(\sup_{\tau\leq\varrho\leq t}\|v^{(n)}(\varrho)\|^2)=&7\|v_{\tau}\|^2+7(\lambda+\delta^2-\alpha \delta)^2\mathbb{E}(\sup_{\tau\leq\varrho\leq t}\|\int_{\tau}^{\varrho}u^{(n-1)}(s)ds\|^2)\\&+7(\alpha-\delta)^2\mathbb{E}(\sup_{\tau\leq\varrho\leq t}\|\int_{\tau}^{\varrho}v^{(n-1)}(s)ds\|^2)+7\nu^2\mathbb{E}(\sup_{\tau\leq\varrho\leq t}\| \int_{\tau}^{\varrho}Au^{(n-1)}(s)ds\|^2)\\&+7\mathbb{E}(\sup_{\tau\leq\varrho\leq t}\|\int_{\tau}^{\varrho}f(s,u^{(n-1)}(s),\mathcal{L}_{u^{(n-1)}(s)})ds\|^2\\&+7\mathbb{E}(\sup_{\tau\leq\rho\leq t}\|\int_{\tau}^{\varrho}g(s)ds\|^2+7\mathbb{E}(\sup_{\tau\leq\varrho\leq t}\|\int_{\tau}^{\varrho}\tilde{\sigma}(s, u^{(n-1)}(s),\mathcal{L}_{u^{(n-1)}(s)})dW(s)\|^2.
	\end{aligned}
	\right.
\end{equation*}
By using \eqref{eq:2.8}, \eqref{eq:2.13}, \eqref{eq:2.14} and the BDG inequality, we get that
\begin{align*}
	\mathbb{E}(\sup_{\tau\leq\varrho\leq t}\|\Phi^{(n)}(\varrho)\|^2)\leq&7\|\Phi_{\tau}\|^2+7[(\lambda+\delta^2-\alpha \delta)^2+(\alpha-\delta)^2+\delta^2+1](T-\tau)\int_{\tau}^{t}\mathbb{E}(\sup_{\tau\leq\varrho\leq s}\|\Phi^{(n-1)}(\varrho)\|^2)ds\\&+[(2\nu^2+42\|\gamma\|^2_{L^\infty([\tau,T],\ell^2)})(T-\tau)+56c_1\|\kappa\|^2_{L^\infty(\mathbb{R},\ell^2)}]\int_{\tau}^{t}\mathbb{E}(\sup_{\tau\leq\varrho\leq s}\|u^{(n-1)}(\varrho)\|^2)ds\\&+[21(T-\tau)\|\gamma\|^2_{L^\infty([\tau,T],\ell^2)}+7\int_{\tau}^{T}\|g(s)\|^2ds+14c_1\|\chi\|^2_{L^\infty(\mathbb{R},\ell^2)}](T-\tau)\\\leq&\bigstar_1+\bigstar_2\int_{\tau}^{t}\mathbb{E}(\sup_{\tau\leq\varrho\leq s}\|\Phi^{(n-1)}(\varrho)\|^2)ds,
\end{align*}
where $c_1$ follows from the BDG inequality. Thus we obtain that for any $m\in\mathbb{N}^+$ and $\tau\leq t\leq T$,
\begin{align*}
	\sup_{1\leq n\leq m}[\mathbb{E}(\sup_{\tau\leq\varrho\leq t}\|\Phi^{(n)}(\varrho)\|^2)]\leq\bigstar_1+\bigstar_2\int_{\tau}^{t}\sup_{1\leq n\leq m}[\mathbb{E}(\sup_{\tau\leq\varrho\leq s}\|\Phi^{(n-1)}(\varrho)\|^2)]ds,
\end{align*}
which implies,
\begin{align*}
	\sup_{1\leq n\leq m}[\mathbb{E}(\sup_{\tau\leq\varrho\leq t}\|\Phi^{(n)}(\varrho)\|^2)]&\leq\bigstar_1+\bigstar_2\int_{\tau}^{t}\left(\mathbb{E}\|\Phi_{\tau}\|^2+\sup_{1\leq n\leq m}[\mathbb{E}(\sup_{\tau\leq\varrho\leq s}\|\Phi^{(n)}(\varrho)\|^2)]\right)ds\\&\leq\bigstar_1+\bigstar_2(T-\tau)\mathbb{E}\|\Phi_{\tau}\|^2+\bigstar_2\int_{\tau}^{t}\sup_{1\leq n\leq m}[\mathbb{E}(\sup_{\tau\leq\varrho\leq s}\|\Phi^{(n)}(\varrho)\|^2)]ds.
\end{align*}
Utilizing the Gronwall inequality, we have that for any $m\in\mathbb{N}^+$ and $\tau\leq t\leq T$,
\begin{align*}
	\sup_{1\leq n\leq m}[\mathbb{E}(\sup_{\tau\leq\varrho\leq t}\|\Phi^{(n)}(\varrho)\|^2)]\leq\left[\bigstar_1+\bigstar_2(T-\tau)\mathbb{E}\|\Phi_{\tau}\|^2\right]e^{\bigstar_2T},
\end{align*}
then by the arbitrariness of $m$, \eqref{eq:3.2} holds.\par
To make the notation simpler, we denote $\Xi^{(n)}(t)=\binom{a^{(n)}(t)}{b^{(n)}(t)}:=\Phi^{(n+1)}(t)-\Phi^{(n)}(t)=\binom{u^{(n+1)}(t)-u^{(n)}(t)}{v^{(n+1)}(t)-v^{(n)}(t)}$ for all $n\in\mathbb{N}$ and $t\in[\tau,T]$. It's clear that $\|\Xi^{(n)}(t)\|^2=\|a^{(n)}(t)\|^2+\|b^{(n)}(t)\|^2$. As $n=0$, $\Xi^{(0)}(t):=\Phi^{(1)}(t)-\Phi_{\tau}$. Let

\begin{align*}
	M:=&8(T-\tau)\{[(\lambda+\delta^2-\alpha \delta)^2+(\alpha-\delta)^2+\delta^2+16\nu^2+1]+6\|\gamma\|^2_{L^\infty([\tau,T],\ell^2)}(T-\tau)\\&+8c_1\|\kappa\|^2_{L^\infty(\mathbb{R},\ell^2)}(T-\tau)\}\mathbb{E}\|\Phi_{\tau}\|^2\\&+8(T-\tau)\{3\|\gamma\|^2_{L^\infty([\tau,T],\ell^2)}(T-\tau)+\int_{\tau}^{T}\|g(s)\|^2ds+2c_1\|\chi\|^2_{L^\infty(\mathbb{R},\ell^2)}\}
\end{align*}
and
\begin{align*}
	N:=7(T-\tau)\{[(\lambda+\delta^2-\alpha \delta)^2+(\alpha-\delta)^2+\delta^2+16\nu^2+1]+4\|\psi\|^2_{L^\infty([\tau,T],\ell^2)}\}+28c_1\|\kappa \|^2_{L^\infty(\mathbb{R},\ell^2)}.
\end{align*}
Our initial inductive claim states that for every $ n \in \mathbb{N} $ and $ \tau \leq t \leq T $,
\begin{align}\label{eq:3.3}
 \mathbb{E}\left(\sup_{\tau \leq \varrho \leq t} \|\Xi^{(n)}(\varrho)\|^2\right) \leq \frac{M[N(t - \tau)]^n}{n!}.
\end{align}
We now proceed to verify this claim. In fact, when $ n = 0 $, it follows that for all $ \tau \leq t \leq T $,
\begin{equation*}
	\left\{
	\begin{aligned}
		u^{(1)}(t)=&u_{\tau}+\int_{\tau}^{t}v^{(0)}(s)ds-\delta \int_{\tau}^{t}u^{(0)}(s)ds,\\		
		v^{(1)}(t)=&v_{\tau}-(\lambda+\delta^2-\alpha \delta)\int_{\tau}^{t}u^{(0)}(s)ds-(\alpha-\delta)\int_{\tau}^{t}v^{(0)}(s)ds-\nu \int_{\tau}^{t}Au^{(0)}(s)ds\\&+\int_{\tau}^{t}f(s,u^{(0)}(s),\mathcal{L}_{u^{(0)}(s)})ds+\int_{\tau}^{t}g(s)ds+\int_{\tau}^{t}\tilde{\sigma}(s, u^{(0)}(s),\mathcal{L}_{u^{(0)}(s)})dW(s)
	\end{aligned}
	\right.
\end{equation*}
Along with \eqref{eq:2.10}, we obtain that for all $ \tau \leq t \leq T $,
\begin{align*}
	\mathbb{E}\left(\sup_{\tau \leq \varrho \leq t} \|\Xi^{(n)}(\varrho)\|^2\right) \leq&6(T-\tau)^2[(\lambda+\delta^2-\alpha \delta)^2+(\alpha-\delta)^2+\delta^2+16\nu^2+1]\mathbb{E}\|\Phi_{\tau}\|^2\\&+6(T-\tau)\int_{\tau}^{T}\|g(s)\|^2ds\\&+6(T-\tau)\int_{\tau}^{t}\mathbb{E}\|f(s,u^{(0)}(s),\mathcal{L}_{u^{(0)}(s)})\|^2ds\\&+6\mathbb{E}(\sup_{\tau \leq \varrho \leq t}\int_{\tau}^{\varrho}\tilde{\sigma}(s, u^{(0)}(s),\mathcal{L}_{u^{(0)}(s)})dW(s) ).
\end{align*}
By using \eqref{eq:2.8} and \eqref{eq:2.13},  we find that for all $ \tau \leq t \leq T $,
\begin{align*}
&6(T-\tau)\int_{\tau}^{t}\mathbb{E}\|f(s,u^{(0)}(s),\mathcal{L}_{u^{(0)}(s)})\|^2ds\\\leq&18(T-\tau)\|\gamma\|^2_{L^\infty([\tau,T],\ell^2)}[(T-\tau)+2(T-\tau)\mathbb{E}\|u_{\tau}\|^2].
\end{align*}
By \eqref{eq:2.8}, \eqref{eq:2.14} and the BDG inequality, we have  that for all $ \tau \leq t \leq T $,
\begin{align*}
	&6\mathbb{E}(\sup_{\tau \leq \varrho \leq t}\int_{\tau}^{\varrho}\tilde{\sigma}(s, u^{(n-1)}(s),\mathcal{L}_{u^{(n-1)}(s)})dW(s) )\\\leq&6c_1\mathbb{E}(\int_{\tau}^{t}\|\tilde{\sigma}(s, u^{(0)}(s),\mathcal{L}_{u^{(0)}(s)})\|^2_{L_2(\ell^2,\ell^2)})\\\leq&48c_1(T-\tau)\|\kappa \|^2_{L^\infty(\mathbb{R},\ell^2)}\mathbb{E}\|u_{\tau}\|^2+12c_1(T-\tau)\|\chi\|^2_{L^\infty(\mathbb{R},\ell^2)},
\end{align*}
where $c_1$ follows from the BDG inequality. Given the preceding reasoning, \eqref{eq:3.3} is valid for $n=0$. Next, suppose that \eqref{eq:3.3}  holds for $k=n-1$. In order to establish the statement for $k=n$, we begin by noting that for all $ \tau \leq t \leq T $ that
\begin{equation*}
	\left\{
	\begin{aligned}
		a^{(n)}(t)=&\int_{\tau}^{t}b^{(n-1)}(s)ds-\delta \int_{\tau}^{t}a^{(n-1)}(s)ds,\\		
		b^{(n)}(t)=&-(\lambda+\delta^2-\alpha \delta)\int_{\tau}^{t}a^{(n-1)}(s)ds-(\alpha-\delta)\int_{\tau}^{t}b^{(n-1)}(s)ds-\nu \int_{\tau}^{t}Aa^{(n-1)}(s)ds\\&+\int_{\tau}^{t}[f(s,u^{(n)}(s),\mathcal{L}_{u^{(n)}(s)})-f(s,u^{(n-1)}(s),\mathcal{L}_{u^{(n-1)}(s)})]ds\\&+\int_{\tau}^{t}[\tilde{\sigma}(s, u^{(n)}(s),\mathcal{L}_{u^{(n)}(s)})-\tilde{\sigma}(s, u^{(n-1)}(s),\mathcal{L}_{u^{(n-1)}(s)})]dW(s).
	\end{aligned}
	\right.
\end{equation*}
Then, we have  that for all $ \tau \leq t \leq T $,
\begin{equation}\label{eq:3.4}
	\begin{aligned}
		\mathbb{E}(\sup_{\tau \leq \varrho \leq t}\|\Xi^{(n)}(\varrho)\|^2)\leq&5(T-\tau)[(\lambda+\delta^2-\alpha \delta)^2+(\alpha-\delta)^2+\delta^2+16\nu^2+1]\int_{\tau}^{t}\mathbb{E}(\sup_{\tau \leq \varrho \leq t}\|\Xi^{(n-1)}(\varrho)\|^2)ds\\&+5(T-\tau)\int_{\tau}^{t}\mathbb{E}\|f(s,u^{(n)}(s),\mathcal{L}_{u^{(n)}(s)})-f(s,u^{(n-1)}(s),\mathcal{L}_{u^{(n-1)}(s)})\|^2ds\\&+5\mathbb{E}(\sup_{\tau \leq \varrho \leq t}\|\int_{\tau}^{\varrho}[\tilde{\sigma}(s, u^{(n)}(s),\mathcal{L}_{u^{(n)}(s)})-\tilde{\sigma}(s, u^{(n-1)}(s),\mathcal{L}_{u^{(n-1)}(s)})]dW(s)\|^2).
	\end{aligned}
\end{equation}
For the second term in the right-hand side of \eqref{eq:3.4}, by using \eqref{eq:2.8} and  \eqref{eq:2.11},  we find that for all $ \tau \leq t \leq T $,
\begin{equation}\label{eq:3.5}
	\begin{aligned}
	&5(T-\tau)\int_{\tau}^{t}\mathbb{E}\|f(s,u^{(n)}(s),\mathcal{L}_{u^{(n)}(s)})-f(s,u^{(n-1)}(s),\mathcal{L}_{u^{(n-1)}(s)})\|^2ds\\\leq&	20(T-\tau)\|\psi\|^2_{L^\infty([\tau,T],\ell^2)}\int_{\tau}^{t}\mathbb{E}(\sup_{\tau \leq \varrho \leq s}\|a^{(n-1)}(\varrho)\|^2)ds.
	\end{aligned}
\end{equation}
For the last term in the right-hand side of \eqref{eq:3.4}, by using the BDG inequality, \eqref{eq:2.8} and  \eqref{eq:2.12},  we find that for all $ \tau \leq t \leq T $,
\begin{equation}\label{eq:3.6}
	\begin{aligned}
		&5\mathbb{E}(\sup_{\tau \leq \varrho \leq t}\|\int_{\tau}^{\varrho}[\tilde{\sigma}(s, u^{(n)}(s),\mathcal{L}_{u^{(n)}(s)})-\tilde{\sigma}(s, u^{(n-1)}(s),\mathcal{L}_{u^{(n-1)}(s)})]dW(s)\|^2)\\\leq&5c_1\mathbb{E}(\int_{\tau}^{t}\|\tilde{\sigma}(s, u^{(n)}(s),\mathcal{L}_{u^{(n)}(s)})-\tilde{\sigma}(s, u^{(n-1)}(s),\mathcal{L}_{u^{(n-1)}(s)})\|_{L_2(\ell^2,\ell^2)}^2ds)\\\leq&20c_1\|\kappa \|^2_{L^\infty(\mathbb{R},\ell^2)}\int_{\tau}^{t}\mathbb{E}(\sup_{\tau \leq \varrho \leq s}\|a^{(n-1)}(\varrho)\|^2)ds.
	\end{aligned}
\end{equation}
Considering \eqref{eq:3.4}-\eqref{eq:3.6}, from the induction hypothesis, it follows that for all $ t \in [\tau, T] $,
\begin{align*}
	\mathbb{E}(\sup_{\tau \leq \varrho
		 \leq t} \|\Xi^{(n)}(\varrho)\|^2) \leq N\int_{\tau}^{t}\mathbb{E}(\sup_{\tau \leq \varrho
		 \leq t} \|\Xi^{(n-1)}(\varrho)\|^2)\leq N \int_\tau^t \frac{M[N(s-\tau)]^{n-1}}{(n-1)!} ds = \frac{M[N(t - \tau)]^n}{n!}.
\end{align*}
Hence, statement \eqref{eq:3.3} holds for every $ n \in \mathbb{N} $.\par
Applying Chebyshev's inequality yields
\begin{align*}
	\mathbb{P}(\sup_{\tau \leq \varrho \leq T} \|\Xi^{(n)}(\varrho)\| > \frac{1}{2^n}) \leq \frac{M(4NT)^n}{n!}.
\end{align*}
Due to the convergence of the positive series $	\sum_{n=0}^\infty \frac{M(4NT)^n}{n!}$
along with the Borel-Cantelli lemma, for almost every $ \omega \in \Omega $ one can find a positive integer $ N = N(\omega) $ such that for every $ n > N $,
\begin{align*} 
	\sup_{\tau \leq \varrho \leq T} \|\Xi^{(n)}(\varrho)\| \leq \frac{1}{2^n},
\end{align*}
which implies that the partial sum sequence
\begin{align*} 
	\Phi^{(n)}(t)=\Phi^{(0)} + \sum_{k=0}^{n-1}\Xi^{(k)}(t) 
\end{align*}
converges uniformly for $ t \in [\tau, T] $ for almost every $ \omega \in \Omega $.  
Let the limit be denoted by $ \Phi(t)=\binom{u(t)}{v(t)} $ for $ t \in [\tau, T] $.
Evidently, $ \Phi $ is $ \mathcal{F}_t $-adapted when $ t \geq \tau $.  
Additionally, using \eqref{eq:3.2} and the dominated convergence theorem, we obtain
\begin{align} \label{eq:3.7}
	\mathbb{E} \left( \sup_{\tau \leq \varrho \leq T} \|\Phi^{(n)}(\varrho) - \Phi(\varrho)\|^2 \right) \to 0, \quad \text{as } n \to \infty.
\end{align}\par
To confirm that $ \Phi(t) $ satisfies equation \eqref{eq:2.16}, by virtue of \eqref{eq:2.10} and \eqref{eq:3.7}, we derive for each $ t \in [\tau, T] $ that
\begin{equation}\label{eq:3.8}
	\begin{aligned}
		\mathbb{E}(&\|\int_{\tau}^{t}(v^{(n-1)}(s)-v(s))ds-\delta\int_{\tau}^{t}(u^{(n-1)}(s)-u(s))ds-(\lambda+\delta^2-\alpha \delta)\int_{\tau}^{t}(u^{(n-1)}(s)-u(s))d\\&-s(\alpha-\delta)\int_{\tau}^{t}(v^{(n-1)}(s)-v(s))ds-\nu\int_{\tau}^{t}A(u^{(n-1)}(s)-u(s))ds\|^2)\\\leq&5(T-\tau)[(\lambda+\delta^2-\alpha \delta)^2+(\alpha-\delta)^2+\delta^2+16\nu^2+1]\int_{\tau}^{t}\mathbb{E}(\sup_{\tau \leq \varrho \leq s}\|\Phi^{(n-1)}(\varrho) - \Phi(\varrho)\|^2)ds\\\to& 0, \quad \text{as } n \to \infty.
	\end{aligned}
\end{equation}
By using \eqref{eq:2.8}, \eqref{eq:2.11}, \eqref{eq:2.12} and \eqref{eq:3.7}, we find for each $ t \in [\tau, T] $,
\begin{equation}\label{eq:3.9}
	\begin{aligned}
		&\mathbb{E}\|\int_{\tau}^{t}[f(s,u^{(n-1)}(s),\mathcal{L}_{u^{(n-1)}(s)})-f(s,u(s),\mathcal{L}_{u(s)})]ds\|^2\\&+\mathbb{E}\|\int_{\tau}^{t}[\tilde{\sigma}(s, u^{(n-1)}(s),\mathcal{L}_{u^{(n-1)}(s)})-\tilde{\sigma}(s, u(s),\mathcal{L}_{u(s)})]dW(s)\|^2\\\leq&\left[2(T-\tau)\|\psi\|^2_{L^\infty([\tau,T],\ell^2)}+4c_1\|\kappa \|^2_{L^\infty(\mathbb{R},\ell^2)}\right]\int_{\tau}^{t}\mathbb{E}(\sup_{\tau \leq \varrho \leq s}\|\Phi^{(n-1)}(\varrho) - \Phi(\varrho)\|^2)ds\\\to& 0, \quad \text{as } n \to \infty.
	\end{aligned}
\end{equation}
Together with \eqref{eq:3.8} and \eqref{eq:3.9}, it follows that $ \Phi $ satisfies \eqref{eq:2.17}-\eqref{eq:2.18} by taking the limit $ n \to \infty $ in \eqref{eq:3.1}. Furthermore, by virtue of \eqref{eq:3.7}, one can choose $ N_0 > 0 $ satisfying
\begin{align*}
	\mathbb{E} \left( \sup_{\tau \leq \varrho \leq T} \|\Phi^{(N_0)}(\varrho) - \Phi(\varrho)\|^2 \right) < \frac{1}{2}.
\end{align*}
Applying \eqref{eq:3.2} once more yields
\begin{equation}\label{eq:3.10}
\begin{aligned}
	\mathbb{E} (\sup_{\tau \leq \varrho \leq T} \|\Phi(\varrho)\|^2 )&\leq 2\mathbb{E} (\sup_{\tau \leq \varrho \leq T} \|\Phi^{(N_0)}(\varrho) - \Phi(\varrho)\|^2) + 2\mathbb{E}( \sup_{\tau \leq \varrho \leq T} \|\Phi^{(N_0)}(\varrho)\|^2) \\
	&\leq 1 + 2(\bigstar_1 + \bigstar_2(T - \tau)\mathbb{E}\|\Phi_{\tau}\|^2)e^{\bigstar_2T}. 
\end{aligned}
\end{equation}
Hence $\Phi\in L^2(\Omega,C([\tau,T],\ell^2\times\ell^2))$.\par
\textbf{Step2:} Next, we prove the uniqueness. Let $\Phi(t)=\binom{u(t)}{v(t)}$ and $\Phi_1(t)=\binom{u_1(t)}{v_1(t)}$ be two strong solutions of \eqref{eq:2.16} in $\Phi\in L^2(\Omega,C([\tau,T],\ell^2\times\ell^2))$. For any $ \tau \leq t \leq T $, notice that
\begin{equation*}
	\left\{
	\begin{aligned}
		u(t)-u_1(t)=&\int_{\tau}^{t}(v(s)-v_1(s))ds-\delta \int_{\tau}^{t}(u(s)-u_1(s))ds,\\		
		v(t)-v_1(t)=&-(\lambda+\delta^2-\alpha \delta)\int_{\tau}^{t}(u(s)-u_1(s))ds-(\alpha-\delta)\int_{\tau}^{t}(v(s)-v_1(s))ds\\&-\nu \int_{\tau}^{t}A(u(s)-u_1(s))ds+\int_{\tau}^{t}[f(s,u(s),\mathcal{L}_{u(s)})-f(s,u_1(s),\mathcal{L}_{u_1(s)})]ds\\&+\int_{\tau}^{t}[\tilde{\sigma}(s, u(s),\mathcal{L}_{u(s)})-\tilde{\sigma}(s, u_1(s),\mathcal{L}_{u_1(s)})]dW(s).
	\end{aligned}
	\right.
\end{equation*}
Applying the $It\hat{o}$ formula, we get that for all $ \tau \leq t \leq T $,
\begin{equation*}
	\begin{aligned}
		\|u(t)-u_1(t)\|^2=2\int_{\tau}^{t}(u(s)-u_1(s),v(s)-v_1(s))ds-2\delta \int_{\tau}^{t}\|u(s)-u_1(s)\|^2ds,\\		
	\end{aligned}
\end{equation*}
and
\begin{equation*}
	\begin{aligned}		
		\|v(t)-v_1(t)\|^2=&-2(\lambda+\delta^2-\alpha \delta)\int_{\tau}^{t}(u(s)-u_1(s),v(s)-v_1(s))ds-2(\alpha-\delta)\int_{\tau}^{t}\|v(s)-v_1(s)\|^2ds\\&-2\nu \int_{\tau}^{t}(A(u(s)-u_1(s)),v(s)-v_1(s))ds\\&+2\int_{\tau}^{t}(v(s)-v_1(s),f(s,u(s),\mathcal{L}_{u(s)})-f(s,u_1(s),\mathcal{L}_{u_1(s)}))ds\\&+2\int_{\tau}^{t}(v(s)-v_1(s),[\tilde{\sigma}(s, u(s),\mathcal{L}_{u(s)})-\tilde{\sigma}(s, u_1(s),\mathcal{L}_{u_1(s)})]dW(s))\\&+\int_{\tau}^{t}\|\tilde{\sigma}(s, u(s),\mathcal{L}_{u(s)})-\tilde{\sigma}(s, u_1(s),\mathcal{L}_{u_1(s)})\|^2_{L_2(\ell^2,\ell^2)}ds.
	\end{aligned}
\end{equation*}
Therefore, we have
\begin{equation}\label{eq:3.11}
	\begin{aligned}
		&\mathbb{E}\{\sup_{\tau \leq \varrho \leq t}[(\lambda+\delta^2-\alpha \delta)\|u(\varrho)-u_1(\varrho)\|^2+\|v(\varrho)-v_1(\varrho)\|^2]\}\\\leq&2\nu\left|\int_{\tau}^{t}\mathbb{E}(A(u(s)-u_1(s)),v(s)-v_1(s))ds\right|\\&+2\int_{\tau}^{t}\mathbb{E}(v(s)-v_1(s),f(s,u(s),\mathcal{L}_{u(s)})-f(s,u_1(s),\mathcal{L}_{u_1(s)}))ds\\&+\int_{\tau}^{t}\mathbb{E}\|\tilde{\sigma}(s, u(s),\mathcal{L}_{u(s)})-\tilde{\sigma}(s, u_1(s),\mathcal{L}_{u_1(s)})\|^2_{L_2(\ell^2,\ell^2)}ds\\&+2\mathbb{E}\{\sup_{\tau \leq \varrho \leq t}\|\int_{\tau}^{\varrho}(v(s)-v_1(s),[\tilde{\sigma}(s, u(s),\mathcal{L}_{u(s)})-\tilde{\sigma}(s, u_1(s),\mathcal{L}_{u_1(s)})]dW(s))\|\}.
	\end{aligned}
\end{equation}
We estimate each term in the right-hand side of \eqref{eq:3.11} individually. For the first term, employing \eqref{eq:2.9} and \eqref{eq:2.10} yields that for all $ \tau \leq t \leq T $,
\begin{equation}\label{eq:3.12}
	\begin{aligned}
		&2\nu\left|\int_{\tau}^{t}\mathbb{E}(A(u(s)-u_1(s)),v(s)-v_1(s))ds\right|\\\leq&2\nu\int_{\tau}^{t}\mathbb{E}\|B(u(s)-u_1(s))\|\|B(v(s)-v_1(s))\|ds\\\leq&\nu\int_{\tau}^{t}\mathbb{E}(\|B(u(s)-u_1(s))\|^2+\|B(v(s)-v_1(s))\|^2)ds\\\leq&4\nu\int_{\tau}^{t}(\sup_{\tau \leq \varrho \leq s}\|\Phi(\varrho)-\Phi_1(\varrho)\|^2)ds.
	\end{aligned}
\end{equation}
For the second term, we infer from \eqref{eq:2.11} and \eqref{eq:2.8} that for all $ \tau \leq t \leq T $,
\begin{equation}\label{eq:3.13}
	\begin{aligned}
		&2\int_{\tau}^{t}\mathbb{E}(v(s)-v_1(s),f(s,u(s),\mathcal{L}_{u(s)})-f(s,u_1(s),\mathcal{L}_{u_1(s)}))ds\\\leq&\int_{\tau}^{t}\mathbb{E}\|v(s)-v_1(s)\|^2ds+\int_{\tau}^{t}\mathbb{E}\|f(s,u(s),\mathcal{L}_{u(s)})-f(s,u_1(s),\mathcal{L}_{u_1(s)})\|^2ds\\\leq&\int_{\tau}^{t}\mathbb{E}\|v(s)-v_1(s)\|^2ds+4\|\psi\|^2_{L^\infty([\tau,T],\ell^2)}\int_{\tau}^{t}\mathbb{E}\|u(s)-u_1(s)\|^2ds\\\leq&(1+4\|\psi\|^2_{L^\infty([\tau,T],\ell^2)})\int_{\tau}^{t}(\sup_{\tau \leq \varrho \leq s}\|\Phi(\varrho)-\Phi_1(\varrho)\|^2)ds.
	\end{aligned}
\end{equation}
For the third term, by \eqref{eq:2.12} and \eqref{eq:2.8}, we get that for all $ \tau \leq t \leq T $,
\begin{equation}\label{eq:3.14}
	\begin{aligned}
		&\int_{\tau}^{t}\mathbb{E}\|\tilde{\sigma}(s, u(s),\mathcal{L}_{u(s)})-\tilde{\sigma}(s, u_1(s),\mathcal{L}_{u_1(s)})\|^2_{L_2(\ell^2,\ell^2)}ds\\\leq&\int_{\tau}^{t}2\|\kappa(s)\|^2\mathbb{E}(\|u(s)-u_1(s)\|^2+\mathbb{E}\|u(s)-u_1(s)\|^2)ds\\\leq&4\|\kappa\|^2_{L^\infty(\mathbb{R},\ell^2)}\int_{\tau}^{t}(\sup_{\tau \leq \varrho \leq s}\|\Phi(\varrho)-\Phi_1(\varrho)\|^2)ds.
	\end{aligned}
\end{equation}
For the last term, by using the BDG inequality and \eqref{eq:3.13}, we obtain that for all $ \tau \leq t \leq T $,
\begin{equation}\label{eq:3.15}
	\begin{aligned}
		&2\mathbb{E}\{\sup_{\tau \leq \varrho \leq t}\|\int_{\tau}^{\varrho}(v(s)-v_1(s),[\tilde{\sigma}(s, u(s),\mathcal{L}_{u(s)})-\tilde{\sigma}(s, u_1(s),\mathcal{L}_{u_1(s)})]dW(s))\|\}
		\\\leq&2c_1\mathbb{E}\{\int_{\tau}^{t}\|(v(s)-v_1(s),\tilde{\sigma}(s, u(s),\mathcal{L}_{u(s)})-\tilde{\sigma}(s, u_1(s),\mathcal{L}_{u_1(s)}))\|^2ds\}^{\frac{1}{2}}
		\\\leq&2c_1\mathbb{E}\{\int_{\tau}^{t}\|v(s)-v_1(s)\|^2\|\tilde{\sigma}(s, u(s),\mathcal{L}_{u(s)})-\tilde{\sigma}(s, u_1(s),\mathcal{L}_{u_1(s)})\|^2_{L_2(\ell^2,\ell^2)}ds\}^{\frac{1}{2}}
		\\\leq&2c_1\mathbb{E}(\sup_{\tau \leq \varrho \leq t}\|v(\varrho)-v_1(\varrho)\|^2)^{\frac{1}{2}}(\int_{\tau}^{t}\|\tilde{\sigma}(s, u(s),\mathcal{L}_{u(s)})-\tilde{\sigma}(s, u_1(s),\mathcal{L}_{u_1(s)})\|^2_{L_2(\ell^2,\ell^2)}ds)^{\frac{1}{2}}
		\\\leq&\frac{1}{2}\mathbb{E}(\sup_{\tau \leq \varrho \leq t}\|v(\varrho)-v_1(\varrho)\|^2)+2c_1^2\mathbb{E}\|\tilde{\sigma}(s, u(s),\mathcal{L}_{u(s)})-\tilde{\sigma}(s, u_1(s),\mathcal{L}_{u_1(s)})\|^2_{L_2(\ell^2,\ell^2)}ds
		\\\leq&\frac{1}{2}\mathbb{E}(\sup_{\tau \leq \varrho \leq t}\|v(\varrho)-v_1(\varrho)\|^2)+8c_1^2\|\kappa\|^2_{L^\infty(\mathbb{R},\ell^2)}\int_{\tau}^{t}(\sup_{\tau \leq \varrho \leq s}\|\Phi(\varrho)-\Phi_1(\varrho)\|^2)ds.
	\end{aligned}
\end{equation}
Combining \eqref{eq:3.11}-\eqref{eq:3.15}, we have that for all $ \tau \leq t \leq T $,
\begin{equation}\label{eq:3.16}
	\begin{aligned}
		\mathbb{E}(\sup_{\tau \leq \varrho \leq t}\|\Phi(\varrho)-\Phi_1(\varrho)\|^2)\leq[(&\lambda+\delta^2-\alpha \delta)\wedge \frac{1}{2}]^{-1}[1+4\nu+4\|\psi\|^2_{L^\infty([\tau,T],\ell^2)}\\&+4(1+2c_1^2)\|\kappa\|^2_{L^\infty(\mathbb{R},\ell^2)}]\int_{\tau}^{t}(\sup_{\tau \leq \varrho \leq s}\|\Phi(\varrho)-\Phi_1(\varrho)\|^2)ds.
	\end{aligned}
\end{equation}
Employing the Gronwall inequality, we see
\begin{equation}\label{eq:3.17}
	\begin{aligned}
		\mathbb{E}(\sup_{\tau \leq \varrho \leq T}\|\Phi(\varrho)-\Phi_1(\varrho)\|^2)=0
	\end{aligned}
\end{equation}
which implies the uniqueness.\par
\textbf{Step3:} Finally, we prove that $\mathcal{L}_{\Phi_{(t)}}\in\mathcal{P}^2(\ell^2\times\ell^2)$ for any $\tau\leq t\leq T$. Applying the $It\hat{o}$ formula and using \eqref{eq:2.17}-\eqref{eq:2.18}, we get that
\begin{equation*}
	\begin{aligned}
		\|u(t)\|^2=\|u_{\tau}\|^2+2\int_{\tau}^{t}(u(s),v(s))ds-2\delta \int_{\tau}^{t}\|u(s)\|^2ds,\\		
	\end{aligned}
\end{equation*}
and
\begin{equation*}
	\begin{aligned}		
		\|v(t)\|^2=&\|v_{\tau}\|^2-2(\lambda+\delta^2-\alpha \delta)\int_{\tau}^{t}(u(s),v(s))ds-2(\alpha-\delta)\int_{\tau}^{t}\|v(s)\|^2ds-2\nu \int_{\tau}^{t}(Au(s),v(s))ds\\&+2\int_{\tau}^{t}(v(s),g(s))ds+2\int_{\tau}^{t}(v(s),f(s,u(s),\mathcal{L}_{u(s)}))ds\\&+2\int_{\tau}^{t}(v(s),\tilde{\sigma}(s, u(s),\mathcal{L}_{u(s)})dW(s))+\int_{\tau}^{t}\|\tilde{\sigma}(s, u(s),\mathcal{L}_{u(s)})\|^2_{L_2(\ell^2,\ell^2)}ds.
	\end{aligned}
\end{equation*}
which implies that for any $\tau\leq t\leq T$,
\begin{equation}\label{eq:3.18}
	\begin{aligned}
		&\mathbb{E}(\sup_{\tau \leq \varrho \leq t}[(\lambda+\delta^2-\alpha \delta)\|u(\varrho)\|^2+\|v(\varrho)\|^2])\\\leq&\mathbb{E}[(\lambda+\delta^2-\alpha \delta)\|u_{\tau}\|^2+\|v_{\tau}\|^2]+2\nu\left|\int_{\tau}^{t}\mathbb{E}(Au(s),v(s))ds\right|\\&+2\int_{\tau}^{t}\mathbb{E}(v(s),g(s))ds+2\int_{\tau}^{t}\mathbb{E}(v(s),f(s,u(s),\mathcal{L}_{u(s)}))ds\\&+\int_{\tau}^{t}\mathbb{E}\|\tilde{\sigma}(s, u(s),\mathcal{L}_{u(s)})\|^2_{L_2(\ell^2,\ell^2)}ds+2\mathbb{E}(\sup_{\tau \leq \varrho \leq t}\|\int_{\tau}^{\varrho}(v(s),\tilde{\sigma}(s, u(s),\mathcal{L}_{u(s)})dW(s))\|).
	\end{aligned}
\end{equation}
 For the second term on the right-hand side of \eqref{eq:3.18}, by using \eqref{eq:2.9} and \eqref{eq:2.10} we have that for all $ \tau \leq t \leq T $,
\begin{equation}\label{eq:3.19}
	\begin{aligned}
		2\nu\left|\int_{\tau}^{t}\mathbb{E}(Au(s),v(s))ds\right|\leq&\nu\int_{\tau}^{t}\mathbb{E}(\|Bu(s)\|^2+\|Bv(s)\|^2)ds\\\leq&4\nu\int_{\tau}^{t}(\sup_{\tau \leq \varrho \leq s}\|\Phi(\varrho)\|^2)ds.
	\end{aligned}
\end{equation}
For the third term on the right-hand side of \eqref{eq:3.18}, by using Holder inequality and Young inequality, we find that for all $ \tau \leq t \leq T $,
\begin{equation}\label{eq:3.20}
	\begin{aligned}
		2\int_{\tau}^{t}\mathbb{E}(v(s),g(s))ds\leq&\int_{\tau}^{t}\mathbb{E}\|v(s)\|^2ds+\int_{\tau}^{t}\mathbb{E}\|g(s)\|^2\\\leq&\int_{\tau}^{t}(\sup_{\tau \leq \varrho \leq s}\|\Phi(\varrho)\|^2)ds+\int_{\tau}^{T}\mathbb{E}\|g(s)\|^2.
	\end{aligned}
\end{equation}
For the fourth term on the right-hand side of \eqref{eq:3.18}, by using \eqref{eq:2.13} and \eqref{eq:2.8}, we get that for all $ \tau \leq t \leq T $,
\begin{equation}\label{eq:3.21}
	\begin{aligned}
		2\int_{\tau}^{t}\mathbb{E}(v(s),f(s,u(s),\mathcal{L}_{u(s)}))ds\leq&\int_{\tau}^{t}\mathbb{E}\|v(s)\|^2ds+\int_{\tau}^{t}\mathbb{E}\|f(s,u(s),\mathcal{L}_{u(s)})\|^2ds\\\leq&\int_{\tau}^{t}\mathbb{E}\|v(s)\|^2ds+\int_{\tau}^{t}\mathbb{E}(3\|\gamma(s)\|^2(1+\|u(s)\|^2+\mathbb{E}\|u(s)\|^2))ds+\\\leq&(1+6\|\gamma\|^2_{L^\infty([\tau,T],\ell^2)})\int_{\tau}^{t}(\sup_{\tau \leq \varrho \leq s}\|\Phi(\varrho)\|^2)ds+3(T-\tau)\|\gamma\|^2_{L^\infty([\tau,T],\ell^2)}.
	\end{aligned}
\end{equation}
For the fifth term on the right-hand side of \eqref{eq:3.18}, by using \eqref{eq:2.14} and \eqref{eq:2.8}, we know that for all $ \tau \leq t \leq T $,
\begin{equation}\label{eq:3.22}
	\begin{aligned}
		\int_{\tau}^{t}\mathbb{E}\|\tilde{\sigma}(s, u(s),\mathcal{L}_{u(s)})\|^2_{L_2(\ell^2,\ell^2)}ds\leq&8\|\kappa\|^2_{L^\infty(\mathbb{R},\ell^2)}\int_{\tau}^{t}(\sup_{\tau \leq \varrho \leq s}\|\Phi(\varrho)\|^2)ds+2(T-\tau)\|\chi\|^2_{L^\infty(\mathbb{R},\ell^2)}.
	\end{aligned}
\end{equation}
For the fourth term on the right-hand side of \eqref{eq:3.18}, by using the BDG inequality and \eqref{eq:3.22}, we get that for all $ \tau \leq t \leq T $,
\begin{equation}\label{eq:3.23}
	\begin{aligned}
		&2\mathbb{E}(\sup_{\tau \leq \varrho \leq t}\|\int_{\tau}^{\varrho}(v(s),\tilde{\sigma}(s, u(s),\mathcal{L}_{u(s)})dW(s))\|)
		\\\leq&2c_1\mathbb{E}\{\int_{\tau}^{t}\|(v(s),\tilde{\sigma}(s, u(s),\mathcal{L}_{u(s)}))\|^2ds\}^{\frac{1}{2}}
		\\\leq&2c_1\mathbb{E}\{\int_{\tau}^{t}\|v(s)\|^2\|\tilde{\sigma}(s, u(s),\mathcal{L}_{u(s)})\|^2_{L_2(\ell^2,\ell^2)}ds\}^{\frac{1}{2}}
		\\\leq&2c_1\mathbb{E}(\sup_{\tau \leq \varrho \leq t}\|v(\varrho)\|^2)^{\frac{1}{2}}(\int_{\tau}^{t}\|\tilde{\sigma}(s, u(s),\mathcal{L}_{u(s)})\|^2_{L_2(\ell^2,\ell^2)}ds)^{\frac{1}{2}}
		\\\leq&\frac{1}{2}\mathbb{E}(\sup_{\tau \leq \varrho \leq t}\|v(\varrho)\|^2)+2c_1^2\mathbb{E}\|\tilde{\sigma}(s, u(s),\mathcal{L}_{u(s)})\|^2_{L_2(\ell^2,\ell^2)}ds
		\\\leq&\frac{1}{2}\mathbb{E}(\sup_{\tau \leq \varrho \leq t}\|v(\varrho)\|^2)+16c_1^2\|\kappa\|^2_{L^\infty(\mathbb{R},\ell^2)}\int_{\tau}^{t}(\sup_{\tau \leq \varrho \leq s}\|\Phi(\varrho)\|^2)ds+4c_1(T-\tau)\|\chi\|^2_{L^\infty(\mathbb{R},\ell^2)}.
	\end{aligned}
\end{equation}
Together with \eqref{eq:3.18}-\eqref{eq:3.23}, we obtain that for all $ \tau \leq t \leq T $,
\begin{equation*}
	\begin{aligned}
		\mathbb{E}(\sup_{\tau \leq \varrho \leq t}\|\Phi(\varrho)\|^2)\leq&2\mathbb{E}\|\Phi_{\tau}\|^2\\&+(\lambda+\delta^2-\alpha \delta)^{-1}\{\int_{\tau}^{T}\mathbb{E}\|g(s)\|^2+[(T-\tau)[3\|\gamma\|^2_{L^\infty([\tau,T],\ell^2)}+2(1+2c_1^2)\|\chi\|^2_{L^\infty(\mathbb{R},\ell^2)}]]\}\\&+2[1+2\nu+3\|\gamma\|^2_{L^\infty([\tau,T],\ell^2)}+8(1+2c_1^2)\|\kappa\|^2_{L^\infty(\mathbb{R},\ell^2)}]\int_{\tau}^{t}(\sup_{\tau \leq \varrho \leq s}\|\Phi(\varrho)\|^2)ds\\:=&\triangle_1+\triangle_2\int_{\tau}^{t}(\sup_{\tau \leq \varrho \leq s}\|\Phi(\varrho)\|^2)ds,
	\end{aligned}
\end{equation*}
obviously, $\triangle_1$ and $\triangle_2$ are positive constants independent of t.\\
Applying the Gronwall inequality, we have that we obtain that for all $ \tau \leq t \leq T $,
\begin{equation}\label{eq:3.24}
	\begin{aligned}
		\mathbb{E}\|\Phi(t)\|^2)\leq\triangle_1e^{2\triangle_2(T-\tau)}<\infty.
	\end{aligned}
\end{equation}
Hence, $\mathcal{L}_{\Phi_{(t)}}\in\mathcal{P}^2(\ell^2\times\ell^2)$ for all $\tau\leq t\leq T$. This completes the proof.	
\end{proof}
\textbf{Remark 3.1.} Moreover, by analogous arguments as Step 1 in Prop.3.1, if $ \Phi_{\tau} \in L^4(\Omega, \ell^2\times\ell^2) $ is an $ \mathcal{F}_r $-measurable random variable and conditions (A1)-(A3) hold, then system \eqref{eq:2.16} admits a strong solution $ \Phi \in L^4(\Omega, C([\tau, T], \ell^2\times\ell^2)) $ for any $ T > \tau$.\par
Furthermore, we can use Prop.3.2 in \cite{32} to obtain the existence and uniqueness of weak solutions for system \eqref{eq:2.16}.

\section{Uniform Estimates}	
\indent
\par
In this section, we derive the $\mathcal{D}$-uniform moment estimates of solutions of \eqref{eq:2.16}, which play a pivotal role in substantiating the existence and convergence of pullback measure attractors.\par 
We denote by $\Phi(t,\tau-t,\Phi_{\tau-t})=\binom{u(t,\tau-t,\Phi_{\tau-t})}{v(t,\tau-t,\Phi_{\tau-t})}$ the solution of \eqref{eq:2.16} with initial condition $\Phi_{\tau-t}=\binom{u_{\tau-t}}{v_{\tau-t}}\in L^2_{\mathcal{F}_{\tau-t}}(\Omega,\ell^2\times\ell^2)$. Firstly, we discuss uniform estimates of solutions of the solutions in $L^2(\Omega,\mathbb{P})$.\par
\begin{lem}
\label{lem:4.1}
Suppose (A1)-(A3), \eqref{eq:2.21}--\eqref{eq:2.23} and \eqref{eq:2.25} hold. Then for every $\tau\in\mathbb{R}$ and $D_1=\{D_1(t):t\in\mathbb{R}\}\in\mathcal{D}_0$, there exists $T:=T(\tau,D_1)>0$, such that for all $t\geq T,\ \Phi_{\tau-t}=\binom{u_{\tau-t}}{v_{\tau-t}}\in L^2_{\mathcal{F}_{\tau-t}}(\Omega,\ell^2\times\ell^2)$ with $\mathcal{L}_{\Phi_{\tau-t}}\in D_1(\tau-t)$, the solution $\Phi=\binom{u}{v}$ of system \eqref{eq:2.16} satisfies
\begin{equation}
\begin{aligned}
	\label{eq:4.1}
	\mathbb{E}\left(\|u(\tau,\tau-t,u_{\tau-t})\|^2+\|v(\tau,\tau-t,v_{\tau-t})\|^2\right)\leq L_1+L_1\int_{-\infty}^{\tau}e^{\beta(s-\tau)}\|g(s)\|^2ds,
\end{aligned}
\end{equation}
where $L_1>0$ is a constant depending on $\beta,\ \alpha,\ \delta,\ \|\gamma\|_{L^\infty([\tau,T],\ell^2)}$ and $\|\chi\|_{L^\infty(\mathbb{R},\ell^2)}$, but independent of $\tau$ and $D_1$.
\end{lem}
\begin{proof}
By \eqref{eq:2.16} and $It\hat{o}$ formula, we have for all $t\geq\tau$,
\begin{align*}
	d\|u(t)\|^2=2(u(t),v(t))dt-2\delta\|u(t)\|^2dt
\end{align*}
and
\begin{align*}
	d\|v(t)\|^2=&-2(\lambda+\delta^2-\alpha\delta)(u(t),v(t))dt-2(\alpha-\delta)\|v(t)\|^2dt-2\nu(Au(t),v(t))dt\\&+2(f(t,u(t),\mathcal{L}_{u(t)}),v(t))dt+2(g(t),v(t))dt\\&+\|\tilde{\sigma}(t, u(t),\mathcal{L}_{u(t)})\|_{L_2(\ell^2,\ell^2)}^2dt+2(\tilde{\sigma}(t, u(t),\mathcal{L}_{u(t)}),v(t)dW(t)).
\end{align*}
Therefore, we have
\begin{equation}
	\begin{aligned}
	\label{eq:4.2}
	&\dfrac{d}{dt}\mathbb{E}\left((\lambda+\delta^2-\alpha\delta)\|u(t)\|^2+\|v(t)\|^2+\nu\|Bu(t)\|^2\right)\\&+2\delta(\lambda+\delta^2-\alpha\delta)\mathbb{E}(\|u(t)\|^2)+2(\alpha-\delta)\mathbb{E}(\|v(t)\|^2)+2\nu\delta\mathbb{E}(\|Bu(t)\|^2)\\=&2\mathbb{E}(f(t,u(t),\mathcal{L}_{u(t)}),v(t))+2\mathbb{E}(g(t),v(t))+\mathbb{E}\left(\|\tilde{\sigma}(t, u(t),\mathcal{L}_{u(t)})\|_{L_2(\ell^2,\ell^2)}^2\right).
\end{aligned}
\end{equation}
By  \eqref{eq:2.13}, \eqref{eq:2.8} and Young's inequality, we obtain
\begin{equation}
	\begin{aligned}
		\label{eq:4.3}
	2\mathbb{E}(f(t,u(t),\mathcal{L}_{u(t)}),v(t))&\leq \dfrac{2}{(\alpha-\delta)}\mathbb{E}\|f(t,u(t),\mathcal{L}_{u(t)}\|^2+\dfrac{(\alpha-\delta)}{2}\mathbb{E}\|v(t)\|^2
	 \\&\leq\dfrac{6}{(\alpha-\delta)}\|\gamma\|^2_{L^\infty([\tau,T],\ell^2)}+\dfrac{12}{(\alpha-\delta)}\|\gamma\|^2_{L^\infty([\tau,T],\ell^2)}\mathbb{E}\|u(t)\|^2+\dfrac{(\alpha-\delta)}{2}\mathbb{E}\|v(t)\|^2.
	\end{aligned}
\end{equation}
and
\begin{align}
	\label{eq:4.4}
2\mathbb{E}(g(t),v(t))\leq\dfrac{2}{(\alpha-\delta)}\mathbb{E}\|g(t)\|^2+\dfrac{2}{(\alpha-\delta)}\mathbb{E}\|v(t)\|^2.
\end{align}
For the last term on the right-hand side of \eqref{eq:4.2}, we infer from \eqref{eq:2.14} and \eqref{eq:2.8} that
\begin{equation}
	\begin{aligned}
		\label{eq:4.5}
		\mathbb{E}\left(\|\tilde{\sigma}(t, u(t),\mathcal{L}_{u(t)})\|_{L_2(\ell^2,\ell^2)}^2\right)\leq 8\|\kappa\|^2_{L^\infty(\mathbb{R},\ell^2)}\mathbb{E}\|u(t)\|^2+2\|\chi\|^2_{L^\infty(\mathbb{R},\ell^2)}.
	\end{aligned}
\end{equation}
Substituting \eqref{eq:4.3}--\eqref{eq:4.5} into \eqref{eq:4.2}, by \eqref{eq:2.21}--\eqref{eq:2.23}, we have
\begin{equation}
	\begin{aligned}
		\label{eq:4.6}
		&\dfrac{d}{dt}\mathbb{E}\left[(\lambda+\delta^2-\alpha\delta)\|u(t)\|^2+\|v(t)\|^2+\nu\|Bu(t)\|^2\right]\\&+\beta\mathbb{E}\left[(\lambda+\delta^2-\alpha\delta)\|u(t)\|^2+\|v(t)\|^2+\nu\|Bu(t)\|^2\right]\\\leq&\dfrac{2}{(\alpha-\delta)}(3\|\gamma\|^2_{L^\infty([\tau,T],\ell^2)}+\|g(t)\|^2)+2\|\chi\|^2_{L^\infty(\mathbb{R},\ell^2)}.
	\end{aligned}
\end{equation}
Multiplying \eqref{eq:4.6} by $e^{\beta t}$ and integrating over $(\tau-t,\tau)$ with $t\in\mathbb{R}^+$, we see 
\begin{equation}
	\begin{aligned}
		\label{eq:4.7}
	\mathbb{E}&\left[(\lambda+\delta^2-\alpha\delta)\|u(\tau,\tau-t,u_{\tau-t})\|^2+\|v(\tau,\tau-t,v_{\tau-t})\|^2+\nu\|Bu(\tau,\tau-t,u_{\tau-t})\|^2\right]
	\\&\leq e^{-\beta t}\mathbb{E}\left[(\lambda+\delta^2-\alpha\delta)\|u_{\tau-t}\|^2+\|v_{\tau-t}\|^2+\nu\|Bu_{\tau-t}\|^2\right]
	\\&+\dfrac{2}{(\alpha-\delta)}\int_{\tau-t}^{\tau}e^{\beta(s-\tau)}\|g(s)\|^2ds+\frac{1}{\beta}(\dfrac{6}{(\alpha-\delta)}\|\gamma\|^2_{L^\infty([\tau,T],\ell^2)}+2\|\chi\|^2_{L^\infty(\mathbb{R},\ell^2)}).
	\end{aligned}
\end{equation}
Since $\mathcal{L}_{\Phi_{\tau-t}}\in D_1(\tau-t)$ with $D_1\in\mathcal{D}_0$, we get
\begin{equation*}
	\begin{aligned}
		&\lim_{t \to\infty}e^{-\beta t}\mathbb{E}\left((\lambda+\delta^2-\alpha\delta)\|u_{\tau-t}\|^2+\|v_{\tau-t}\|^2+\nu\|Bu_{\tau-t}\|^2\right)\\ &\leq(\lambda+\delta^2-\alpha\delta+4\nu+1)e^{-\beta\tau}\lim_{t \to\infty}e^{\beta(\tau-t)}\|D_1(\tau-t)\|_{\mathcal{P}_2(\ell^2\times\ell^2)}^2=0.
	\end{aligned}
\end{equation*}
Hence, there exists $T=T(\tau,D_1)$ such that for all $t\geq T$,
\begin{equation*}
	\begin{aligned}
		e^{-\beta t}\mathbb{E}\left((\lambda+\delta^2-\alpha\delta)\|u_{\tau-t}\|^2+\|v_{\tau-t}\|^2+\nu\|Bu_{\tau-t}\|^2\right)\leq\frac{1}{\beta},
	\end{aligned}
\end{equation*}
This along with \eqref{eq:4.7} and \eqref{eq:2.25} concludes the proof.
\end{proof}
\par
\begin{lem}
	\label{lem:4.2}
	Suppose (A1)-(A3), \eqref{eq:2.21}--\eqref{eq:2.23} and \eqref{eq:2.25} hold. Then the solution $\Phi(\tau,\tau-t,\Phi_{\tau-t})=\binom{u(\tau,\tau-t,u_{\tau-t})}{v(\tau,\tau-t,v_{\tau-t})}$ of system \eqref{eq:2.16} with initial time $\tau-t$ satisfies
	\begin{equation}
		\begin{aligned}
			\label{eq:4.8}
			\sup_{t\geq0}\mathbb{E}\left(\|u(\tau,\tau-t,u_{\tau-t})\|^2+\|v(\tau,\tau-t,v_{\tau-t})\|^2\right)\leq L_2\mathbb{E}(\|u_{\tau-t}\|^2+\|v_{\tau-t}\|^2)+L_2(1+\|g\|^2),
		\end{aligned}
	\end{equation}
	where $\Phi_{\tau-t}=\binom{u_{\tau-t}}{v_{\tau-t}}\in L^2(\Omega,\mathcal{F}_{\tau-t},\ell^2\times\ell^2)$ and $L_2>0$ is a constant depending on $\lambda,\ \beta,\ \alpha,\ \|\gamma\|_{L^\infty([\tau,T],\ell^2)}$ and $\|\chi\|_{L^\infty(\mathbb{R},\ell^2)}$, but independent of $t$ and $\Phi_{\tau-t}$.
\end{lem}
\begin{proof}
	Since $g$ is a time-independent deterministic function, the desired result yields from \eqref{eq:4.6} and the Gronwall Lemma immediately.
\end{proof}\par

The next lemma entails acquiring uniform estimates on the tails of solutions to equation \eqref{eq:2.16}, which is essential in the proof of asymptotic compactness of the probability distribution of solutions. For this purpose, we define a smooth function $\rho:\mathbb{R}\rightarrow [0,1]$ such that
\begin{equation*} \rho (s)=
	\begin{cases} 0, & if \ \ \vert s\vert \leq1, \\1, & if \ \ \vert s\vert \geq2.
	\end{cases}
\end{equation*}

\begin{lem}
	\label{lem:4.3}
	Suppose (A1)-(A3), \eqref{eq:2.21}--\eqref{eq:2.23} and \eqref{eq:2.25} hold. Then for every $\varepsilon>0,\ \tau\in\mathbb{R}$ and $D_1=\{D_1(t):t\in\mathbb{R}\}\in\mathcal{D}_0$, there exists $T=T(\varepsilon,\tau,D_1)$ and $N=N(\varepsilon,\tau)\in\mathbb{N}^+$, such that for all $ t\geq \tau$, $n\geq\mathbb{N}$ and $\Phi_{\tau-t}=\binom{u_{\tau-t}}{v_{\tau-t}}\in L^2_{\mathcal{F}_{\tau-t}}(\Omega,\ell^2\times\ell^2)$ with $\mathcal{L}_{\Phi_{\tau-t}}\in D_1(\tau-t)$,  the solution $\Phi=\binom{u}{v}$ of system \eqref{eq:2.16} satisfies
\begin{align*}
	\mathbb{E}\left(\sum_{\left|i\right|\geq n}(\left|u_i(\tau,\tau-t,u_{\tau-t})\right|^2+\left|v_i(\tau,\tau-t,v_{\tau-t})\right|^2)\right)\leq\varepsilon.
\end{align*}
\end{lem}
\begin{proof}
	For each $n\in\mathbb{N}$, we denote $\rho_n=\left(\rho\left(\frac{\left|i\right|}{n}\right)\right)_{i\in\mathbb{Z}}$, $\rho_nu=\left(\rho\left(\frac{\left|i\right|}{n}\right)u_i\right)_{i\in\mathbb{Z}}$, and $\rho_nv=\left(\rho\left(\frac{\left|i\right|}{n}\right)v_i\right)_{i\in\mathbb{Z}}$.By \eqref{eq:2.16}, we get
\begin{align}
	\label{eq:4.9}
	d\rho_n\Phi(t)=\rho_n F(\Phi(t))dt+\rho_nG(\Phi(t))dW(t),
\end{align}	
	which along with $It\hat{o}$ formula implies that
\begin{align*}
	e^{\beta t}\|\rho_nu(t)\|^2=e^{\beta \tau}\|\rho_nu_{\tau}\|^2+(\beta-2\delta)\int_{\tau}^{t}e^{\beta s}\|\rho_nu(s)\|^2ds+2\int_{\tau}^{t}e^{\beta s}(\rho_n^2u(s),v(s))ds
\end{align*}
and
\begin{align*}
	&e^{\beta t}\|\rho_nv(t)\|^2+2(\lambda+\delta^2-\alpha\delta)\int_{\tau}^{t}e^{\beta s}(u(s),\rho_n^2v(s))ds+2(\alpha-\delta)\int_{\tau}^{t}e^{\beta s}\|\rho_nv(s)\|^2ds\\&+2\nu\int_{\tau}^{t}e^{\beta s}(Au(s),\rho_n^2v(s))ds\\=&e^{\beta \tau}\|\rho_nv_{\tau}\|^2+2\int_{\tau}^{t}e^{\beta s}(f(s,u(s),\mathcal{L}_{u(s)}),\rho_n^2v(s))ds+2\int_{\tau}^{t}e^{\beta s}(g(s),\rho_n^2v(s))ds\\&+\int_{\tau}^{t}e^{\beta s}\|\rho_n\tilde{\sigma}(s, u(s),\mathcal{L}_{u(s)})\|_{L_2(\ell^2,\ell^2)}^2ds+2\int_{\tau}^{t}e^{\beta s}(\rho_n^2v(s),\tilde{\sigma}(s, u(s),\mathcal{L}_{u(s)})dW(s))
\end{align*}
Thus, we have
\begin{equation}
	\begin{aligned}
		\label{eq:4.10}
		&\mathbb{E}\left[e^{\beta t}((\lambda+\delta^2-\alpha\delta)\|\rho_nu(t)\|^2+\|\rho_nv(t)\|^2)\right]+(2\delta-\beta)\mathbb{E}(\int_{\tau}^{t}e^{\beta s}(\lambda+\delta^2-\alpha\delta)\|\rho_nu(s)\|^2ds)\\&+(2(\alpha-\delta)-\beta)\mathbb{E}(\int_{\tau}^{t}e^{\beta s}\|\rho_nv(s)\|^2ds)
		\\=&\mathbb{E}\left[e^{\beta \tau}((\lambda+\delta^2-\alpha\delta)\|\rho_nu_{\tau}\|^2+\|\rho_nv_{\tau}\|^2)\right]-2\nu\mathbb{E}(\int_{\tau}^{t}e^{\beta s}(Au(s),\rho_n^2\dot{u}(s))ds)\\&-2\nu\delta\mathbb{E}(\int_{\tau}^{t}e^{\beta s}(Au(s),\rho_n^2u(s))ds)+2\mathbb{E}(\int_{\tau}^{t}e^{\beta s}(f(s,u(s),\mathcal{L}_{u(s)}),\rho_n^2v(s))ds)\\&+2\mathbb{E}(\int_{\tau}^{t}e^{\beta s}(g(s),\rho_n^2v(s))ds)+\mathbb{E}(\int_{\tau}^{t}e^{\beta s}\|\rho_n\tilde{\sigma}(s, u(s),\mathcal{L}_{u(s)})\|_{L_2(\ell^2,\ell^2)}^2ds).
	\end{aligned}
\end{equation}	
In the view of the following fact that
\begin{equation}
	\begin{aligned}
		\label{eq:4.11}
	(Au(s),\rho_n^2\dot{u}(s))=&(Bu(s),B(\rho_n^2\dot{u}(s)))\\=&\sum_{i\in\mathbb{Z}}(u_{i+1}(s)-u_i(s))\left(\rho^2\left(\frac{i+1}{n}\right)\dot{u}_(i+1)(s)-\rho^2\left(\frac{i}{n}\right)\dot{u}_i(s)\right)\\=&\sum_{i\in\mathbb{Z}}\rho^2\left(\frac{i}{n}\right)(u_{i+1}(s)-u_i(s))(\dot{u}_{i+1}(s)-\dot{u}_i(s))\\&+\sum_{i\in\mathbb{Z}}(\rho^2\left(\frac{i+1}{n}\right)-\rho^2\left(\frac{i}{n}\right))(u_{i+1}(s)-u_i(s))\dot{u}_{i+1}(s),
	\end{aligned}
\end{equation}	
then, the last term on the right-hand side of \eqref{eq:4.11} satisfies
\begin{equation}
	\begin{aligned}
		\label{eq:4.12}
	&\sum_{i\in\mathbb{Z}}(\rho^2\left(\frac{i+1}{n}\right)-\rho^2\left(\frac{i}{n}\right))(u_{i+1}(s)-u_i(s))\dot{u}_{i+1}(s)\\\leq&2n^{-1}c_0\sum_{i\in\mathbb{Z}}\left|u_{i+1}(s)-u_i(s)\right|\left|\dot{u}_{i+1}(s)\right|\\\leq&2n^{-1}c_0\sum_{i\in\mathbb{Z}}(\left|u_i(s)\right|^2+\left|\dot{u}_{i+1}(s)\right|^2)\\\leq&2n^{-1}c_0(\|u(s)\|^2+\|v(s)-\delta u(s)\|^2)\\\leq&2n^{-1}c_0(2\delta^2+3)(\|u(s)\|^2+\|v(s)\|^2),
	\end{aligned}
\end{equation}
where $c_0$ is a constant such that $\left|\rho'(s)\right|\leq c_0$ for all $s\in\mathbb{R}$.		
Then, it follows from \eqref{eq:4.11} and \eqref{eq:4.12} that
\begin{equation}
	\begin{aligned}
		\label{eq:4.13}
		(Au(s),\rho_n^2\dot{u}(s))\geq\frac{1}{2}\dfrac{d}{ds}\|\rho_nBu(s)\|^2-2n^{-1}c_0(2\delta^2+3)(\|u(s)\|^2+\|v(s)\|^2).
	\end{aligned}
\end{equation}
Therefore, the second term on the right-hand side of \eqref{eq:4.10} satisfies
\begin{equation}
	\begin{aligned}
		\label{eq:4.14}
	-2\nu\mathbb{E}(\int_{\tau}^{t}e^{\beta s}(Au(s),\rho_n^2\dot{u}(s))ds)\leq& -2\nu\mathbb{E}(\int_{\tau}^{t}e^{\beta s}\frac{1}{2}d\|\rho_nBu(s)\|^2)\\&+4\nu n^{-1}c_0(2\delta^2+3)\mathbb{E}(\int_{\tau}^{t}e^{\beta s}(\|u(s)\|^2+\|v(s)\|^2)ds)\\\leq&-\nu\mathbb{E}(e^{\beta t}\|\rho_nBu(t)\|^2-e^{\beta\tau}\|\rho_nBu_{\tau}\|^2)+\nu\beta\mathbb{E}(\int_{\tau}^{t}e^{\beta s}\|\rho_nBu(s)\|^2ds)\\&+4\nu n^{-1}c_0(2\delta^2+3)\mathbb{E}(\int_{\tau}^{t}e^{\beta s}(\|u(s)\|^2+\|v(s)\|^2)ds).
	\end{aligned}
\end{equation}	
Similarly to \eqref{eq:4.11}, we obtain
\begin{equation}
	\begin{aligned}
		\label{eq:4.15}
		(Au(s),\rho_n^2u(s))=&(Bu(s),B(\rho_n^2u(s)))=\sum_{i\in\mathbb{Z}}(u_{i+1}(s)-u_i(s))(\rho^2\left(\frac{i+1}{n}\right)u_{i+1}(s)-\rho^2\left(\frac{i}{n}\right)u_i(s))\\=&\sum_{i\in\mathbb{Z}}\rho^2\left(\frac{i}{n}\right)(u_{i+1}(s)-u_i(s))^2+\sum_{i\in\mathbb{Z}}(\rho^2\left(\frac{i+1}{n}\right)-\rho^2\left(\frac{i}{n}\right))(u_{i+1}(s)-u_i(s))u_{i+1}(s).
	\end{aligned}
\end{equation}	
Note that the last term on the right-hand side of \eqref{eq:4.15} is bounded by
	\begin{equation}
		\begin{aligned}
			\label{eq:4.16}
			\sum_{i\in\mathbb{Z}}(\rho^2\left(\frac{i+1}{n}\right)-\rho^2\left(\frac{i}{n}\right))(u_{i+1}(s)-u_i(s))u_{i+1}(s)\leq4n^{-1}c_0\mathbb{E}(\int_{\tau}^{t}e^{\beta s}\|u(s)\|^2ds).
		\end{aligned}
	\end{equation}
It yields from  \eqref{eq:4.15} and \eqref{eq:4.16} that the third term on the right-hand side of \eqref{eq:4.10} satisfies
\begin{equation}
	\begin{aligned}
		\label{eq:4.17}
	-2\nu\delta\mathbb{E}(\int_{\tau}^{t}e^{\beta s}(Au(s),\rho_n^2u(s))ds)\leq-2\nu\delta\mathbb{E}(\int_{\tau}^{t}e^{\beta s}\|\rho_nBu(s)\|^2ds)+8\nu\delta n^{-1}c_0\mathbb{E}(\int_{\tau}^{t}e^{\beta s}\|u(s)\|^2ds).
	\end{aligned}
\end{equation}
By \eqref{eq:2.4}, the forth term on the right-hand side of \eqref{eq:4.10} satisfies
\begin{equation}
	\begin{aligned}
		\label{eq:4.18}
		&2\mathbb{E}(\int_{\tau}^{t}e^{\beta s}(f(s,u(s),\mathcal{L}_{u(s)}),\rho_n^2v(s))ds)\\\leq&\mathbb{E}(\int_{\tau}^{t}e^{\beta s}(\dfrac{2}{\alpha-\delta}\|\rho_nf(s,u(s),\mathcal{L}_{u(s)})\|^2+\dfrac{\alpha-\delta}{2}\|\rho_nv(s)\|^2)ds)\\\leq&\dfrac{6}{\alpha-\delta}\mathbb{E}(\int_{\tau}^{t}e^{\beta s}\sum_{i\in\mathbb{Z}}\rho^2\left(\frac{i}{n}\right)\gamma_i(s)(1+\left|u_i(s)\right|^2+\mathbb{E}\left|u_i(s)\right|^2)ds)+\dfrac{\alpha-\delta}{2}\mathbb{E}(\int_{\tau}^{t}e^{\beta s}\|\rho_nv(s)\|^2ds)\\\leq&\dfrac{6}{\alpha-\delta}\int_{\tau}^{t}e^{\beta s}\sum_{\left|i\right|\geq n}\left|\gamma_i(s)\right|ds+\dfrac{12}{\alpha-\delta}\|\gamma\|^2_{L^\infty([\tau,T],\ell^2)}\mathbb{E}(\int_{\tau}^{t}e^{\beta s}\|\rho_nu(s)\|^2ds)+\dfrac{\alpha-\delta}{2}\mathbb{E}(\int_{\tau}^{t}e^{\beta s}\|\rho_nv(s)\|^2ds).
	\end{aligned}
\end{equation}
For the forth term on the right-hand side of \eqref{eq:4.10}, we have
\begin{align}
	\label{eq:4.19}
	2\mathbb{E}(\int_{\tau}^{t}e^{\beta s}(g(s),\rho_n^2v(s))ds)\leq\frac{2}{\alpha-\delta}\int_{\tau}^{t}e^{\beta s}\sum_{\left|i\right|\geq n}\left|g_i(s)\right|^2ds+\frac{\alpha-\delta}{2}\mathbb{E}(\int_{\tau}^{t}e^{\beta s}\|\rho_nv(s)\|^2ds).
\end{align}
By \eqref{eq:2.7}, we see that for any $u = (u_i)_{i \in \mathbb{Z}} \in \ell^2$ and $\mu = (\mu_i)_{i \in \mathbb{Z}} \in \mathcal{L}^2$,
\begin{align*}
\|\rho_n\tilde{\sigma}(s, u(s),\mu)\|_{L_2(\ell^2,\ell^2)}^2\leq 4\|\kappa(s)\|^2(\|\rho_nu(s)\|^2+\sum_{i\in\mathbb{Z}}\rho^2\left(\frac{i}{n}\right)\mathbb{W}_2^{\mathbb{R}}(\mu, \delta_0)^2)+2\|\rho_n\chi(s)\|^2,
\end{align*}
which implies that
\begin{equation}
	\begin{aligned}
		\label{eq:4.20}
		&\mathbb{E}(\int_{\tau}^{t}e^{\beta s}\|\rho_n\tilde{\sigma}(s, u(s),\mathcal{L}_{u(s)})\|_{L_2(\ell^2,\ell^2)}^2ds)\\\leq&8\|\kappa\|^2_{L^\infty(\mathbb{R},\ell^2)}\mathbb{E}(\int_{\tau}^{t}e^{\beta s}\|\rho_nu(s)\|^2ds)+2\int_{\tau}^{t}e^{\beta s}\sum_{\left|i\right|\geq n}\left|\chi_i(s)\right|ds.
	\end{aligned}
\end{equation}
By \eqref{eq:4.14},\eqref{eq:4.17}--\eqref{eq:4.20}, we find from \eqref{eq:4.10} that
\begin{equation}
	\begin{aligned}
		\label{eq:4.21}
		&\mathbb{E}\left[e^{\beta t}((\lambda+\delta^2-\alpha\delta)\|\rho_nu(t)\|^2+\|\rho_nv(t)\|^2)+\nu\|\rho_nBu(t)\|^2\right]\\&+(2\delta-\beta)\mathbb{E}(\int_{\tau}^{t}e^{\beta s}(\lambda+\delta^2-\alpha\delta)\|\rho_nu(s)\|^2ds)\\&+[(\alpha-\delta)-\beta]\mathbb{E}(\int_{\tau}^{t}e^{\beta s}\|\rho_nv(s)\|^2ds)+
		(2\delta-\beta)\mathbb{E}(\int_{\tau}^{t}e^{\beta s}\nu\|\rho_nBu(s)\|^2ds)\\\leq&\mathbb{E}\left[e^{\beta \tau}((\lambda+\delta^2-\alpha\delta)\|\rho_nu_{\tau}\|^2+\|\rho_nv_{\tau}\|^2)+\nu\|\rho_nBu_{\tau}\|^2\right]\\&+4\nu n^{-1}c_0(2\delta^2+2\delta+3)\mathbb{E}(\int_{\tau}^{t}e^{\beta s}(\|u(s)\|^2+\|v(s)\|^2)ds)\\&+(\dfrac{12}{\alpha-\delta}\|\gamma\|^2_{L^\infty([\tau,T],\ell^2)}+8\|\kappa\|^2_{L^\infty(\mathbb{R},\ell^2)})\mathbb{E}(\int_{\tau}^{t}e^{\beta s}\|\rho_nu(s)\|^2ds)\\&+\dfrac{6}{\alpha-\delta}\int_{\tau}^{t}e^{\beta s}\sum_{\left|i\right|\geq n}\left|\gamma_i(s)\right|ds+\frac{2}{\alpha-\delta}\int_{\tau}^{t}e^{\beta s}\sum_{\left|i\right|\geq n}\left|g_i(s)\right|^2ds+2\int_{\tau}^{t}e^{\beta s}\sum_{\left|i\right|\geq n}\left|\chi_i(s)\right|ds.
	\end{aligned}
\end{equation}	
Following \eqref{eq:2.23}, then replacing where $\tau$ and $t$ by $\tau-t$ and $\tau$, respectively, we obtain
	\begin{equation}
		\begin{aligned}
			\label{eq:4.22}
			&\mathbb{E}\left[(\lambda+\delta^2-\alpha\delta)\|\rho_nu(\tau,\tau-t,u_{\tau-t})\|^2+\|\rho_nv(\tau,\tau-t,v_{\tau-t})\|^2+\nu\|\rho_nBu(\tau,\tau-t,u_{\tau-t})\|^2\right]\\\leq&\mathbb{E}\left[e^{-\beta t}((\lambda+\delta^2-\alpha\delta)\|\rho_nu_{\tau-t}\|^2+\|\rho_nv_{\tau-t}\|^2+\nu\|\rho_nBu_{\tau-t}\|^2)\right]\\&+4\nu n^{-1}c_0(2\delta^2+2\delta+3)\mathbb{E}(\int_{-\infty}^{\tau}e^{\beta( s-\tau)}(\|u(s)\|^2+\|v(s)\|^2)ds)\\&+\frac{2}{\alpha-\delta}\int_{-\infty}^{\tau}e^{\beta (s-\tau)}\sum_{\left|i\right|\geq n}\left|g_i(s)\right|^2ds+\dfrac{6}{\alpha-\delta}\int_{-\infty}^{\tau}e^{\beta (s-\tau)}\sum_{\left|i\right|\geq n}\left|\gamma_i(s)\right|ds+2\int_{-\infty}^{\tau}e^{\beta (s-\tau)}\sum_{\left|i\right|\geq n}\left|\chi_i(s)\right|ds.
		\end{aligned}
	\end{equation}
For the first term on the right-hand side of \eqref{eq:4.22}, since $\mathcal{L}_{\Phi_{\tau-t}}\in D_1(\tau-t)$, we get that
	\begin{align*}
		&e^{-\beta\tau} e^{\beta(\tau-t)}\mathbb{E}\left[(\lambda+\delta^2-\alpha\delta)\|u_{\tau-t}\|^2+\|nv_{\tau-t}\|^2+\nu\|Bu_{\tau-t}\|^2\right]\\\leq& (\lambda+\delta^2-\alpha\delta+1+4
		\nu)e^{-\beta\tau}e^{\beta(\tau-t)}\mathbb{E}\|D_1(\tau-t)\|_{\mathcal{P}_2(\ell^2\times\ell^2)}^2\\&\rightarrow0,
	\end{align*}
as $t\rightarrow\infty$. Hence for $\varepsilon>0$, there exists $T_2=T_2(\varepsilon,\tau,D_1)\geq T_1$ such that for all $t\geq T_2$,
	\begin{align*}
		&e^{-\beta t}\mathbb{E}\left[(\lambda+\delta^2-\alpha\delta)\|\rho_nu_{\tau-t}\|^2+\|\rho_nv_{\tau-t}\|^2+\nu\|\rho_nBu_{\tau-t}\|^2\right]\\\leq& \frac{\varepsilon}{4}.
	\end{align*}
	Moreover, it follows from Lemma \ref{lem:4.2} that there exists a costant $C_2>0$ indepensent of $n$ such that for all $t\geq \tau$,
\begin{equation*}
	\begin{aligned}
		&4\nu n^{-1}c_0(2\delta^2+2\delta+3)\mathbb{E}(\int_{-\infty}^{\tau}e^{\beta( s-\tau)}(\|u(s)\|^2+\|v(s)\|^2)ds)\\\leq&4\nu (\beta n)^{-1}c_0(2\delta^2+2\delta+3)\sup_{t\geq0}\mathbb{E}((\|u(\tau,\tau-t,u_{\tau-t})\|^2+\|v(\tau,\tau-t,v_{\tau-t})\|^2)ds) \\\leq& \frac{C_2}{n}\rightarrow0\text{\ \ as\ \ }n\rightarrow\infty.
\end{aligned}
\end{equation*}
Therefore,  there exists $N_1=N_1(\varepsilon,\tau)\in\mathbb{N}^+$ such that for all $n\geq N_1$,
\begin{equation*}
	\begin{aligned}
		4\nu n^{-1}c_0(2\delta^2+2\delta+3)\mathbb{E}(\int_{-\infty}^{\tau}e^{\beta( s-\tau)}(\|u(s)\|^2+\|v(s)\|^2)ds)\leq\frac{\varepsilon}{4}.
	\end{aligned}
\end{equation*}
For the fourth term on the right-hand side of \eqref{eq:4.22}, by using \eqref{eq:2.25}, we notice that
\begin{equation*}
	\begin{aligned}
		\frac{2}{\alpha-\delta}\int_{-\infty}^{\tau}e^{\beta (s-\tau)}\sum_{\left|i\right|\geq n}\left|g_i(s)\right|^2ds\leq\frac{2}{\alpha-\delta}\int_{-\infty}^{\tau}e^{\beta (s-\tau)}\|g(s)\|^2ds<\infty.
	\end{aligned}
\end{equation*}
By the dominated convergence theorem, there exists $N_2=N_2(\varepsilon,\tau)\geq N_1$ such that for all $n\geq N_2$,
\begin{equation*}
	\begin{aligned}
		\frac{2}{\alpha-\delta}\int_{-\infty}^{\tau}e^{\beta (s-\tau)}\sum_{\left|i\right|\geq n}\left|g_i(s)\right|^2ds\leq\frac{\varepsilon}{4}.
	\end{aligned}
\end{equation*}
For the last two terms on the right-hand side of \eqref{eq:4.22}, using the fact that $\gamma\in L^\infty_{loc}(\mathbb{R}, \ell^2)$ and $\chi\in L^\infty(\mathbb{R}, \ell^2)$, we know that
\begin{equation*}
	\begin{aligned}
	&\dfrac{6}{\alpha-\delta}\int_{-\infty}^{\tau}e^{\beta (s-\tau)}\sum_{\left|i\right|\geq n}\left|\gamma_i(s)\right|ds+2\int_{-\infty}^{\tau}e^{\beta (s-\tau)}\sum_{\left|i\right|\geq n}\left|\chi_i(s)\right|ds\\\leq&(\dfrac{6}{\alpha-\delta}+2)\int_{-\infty}^{\tau}e^{\beta (s-\tau)}(\left|\gamma(s)\right|^2+\left|\chi(s)\right|^2)ds\\\leq&(\dfrac{6}{\alpha-\delta}+2)\beta^{-1}(\|\gamma\|_{L^\infty_{loc}(\mathbb{R}, \ell^2)}^2+\|\chi\|_{L^\infty(\mathbb{R}, \ell^2)}^2).
	\end{aligned}
\end{equation*}
By the dominated convergence theorem again, there exists $N_3=N_3(\varepsilon,\tau)\geq N_2$ such that for all $n\geq N_3$,
\begin{equation*}
	\begin{aligned}
		\dfrac{6}{\alpha-\delta}\int_{-\infty}^{\tau}e^{\beta (s-\tau)}\sum_{\left|i\right|\geq n}\left|\gamma_i(s)\right|ds+2\int_{-\infty}^{\tau}e^{\beta (s-\tau)}\sum_{\left|i\right|\geq n}\left|\chi_i(s)\right|ds\leq\frac{\varepsilon}{4}.
	\end{aligned}
\end{equation*}
Combining all above, we have for all $t\geq T_2$ and $n\geq N_3$ that
\begin{align*}
	\mathbb{E}\left(\sum_{\left|i\right|\geq n}(\left|u_i(\tau,\tau-t,u_{\tau-t})\right|^2+\left|v_i(\tau,\tau-t,v_{\tau-t})\right|^2)\right)<\varepsilon,
\end{align*}
as desired.
\end{proof}
\begin{lem}
	\label{lem:4.4}
	Suppose (A1)-(A3), \eqref{eq:2.21}--\eqref{eq:2.26} hold. Then for every $ \tau\in\mathbb{R}$ and $D_2=\{D_2(t):t\in\mathbb{R}\}\in\mathcal{D}$, there exists $T:=T(\tau,D_2)>0$, such that for all $ t>T$, $\Phi_{\tau-t}=\binom{u_{\tau-t}}{v_{\tau-t}}\in L^2_{\mathcal{F}_{\tau-t}}(\Omega,\ell^2\times\ell^2)$ with $\mathcal{L}_{\Phi_{\tau-t}}\in D_2(\tau-t)$,  the solution $\Phi=\binom{u}{v}$ of system \eqref{eq:2.16} satisfies
\begin{equation}
	\begin{aligned}
		\label{eq:4.23}
		\mathbb{E}\left(\|u(\tau,\tau-t,u_{\tau-t})\|^4+\|v(\tau,\tau-t,v_{\tau-t})\|^4\right)\leq L_3+L_3\int_{-\infty}^{\tau}e^{2\beta(s-\tau)}\|g(s)\|^4ds,
	\end{aligned}
\end{equation}
where $L_3>0$ is a constant depending on $\beta,\ \alpha,\ \delta,\ \|\gamma\|_{L^\infty([\tau,T],\ell^2)}$ and $\|\chi\|_{L^\infty(\mathbb{R},\ell^2)}$, but independent of $\tau$ and $D_2$.
\end{lem}
\begin{proof}
Let $\Phi(t)=\binom{u(t)}{v(t)}=\binom{u(\tau,\tau-t,u_{\tau-t})}{v(\tau,\tau-t,v_{\tau-t})}$ be the solution of \eqref{eq:2.16} with initial data $\Phi_{\tau-t}=\binom{u_{\tau-t}}{v_{\tau-t}}$ at initial time $\tau-t.$\par
By \eqref{eq:4.2} and $It\hat{o}$ formula, we have for all $t\geq\tau$,
\begin{align*}
	e^{2\beta t}\|u(t)\|^4=e^{2\beta\tau}\|u_{\tau}\|^4+2(\beta-2\delta)\int_{\tau}^{t}e^{2\beta s}\|u(s)\|^4ds+4\int_{\tau}^{t}e^{2\beta s}\|u(s)\|^2(u(s),v(s))ds,
\end{align*}
and
\begin{align*}
	e^{2\beta t}\|v(t)\|^4=&e^{2\beta\tau}\|v_{\tau}\|^4+4\int_{\tau}^{t}e^{2\beta s}\|v(s)\|^2(v(s),\tilde{\sigma}(s, u(s),\mathcal{L}_{u(s)})dW(s))\\&+2[\beta-2(\alpha-\delta)]\int_{\tau}^{t}e^{2\beta s}\|v(s)\|^4ds
	-4(\lambda+\delta^2-\alpha\delta)\int_{\tau}^{t}e^{2\beta s}\|v(s)\|^2(u(s),v(s))ds\\&-4\nu\int_{\tau}^{t}e^{2\beta s}\|v(s)\|^2(Au(s),v(s))ds+4\int_{\tau}^{t}e^{2\beta s}\|v(s)\|^2(f(s,u(s),\mathcal{L}_{u(s)}),v(s))ds\\&+4\int_{\tau}^{t}e^{2\beta s}\|v(s)\|^2(g(s),v(s))ds+6\int_{\tau}^{t}e^{2\beta s}\|v(s)\|^2\|\tilde{\sigma}(s, u(s),\mathcal{L}_{u(s)})\|_{L_2(\ell^2,\ell^2)}^2ds.
\end{align*}
Therefore, we have
\begin{equation}
	\begin{aligned}
		\label{eq:4.24}
		&\mathbb{E}\left(e^{2\beta t}(\|u(t)\|^4+\|v(t)\|^4)\right)+\mathbb{E}\left(2(2\delta-\beta)\int_{\tau}^{t}e^{2\beta s}\|u(s)\|^4ds+2(2(\alpha-\delta)-\beta)\int_{\tau}^{t}e^{2\beta s}\|v(s)\|^4ds\right)\\=&\mathbb{E}\left(e^{2\beta\tau}(\|u_{\tau}\|^4+\|v_{\tau}\|^4)\right)+4\mathbb{E}\left(\int_{\tau}^{t}e^{2\beta s}\|u(s)\|^2(u(s),v(s))ds\right)\\&-4(\lambda+\delta^2-\alpha\delta)\mathbb{E}\left(\int_{\tau}^{t}e^{2\beta s}\|v(s)\|^2(u(s),v(s))ds\right)-4\nu\mathbb{E}\left(\int_{\tau}^{t}e^{2\beta s}\|v(s)\|^2(Au(s),v(s))ds\right)\\&+4\mathbb{E}\left(\int_{\tau}^{t}e^{2\beta s}\|v(s)\|^2(f(s,u(s),\mathcal{L}_{u(s)}),v(s))ds\right)+4\mathbb{E}\left(\int_{\tau}^{t}e^{2\beta s}\|v(s)\|^2(g(s),v(s))ds\right)\\&+6\mathbb{E}\left(\int_{\tau}^{t}e^{2\beta s}\|v(s)\|^2\|\tilde{\sigma}(s, u(s),\mathcal{L}_{u(s)})\|_{L_2(\ell^2,\ell^2)}^2ds\right).
	\end{aligned}
\end{equation}
Then, we estimate each terms on the right-hand side of \eqref{eq:4.24} separately. For the second term, third term and sixth term, by H$\ddot{o}$lder and Young's inequality, we obtain,
\begin{equation}
	\begin{aligned}
		\label{eq:4.25}
		4\mathbb{E}\left(\int_{\tau}^{t}e^{2\beta s}\|u(s)\|^2(u(s),v(s))ds\right)&\leq4\mathbb{E}\left(\int_{\tau}^{t}e^{2\beta s}\|u(s)\|^3\|v(s)\|ds\right)
		\\&\leq\dfrac{(\alpha-\delta)}{3}\mathbb{E}\left(\int_{\tau}^{t}e^{2\beta s}\|v(s)\|^4ds\right)+3\varepsilon_1\mathbb{E}\left(\int_{\tau}^{t}e^{2\beta s}\|u(s)\|^4ds\right),
	\end{aligned}
\end{equation}
\begin{equation}
	\begin{aligned}
		\label{eq:4.26}
		&\left|-4(\lambda+\delta^2-\alpha\delta)\mathbb{E}\left(\int_{\tau}^{t}e^{2\beta s}\|v(s)\|^2(u(s),v(s))ds\right)\right|\\\leq&4(\lambda+\delta^2-\alpha\delta)\mathbb{E}\left(\int_{\tau}^{t}e^{2\beta s}\|u(s)\|\|v(s)\|^3ds\right)
		\\\leq&\dfrac{(\alpha-\delta)}{3}\mathbb{E}\left(\int_{\tau}^{t}e^{2\beta s}\|v(s)\|^4ds\right)+\varepsilon_2(\lambda+\delta^2-\alpha\delta)^4\mathbb{E}\left(\int_{\tau}^{t}e^{2\beta s}\|u(s)\|^4ds\right),
	\end{aligned}
\end{equation}
and
\begin{align}
	\label{eq:4.27}
	4\mathbb{E}\left(\int_{\tau}^{t}e^{2\beta s}\|v(s)\|^2(g(s),v(s))ds\right)\leq\dfrac{(\alpha-\delta)}{3}\mathbb{E}\left(\int_{\tau}^{t}e^{2\beta s}\|v(s)\|^4ds\right)+\varepsilon_2\mathbb{E}\left(\int_{\tau}^{t}e^{2\beta s}\|g(s)\|^4ds\right),
\end{align}
where $\varepsilon^3_1=\dfrac{3}{(\alpha-\delta)}$, $\varepsilon_2=\left(\dfrac{9}{(\alpha-\delta)}\right)^3$ are positive constants.\par
For the fourth term	on the right-hand side of \eqref{eq:4.24}, by \eqref{eq:2.10}, we see,
\begin{equation}
	\begin{aligned}
		\label{eq:4.28}
		&\left|-4\nu\mathbb{E}\left(\int_{\tau}^{t}e^{2\beta s}\|v(s)\|^2(Au(s),v(s))ds\right)\right|\\\leq&16\nu\mathbb{E}\left(\int_{\tau}^{t}e^{2\beta s}\|u(s)\|\|v(s)\|^3ds\right)
		\\\leq&\dfrac{(\alpha-\delta)}{3}\mathbb{E}\left(\int_{\tau}^{t}e^{2\beta s}\|v(s)\|^4ds\right)+256\varepsilon_2\nu^4\mathbb{E}\left(\int_{\tau}^{t}e^{2\beta s}\|u(s)\|^4ds\right).
	\end{aligned}
\end{equation}	
For the fifth term	on the right-hand side of \eqref{eq:4.24}, by \eqref{eq:2.13} and \eqref{eq:2.8}, we get,
\begin{equation}
	\begin{aligned}
		\label{eq:4.29}
		&4\mathbb{E}\left(\int_{\tau}^{t}e^{2\beta s}\|v(s)\|^2(f(s,u(s),\mathcal{L}_{u(s)}),v(s))ds\right)\\\leq&\dfrac{(\alpha-\delta)}{3}\mathbb{E}\left(\int_{\tau}^{t}e^{2\beta s}\|v(s)\|^4ds\right)+\varepsilon_2\mathbb{E}\left(\int_{\tau}^{t}e^{2\beta s}\|f(s,u(s),\mathcal{L}_{u(s)})\|^4ds\right)\\\leq&\dfrac{(\alpha-\delta)}{3}\mathbb{E}\left(\int_{\tau}^{t}e^{2\beta s}\|v(s)\|^4ds\right)+\varepsilon_2\mathbb{E}\left(\int_{\tau}^{t}e^{2\beta s}[9\|\gamma(s)\|^4(3+3\|u(s)\|^4+3(\mathbb{E}\|u(s)\|^2)^2)]ds\right)\\\leq&\dfrac{(\alpha-\delta)}{3}\mathbb{E}\left(\int_{\tau}^{t}e^{2\beta s}\|v(s)\|^4ds\right)+27\varepsilon_2\mathbb{E}\left(\int_{\tau}^{t}e^{2\beta s}\|\gamma(s)\|^4(1+2\|u(s)\|^4)ds\right).
	\end{aligned}
\end{equation}	
For the last term	on the right-hand side of \eqref{eq:4.24}, by \eqref{eq:2.14} and \eqref{eq:2.8}, we know,
\begin{equation}
	\begin{aligned}
		\label{eq:4.30}
		&6\mathbb{E}\left(\int_{\tau}^{t}e^{2\beta s}\|v(s)\|^2\|\tilde{\sigma}(s, u(s),\mathcal{L}_{u(s)})\|_{L_2(\ell^2,\ell^2)}^2ds\right)\\\leq&\dfrac{(\alpha-\delta)}{3}\mathbb{E}\left(\int_{\tau}^{t}e^{2\beta s}\|v(s)\|^4ds\right)+\dfrac{9}{2\varepsilon_3}\mathbb{E}\left(\int_{\tau}^{t}e^{2\beta s}\|\tilde{\sigma}(s, u(s),\mathcal{L}_{u(s)})\|_{L_2(\ell^2,\ell^2)}^4ds\right)\\\leq&\dfrac{(\alpha-\delta)}{3}\mathbb{E}\left(\int_{\tau}^{t}e^{2\beta s}\|v(s)\|^4ds\right)+\dfrac{9}{2\varepsilon_3}\mathbb{E}\left(\int_{\tau}^{t}e^{2\beta s}[32\|\kappa(s)\|^4(2\|u(s)\|^4+2(\mathbb{E}\|u(s)\|^2)^2)+8\|\chi(s)\|^2]ds\right)\\\leq&\dfrac{(\alpha-\delta)}{3}\mathbb{E}\left(\int_{\tau}^{t}e^{2\beta s}\|v(s)\|^4ds\right)+\dfrac{576}{\varepsilon_3}\|\kappa\|_{L^\infty(\mathbb{R}, \ell^2)}^4\mathbb{E}\left(\int_{\tau}^{t}e^{2\beta s}\|u(s)\|^4ds\right)+\dfrac{36}{\varepsilon_3}\mathbb{E}\left(\int_{\tau}^{t}e^{2\beta s}\|\chi(s)\|^4ds\right),
	\end{aligned}
\end{equation}	
where $\varepsilon_3=\dfrac{(\alpha-\delta)}{3}$.\par 
It follows from \eqref{eq:4.24}--\eqref{eq:4.30} that for all $t\geq \tau$,
\begin{equation}
	\begin{aligned}
		\label{eq:4.31}
		&\mathbb{E}\left(e^{2\beta t}(\|u(t)\|^4+\|v(t)\|^4)\right)+\mathbb{E}\left(2(2\delta-\beta)\int_{\tau}^{t}e^{2\beta s}\|u(s)\|^4ds+2(2(\alpha-\delta)-\beta)\int_{\tau}^{t}e^{2\beta s}\|v(s)\|^4ds\right)\\\leq&\mathbb{E}\left(e^{2\beta\tau}(\|u_{\tau}\|^4+\|v_{\tau}\|^4)\right)\\&+\{3\varepsilon_1+\varepsilon_2[(\lambda+\delta^2-\alpha\delta)^4+256\nu^4+54\|\gamma\|^4_{L^\infty([\tau,T],\ell^2)}]+\dfrac{576}{\varepsilon_3}\|\kappa\|_{L^\infty(\mathbb{R}, \ell^2)}^4\}\mathbb{E}\left(\int_{\tau}^{t}e^{2\beta s}\|u(s)\|^4ds\right)\\&+2(\alpha-\delta)\mathbb{E}\left(\int_{\tau}^{t}e^{2\beta s}\|v(s)\|^4ds\right)+27\varepsilon_2\mathbb{E}\left(\int_{\tau}^{t}e^{2\beta s}\|\gamma(s)\|^4ds\right)\\&+\varepsilon_2\mathbb{E}\left(\int_{\tau}^{t}e^{2\beta s}\|g(s)\|^4ds\right)+\dfrac{36}{\varepsilon_3}\mathbb{E}\left(\int_{\tau}^{t}e^{2\beta s}\|\chi(s)\|^4ds\right).
	\end{aligned}
\end{equation}
Replacing $\tau$ and $t$ in \eqref{eq:4.31} by $\tau-t$ and $\tau$, respectively, by \eqref{eq:2.24} we have for all $t\geq \tau$,
\begin{equation}
	\begin{aligned}
		\label{eq:4.32}
		\mathbb{E}\left(\|u(\tau,\tau-t,u_{\tau-t})\|^4+\|v(\tau,\tau-t,v_{\tau-t})\|^4\right)\leq&\mathbb{E}\left(e^{-2\beta t}(\|u_{\tau-t}\|^4+\|v_{\tau-t}\|^4)\right)\\&+27\varepsilon_2{(2\beta)}^{-1}\|\gamma\|^4_{L^\infty([\tau,T],\ell^2)}+18{(\varepsilon_3\beta)}^{-1}\|\chi\|_{L^\infty(\mathbb{R}, \ell^2)}^4\\&+\varepsilon_2\mathbb{E}\left(\int_{\tau-t}^{\tau}e^{2\beta(s-\tau)}\|g(s)\|^4ds\right).
	\end{aligned}
\end{equation}
Note that $\mathcal{L}_{\Phi_{\tau-t}}\in D_2(\tau-t)$ with $D_2\in\mathcal{D}$, we get
\begin{equation*}
	\begin{aligned}
		\lim_{t \to\infty}e^{-2\beta t}\mathbb{E}\left(\|u_{\tau-t}\|^4+\|v_{\tau-t}\|^4\right)\leq e^{-2\beta\tau}\lim_{t \to\infty}e^{2\beta(\tau-t)}\|D_2(\tau-t)\|_{\mathcal{P}_4(\ell^2\times\ell^2)}^2=0.
	\end{aligned}
\end{equation*}
Hence, there exists $T=T(\tau,D_2)$ such that for all $t\geq T$,
\begin{equation*}
	\begin{aligned}
		e^{-2\beta t}\mathbb{E}\left(\|u_{\tau-t}\|^4+\|v_{\tau-t}\|^4\right)\leq1,
	\end{aligned}
\end{equation*}
Then we can obtain that for all $t\geq T$,
\begin{equation*}
	\begin{aligned}
			\mathbb{E}\left(\|u(\tau,\tau-t,u_{\tau-t})\|^4+\|v(\tau,\tau-t,v_{\tau-t})\|^4\right)\leq&1+27\varepsilon_2{(2\beta)}^{-1}\|\gamma\|^4_{L^\infty([\tau,T],\ell^2)}+18{(\varepsilon_3\beta)}^{-1}\|\chi\|_{L^\infty(\mathbb{R}, \ell^2)}^4\\&+\varepsilon_2\mathbb{E}\left(\int_{\tau-t}^{\tau}e^{2\beta(s-\tau)}\|g(s)\|^4ds\right).
	\end{aligned}
\end{equation*}
which along with \eqref{eq:2.26} completes the proof.

\end{proof}
\section{Existence of pullback measure attractors}
\indent
\par
In this section, we are dedicated proving the existence of $\mathcal{D}$-pullback measure attractors  of system \eqref{eq:2.16} in $\mathcal{P}_4(\ell^2\times\ell^2)$. For this purpose, we first define a  non-autonomous dynamical system in $\mathcal{P}_2(\ell^2\times\ell^2)$.\par 
Given $\tau,t\in\mathbb{R}$ with $t\geq\tau$, define $p^*(\tau,t):\mathcal{P}_2(\ell^2\times\ell^2)\rightarrow\mathcal{P}_2(\ell^2\times\ell^2)$ by, for every $\mu\in\mathcal{P}_2(\ell^2\times\ell^2)$,
	\begin{align}
	\label{eq:5.1}
	p^*_{\tau,t}\mu=\mathcal{L}_{\Phi(t,\tau,\Phi_{\tau})},
\end{align}
where $\Phi(t,\tau,\Phi_{\tau})$ is the solution of \eqref{eq:2.16} with $\Phi_{\tau}=\binom{u_{\tau}}{v_{\tau}}\in L^2(\Omega, C([\tau, T], \ell^2\times\ell^2)) $ such that $\mathcal{L}_{\Phi_{\tau}}=\mu$.\par 
For every $t\in \mathbb{R}^+$ and $\tau\in \mathbb{R}$, we define $\Theta(t,\tau):\mathcal{P}_2(\ell^2\times\ell^2)\rightarrow\mathcal{P}_2(\ell^2\times\ell^2)$ by
	\begin{align}
	\label{eq:5.2}
	\Theta(t,\tau)\mu=p^*(\tau,\tau+t)\mu,\ \ \ \ \ \forall\mu\in\mathcal{P}_2(\ell^2\times\ell^2).
\end{align}
By the uniqueness of solutions of \eqref{eq:2.16}, we can get that for all $t,s\in \mathbb{R}^+$, $\tau\in \mathbb{R}$ and $\Phi_{\tau}\in L^2_{\mathcal{F}_{\tau}}(\Omega, \ell^2\times\ell^2)$,
	\begin{align*}
	\Phi(t+s+\tau,\tau,\Phi_{\tau})=\Phi(t+s+\tau,s+\tau,\Phi(s+\tau,\tau,\Phi_{\tau})),
\end{align*}
hence, for all $t,s\in \mathbb{R}^+$, $\tau\in \mathbb{R}$ and $\mu\in\mathcal{P}_2(\ell^2\times\ell^2)$,
	\begin{align}
	\label{eq:5.3}
	\Theta(t+s,\tau)\mu=\Theta(t,s+\tau)(\Theta(s,\tau)\mu),\ \ \ \ \ \forall\mu\in\mathcal{P}_4(\ell^2\times\ell^2).
\end{align}
Obviously, $\Theta$ is a non-autonomous dynamical system on $\mathcal{P}_2(\ell^2\times\ell^2)$.
\par
Next we will prove $\Theta$ is continuous over bounded subset $D_2$ of $\mathcal{P}_4(\ell^2\times\ell^2)$. To this end, we first prove the following lemma.
\begin{lem}
	\label{lem:5.1}
	Suppose (A1)-(A3) hold. Then, for any $\mu_{1},\mu_{2}\in\mathcal{P}_2(\ell^2\times\ell^2)$ and $t\geq\tau$,
	\begin{align*}
	\mathbb{W}_2(p^*_{\tau,t}\mu_1,	p^*_{\tau,t}\mu_2)\leq a_1\mathbb{W}_2(\mu_{1},\mu_{2})e^{a_2(t-\tau)},
\end{align*}
where $a_1,a_2>0$ are two constants independent of $\tau$ and $t$.	
\end{lem}
\begin{proof}
Let $\Phi_1(t)=\binom{u_1(t,\tau,u_{1\tau})}{v_1(t,\tau,v_{1\tau})}$, $\Phi_{1\tau}=\binom{u_{1\tau}}{v_{1\tau}}$  with $\mathcal{L}_{u_{1\tau}}=\mu_{1}$ and $\Phi_2(t)=\binom{u_2(t,\tau,u_{2\tau})}{v_2(t,\tau,v_{2\tau})}$, $\Phi_{2\tau}=\binom{u_{2\tau}}{v_{2\tau}}$  with $\mathcal{L}_{u_{2\tau}}=\mu_{2}$ be two solutions of \eqref{eq:2.16}. Let $y(t)=u_1(t,\tau,u_{1\tau})-u_2(t,\tau,u_{2\tau})$,  $z(t)=v_1(t,\tau,v_{1\tau})-v_2(t,\tau,v_{2\tau})$ and  
	\begin{align*}
	\mathbb{W}_2(\mu_{1},\mu_{2})^2=\mathbb{E}\|u_{1\tau}-u_{2\tau}\|^2.
\end{align*}
Then, by \eqref{eq:2.16}, we have that for all $t\geq\tau$,
	\begin{align*}
	d(\Phi_1(t)-\Phi_2(t))=[F(\Phi_1(t))-F(\Phi_2(t))]dt+[G(\Phi_1(t))-G(\Phi_2(t))]dW(t).
\end{align*}
Following the It\^{o} formula, we get that for all $t\geq\tau$,
\begin{equation}
	\begin{aligned}
		\label{eq:5.4}
		&\mathbb{E}\left((\lambda+\delta^2-\alpha\delta)\|y(t)\|^2+\|z(t)\|^2+\nu\|By(t)\|^2\right)+2\delta(\lambda+\delta^2-\alpha\delta)\mathbb{E}(\int_{\tau}^{t}\|y(s)\|^2ds)\\&+2\nu\delta\mathbb{E}(\int_{\tau}^{t}\|By(s)\|^2ds)+2(\alpha-\delta)\mathbb{E}(\int_{\tau}^{t}\|z(s)\|^2ds)
		\\=&\mathbb{E}\left((\lambda+\delta^2-\alpha\delta)\|u_{1\tau}-u_{2\tau}\|^2+\|v_{1\tau}-v_{2\tau}\|^2+\nu\|B(u_{1\tau}-u_{2\tau})\|^2\right)\\&+2\mathbb{E}(\int_{\tau}^{t}(f(s,u_1(s),\mu_{1})-f(s,u_2(s),\mu_{2}),z(s))ds)\\&+\mathbb{E}(\int_{\tau}^{t}\|\tilde{\sigma}(s, u_1(s),\mu_{1})-\tilde{\sigma}(s, u_2(s),\mu_{2})\|_{L_2(\ell^2,\ell^2)}^2ds).
	\end{aligned}
\end{equation}	
We estimate each term in the right-hand side of \eqref{eq:5.4}, respectively. For the second term, by \eqref{eq:2.3}, \eqref{eq:2.11} and \eqref{eq:2.8}, we have that
\begin{equation}
	\begin{aligned}
		\label{eq:5.5}
	&2\mathbb{E}(\int_{\tau}^{t}(f(s,u_1(s),\mu_{1})-f(s,u_2(s),\mu_{2}),z(s))ds)\\=&2\mathbb{E}(\int_{\tau}^{t}(f(s,u_1(s),\mu_{1})-f(s,u_2(s),\mu_{1}),z(s))ds)+2\mathbb{E}(\int_{\tau}^{t}(f(s,u_2(s),\mu_{1})-f(s,u_2(s),\mu_{2}),z(s))ds)\\\leq&2\mathbb{E}(\int_{\tau}^{t}\|\eta(s)\|\|y(s)\|\|z(s)\|ds)+\mathbb{E}(\int_{\tau}^{t}(\|z(s)\|^2+\|f(s,u_2(s),\mu_{1})-f(s,u_2(s),\mu_{2})\|^2)ds)\\\leq&2\mathbb{E}(\int_{\tau}^{t}\|\eta(s)\|\|y(s)\|\|z(s)\|ds)+\mathbb{E}(\int_{\tau}^{t}(\|z(s)\|^2+2\|\psi(s)\|^2\mathbb{E}\|y(s)\|^2)ds)\\\leq&(\|\eta\|^2_{L^\infty([\tau,T],\ell^2)}+2\|\psi\|^2_{L^\infty([\tau,T],\ell^2)})\int_{\tau}^{t}\mathbb{E}\|y(s)\|^2ds+2\int_{\tau}^{t}\mathbb{E}\|z(s)\|^2ds.
	\end{aligned}
\end{equation}
For the third term, by using \eqref{eq:2.12} and \eqref{eq:2.8}, we obtain
\begin{equation}
	\begin{aligned}
		\label{eq:5.6}
		\mathbb{E}(\int_{\tau}^{t}\|\tilde{\sigma}(s, u_1(s),\mu_{1})-\tilde{\sigma}(s, u_2(s),\mu_{2})\|_{L_2(\ell^2,\ell^2)}^2ds)\leq4\|\kappa\|^2_{L^\infty(\mathbb{R},\ell^2)}\int_{\tau}^{t}\mathbb{E}\|y(s)\|^2ds.
	\end{aligned}
\end{equation}
It follows from \eqref{eq:5.4}-\eqref{eq:5.6}, we know that for all $t\geq\tau$
\begin{equation*}
	\begin{aligned}
		&\mathbb{E}\left((\lambda+\delta^2-\alpha\delta)\|y(t)\|^2+\|z(t)\|^2+\nu\|By(t)\|^2\right)
		\\\leq&\mathbb{E}\left((\lambda+\delta^2-\alpha\delta)\|u_{1\tau}-u_{2\tau}\|^2+\|v_{1\tau}-v_{2\tau}\|^2+\nu\|B(u_{1\tau}-u_{2\tau})\|^2\right)\\&+(\|\eta\|^2_{L^\infty([\tau,T],\ell^2)}+2\|\psi\|^2_{L^\infty([\tau,T],\ell^2)}+4\|\kappa\|^2_{L^\infty(\mathbb{R},\ell^2)})\int_{\tau}^{t}\mathbb{E}\|y(s)\|^2ds+2\int_{\tau}^{t}\mathbb{E}\|z(s)\|^2ds,
	\end{aligned}
\end{equation*}	
which is
\begin{equation}
	\begin{aligned}
		\label{eq:5.7}
		\mathbb{E}\left(\|\Phi_1(t)-\Phi_2(t)\|^2\right)\leq a_1\mathbb{E}\left(\|\Phi_{1\tau}-\Phi_{2\tau}\|^2\right)+a_2\int_{\tau}^{t}\mathbb{E}\left(\|\Phi_1(s)-\Phi_2(s)\|^2\right)ds,
	\end{aligned}
\end{equation}
where $a_1,a_2>0$ are constants independent of $\tau$ and $t$.\par
Then using Gronwall's Lemma, we get that for all $t\geq\tau$,
\begin{align*}
	\mathbb{E}\left(\|\Phi_1(t)-\Phi_2(t)\|^2\right)\leq a_1\mathbb{E}\left(\|\Phi_{1\tau}-\Phi_{2\tau}\|^2\right)e^{a_2(t-\tau)}.
\end{align*}	
\end{proof}
\par

From Lemma \ref{lem:5.1} we obtain as an immediate corollary the continuity of the mapping $\Theta(t,\tau)$ on bounded subsets of $\mathcal{P}_{4}(\ell^{2}\times\ell^{2})$.
\begin{Corollary}\label{coro:5.1}
	Assume (A1)-(A3) hold. Let $S\subset\mathcal{P}_{4}(\ell^{2}\times\ell^{2})$ be a bounded set, i.e. there exists $R>0$ such that
\begin{align*}
	\int_{\ell^{2}\times\ell^{2}}\|x\|_{\ell^{2}\times\ell^{2}}^{4}\,\mathrm{d}\mu(x) \leq R \quad\text{for all }\mu\in S .
\end{align*}
	Then for every fixed $t\geq 0$ and every $\tau\in\mathbb{R}$, the mapping
\begin{align*}
	\Theta(t,\tau) : (S,d_{\mathcal{P}_{4}(\ell^{2}\times\ell^{2})})\longrightarrow(\mathcal{P}_{4}(\ell^{2}\times\ell^{2}),d_{\mathcal{P}_{4}(\ell^{2}\times\ell^{2})})
\end{align*}
is continuous. In other words, if $\{\mu_{n}\}_{n=1}^{\infty}\subset S$ satisfies
\begin{align*}
	d_{\mathcal{P}_{4}(\ell^{2}\times\ell^{2}}(\mu_{n},\mu)\longrightarrow 0\qquad (n\to\infty),
\end{align*}
	then
\begin{align*}
	d_{\mathcal{P}_{4}(\ell^{2}\times\ell^{2})}\left(\Phi(t,\tau)\mu_{n},	\Phi(t,\tau)\mu\right)\longrightarrow 0	\qquad (n\to\infty).
\end{align*}
\end{Corollary}

\begin{proof}
Let $\{\mu_{n}\}_{n=1}^{\infty}\subset S$ be a sequence converging to $\mu\in S$ in $(\mathcal{P}_{4}(\ell^{2}\times\ell^{2}),d_{\mathcal{P}_{4}(\ell^{2}\times\ell^{2})})$. By Lemma \ref{lem:5.1}, for every $t\in\mathbb{R}^{+}$ and $\tau\in\mathbb{R}$,
\begin{equation}\label{eq:5.8}
	\mathbb{W}_{2}\left(p^{*}_{\tau,\tau+t}\mu_{n},	p^{*}_{\tau,\tau+t}\mu\right)\leq a_{1}\mathbb{W}_{2}(\mu_{n},\mu)\,e^{a_{2}t}.
\end{equation}
Since $\mathbb{W}_{2}(\mu_{n},\mu)\to 0$ as $n\to\infty$, inequality \eqref{eq:5.8} implies that
\begin{equation}\label{eq:5.9}
	\mathbb{W}_{2}\left(p^{*}_{\tau,\tau+t}\mu_{n},	p^{*}_{\tau,\tau+t}\mu\right)\longrightarrow 0,	\qquad n\to\infty.
\end{equation}

On the other hand, since $\mathbb{W}_{2}(\mu_{n},\mu)\to 0$, the sequence $\{\mu_{n}\}$ converges to $\mu$ weakly; hence by Skorokhod's representation theorem there exists a probability space $(\tilde\Omega,\tilde{\mathcal{F}},\tilde{\mathbb{P}})$ and $\ell^{2}\times\ell^{2}$-valued random variables
$\xi_{0},\xi_{n}\in L^{2}(\tilde\Omega,\tilde{\mathcal{F}},\tilde{\mathbb{P}};\ell^{2}\times\ell^{2})$ such that $\mathcal{L}_{\xi_{0}}=\mu$, $\mathcal{L}_{\xi_{n}}=\mu_{n}$ and
\begin{equation}\label{eq:5.10}
	\lim_{n\to\infty}\|\xi_{n}-\xi_{0}\|_{\ell^{2}\times\ell^{2}}=0,
	\qquad \tilde{\mathbb{P}}\text{-a.s.}.
\end{equation}
Because $\mu_{n}\in S$ for all $n$, there exists $R>0$ such that
\begin{align*}
\mathbb{E}^{\tilde{\mathbb{P}}}\|\xi_{n}\|_{\ell^{2}\times\ell^{2}}^{4}
= \int_{\ell^{2}\times\ell^{2}}\|x\|_{\ell^{2}\times\ell^{2}}^{4}\,\mathrm{d}\mu_{n}(x) \le R,
\qquad n\in\mathbb{N}^{+}.
\end{align*}
Thus $\{\xi_{n}\}_{n=1}^{\infty}$ is uniformly integrable in $L^{2}(\tilde\Omega,\ell^{2}\times\ell^{2})$.  Combining this with \eqref{eq:5.10} and applying Vitali's convergence theorem, we obtain 
\begin{align*}
\lim_{n\to\infty}
\mathbb{E}^{\tilde{\mathbb{P}}}\|\xi_{n}-\xi_{0}\|_{\ell^{2}\times\ell^{2}}^{2}=0.
\end{align*}
Consequently,
\begin{align*}
\mathbb{W}_{2}(\mu_{n},\mu)\leq \left(\mathbb{E}^{\tilde{\mathbb{P}}}\|\xi_{n}-\xi_{0}\|_{\ell^{2}\times\ell^{2}}^{2}\right)^{\frac{1}{2}}\longrightarrow 0,\qquad n\to\infty,
\end{align*}
which together with \eqref{eq:5.9} completes the proof.
\end{proof}
\par 
It follows from Lemma \ref{lem:5.1} that the mapping $\Theta$ induces a non-autonomous dynamical system on the space $(\mathcal{P}_4(\ell^{2}\times\ell^{2}), d_{\mathcal{P}(\ell^{2}\times\ell^{2})})$, which remains continuous on every bounded subset of $\mathcal{P}_4(\ell^{2}\times\ell^{2})$. Our next objective is to establish the existence of $\mathcal{D}$-pullback absorbing sets for $\Theta$.

\begin{lem}
	\label{lem:5.2}
		Suppose (A1)-(A3), \eqref{eq:2.21}--\eqref{eq:2.26} hold. Given $\tau\in\mathbb{R}$, denote by
	\begin{align}\label{eq:5.11}
		K(\tau)=\left\{\mu\in\mathcal{P}_4(\ell^2\times\ell^2):\int_{\ell^{2}\times\ell^{2}}\|x\|_{\ell^{2}\times\ell^{2}}^{4}\,\mathrm{d}\mu(x)\leq L_4(\tau) \right\},
	\end{align}
	where $L_4(\tau)=L_3+L_3\int_{-\infty}^{\tau}e^{2\beta(s-\tau)}\|g(s)\|^4ds$, and $L_3$ is the same constant as in Lemma \ref{lem:4.4}, independent of $\tau$. Then $K=\{K(\tau):\tau\in\mathbb{R}\}\in\mathcal{D}$ is a closed $\mathcal{D}$-pullback absorbing set of $\Theta$.	
\end{lem}
\begin{proof}
	By \eqref{eq:5.2} and Lemma \ref{lem:4.4}, we get that for each $\tau\in\mathbb{R}$
	and $D_2\in\mathcal{D}$, there exists $T = T (\tau, D_2) > 0$, such that for all $t\geq T$, it holds
	\begin{align}\label{eq:5.12}
		\Theta(t,\tau-t)D_2(\tau-t)\subseteq K(\tau).
	\end{align}
	Then by \eqref{eq:5.11} and \eqref{eq:2.26}, we see
	\begin{equation*}
		\begin{aligned}
			\lim_{\tau\rightarrow -\infty}e^{2\beta\tau}\|K(\tau)\|^4_{\mathcal{P}_4(\ell^2\times\ell^2)}=e^{2\beta\tau}L_3+L_3\int_{-\infty}^{\tau}e^{2\beta s}\|g(s)\|^4ds=0.
		\end{aligned}
	\end{equation*}
	which shows $K=\{K(\tau):\tau\in\mathbb{R}\}\in\mathcal{D}$. Combining this with \eqref{eq:5.12}, we infer that the set $K$ defined in \eqref{eq:5.11} indeed forms a closed $\mathcal{D}$-pullback absorbing set for $\Theta$.
\end{proof}
\par 

\begin{lem}
	\label{lem:5.3}
	Suppose (A1)-(A3), \eqref{eq:2.21}--\eqref{eq:2.26} hold. Given $\tau\in\mathbb{R}$, then $\Theta$ is $\mathcal{D}$-pullback asymptotically compact in $(\mathcal{P}_4(\ell^{2}\times\ell^{2}), d_{\mathcal{P}(\ell^{2}\times\ell^{2})})$.	
\end{lem}
\begin{proof}
Let $\tau\in\mathbb{R}$, $D\in\mathcal{D}$, $t_{n}\to\infty$ and $\mu_{n}\in D(\tau-t_{n})$. We prove that $\{\Theta(t_{n},\tau-t_{n})\mu_{n}\}_{n=1}^{\infty}$ has a convergent subsequence.

Without loss of generality, let $\Phi_{n}\in L^{4}(\Omega,\ell^{2}\times\ell^{2})$ with $\mathcal{L}_{\Phi_{n}}=\mu_{n}$. Denote by $\Phi(\tau,\tau-t_{n},\Phi_{n})=\Phi(\tau)=\binom{u(\tau)}{v(\tau)}$ the solution of \eqref{eq:2.16} with initial data $\Phi_{n}$ at time $\tau-t_{n}$. We first show that the distributions of $\{\Phi(\tau)\}_{n=1}^{\infty}$ are tight in $\ell^{2}\times\ell^{2}$.	

Let $\rho:\mathbb{Z}\to[0,1]$ be a smooth cut-off function such that $\rho(i)=1$ for $|i|\geq 2$ and $\rho(i)=0$ for $|i|\leq 1$. For $m\in\mathbb{N}^{+}$ define  $\theta_{m}(i)=\theta(\frac{i}{m})$, $i\in\mathbb{Z}$. Then decompose the solution as
\begin{equation}\label{eq:5.13}
\Phi(\tau)= \rho_{m}\,\Phi(\tau)+ (1-\rho_{m})\,\Phi(\tau).
	\end{equation}
By Lemma~\ref*{lem:4.2}, there exists $N_1=N_1(\tau,D)\in\mathbb N^+$ such that for all
$m\in\mathbb N^+$ and all $n\ge N_1$,
\begin{equation}\label{eq:5.14}
\mathbb E\|(1-\rho_m)\Phi(\tau)\|_{\ell^{2}\times\ell^{2}}^2
=\mathbb E\sum_{i\in\mathbb Z}(1-\rho_m(i))^2|(u_i(\tau)|^2+|v_i(\tau)|^2)\le C_1,
	\end{equation}
where $C_1=C_1(\tau)$ is independent of $n$ and $m$. By Chebyshev's inequality, for all $m\in\mathbb N^+$, $n\ge N_1$ and $R>0$,
\begin{equation}\label{eq:5.15}
\mathbb P\Bigl(\|(1-\rho_m)(u(\tau),v(\tau))\|_{\ell^{2}\times\ell^{2}}^2>R\Bigr)\le\frac{C_1}{R}.
	\end{equation}
Hence for any $\epsilon>0$, there exists $R(\epsilon,\tau)>0$ such that for all $m\in\mathbb N^+$ and $n\ge N_1$,
\begin{equation}\label{eq:5.16}
\mathbb P\Bigl(\|(1-\rho_m)(u(\tau),v(\tau))\|_{\ell^{2}\times\ell^{2}}>R(\epsilon,\tau)\Bigr)<\epsilon.
\end{equation}
For any fixed $m\in\mathbb N^+$, consider the closed subset of $\ell^{2}\times\ell^{2}$:
\begin{equation}\label{eq:5.17}
\mathcal Z_\epsilon:=\Bigl\{(w^{(1)},w^{(2)})\in\ell^{2}\times\ell^{2}:\|(w^{(1)},w^{(2)})\|_{\ell^{2}\times\ell^{2}}\le R(\epsilon,\tau),\ w^{(1)}_i=w^{(2)}_i=0\ \text{for}\ |i|>2m\Bigr\}.
\end{equation}
We claim $\mathcal Z_\epsilon$ is compact in $\ell^{2}\times\ell^{2}$. Indeed, restricting $w=(w^{(1)},w^{(2)})$ to $\{|i|\le 2m\}$, we obtain $w|_{\{|i|\le 2m\}}\in\widehat{\mathcal Z}_\epsilon$, where
\begin{equation*}
\widehat{\mathcal Z}_\epsilon
:=\Bigl\{\bigl(w^{(1)}_i,w^{(2)}_i\bigr)_{|i|\le 2m}\in\mathbb R^{2(2m+1)}:
\sum_{|i|\le 2m}|w^{(1)}_i|^2+|w^{(2)}_i|^2\le R^2(\epsilon,\tau)\Bigr\},
\end{equation*}
and $\widehat{\mathcal Z}_\epsilon$ is bounded, hence relatively compact in $\mathbb R^{2(2m+1)}$.
By the natural identification $\mathcal Z_\epsilon\cong\widehat{\mathcal Z}_\epsilon$,
$\mathcal Z_\epsilon$ is compact in $\ell^{2}\times\ell^{2}$. From \eqref{eq:5.16} we obtain, for all $m\in\mathbb N^+$ and $n\ge N_1$,
\begin{equation}\label{eq:5.18}
\mathbb P\Bigl((1-\rho_m)\Phi(\tau)\in\mathcal Z_\epsilon\Bigr)>1-\epsilon.
\end{equation}
Since $\epsilon>0$ is arbitrary, \eqref{eq:5.18} implies that for each $m\in\mathbb N^+$,
\begin{equation}\label{eq:5.19}
\Bigl\{\mathcal L_{(1-\rho_m)\Phi(\tau,\tau-t_{n},\Phi_{n})}\Bigr\}_{n=N_1}^\infty
\ \text{is tight in}\ \mathcal P(\ell^{2}\times\ell^{2}).
\end{equation}	

By Lemma~\ref*{lem:4.3} (uniform tail estimate), for every $\varepsilon>0$ there exist
$N_2=N_2(\varepsilon,\tau,D)\in\mathbb N^+$ and $m_0=m_0(\varepsilon,\tau,D)\in\mathbb N^+$ such that for all $n\ge N_2$,
\begin{equation}\label{eq:5.20}
\mathbb E\|\rho_{m_0}(u(\tau),v(\tau))\|_{\ell^{2}\times\ell^{2}}^2
=\mathbb E\sum_{|i|\ge 2m_0}|u_i(\tau)|^2+|v_i(\tau)|^2
<\frac{\varepsilon^2}{9}.
\end{equation}	
From \eqref{eq:5.19}, the sequence $\bigl\{\mathcal L_{(1-\rho_{m_0})(u_n(\tau),v_n(\tau))}\bigr\}_{n=N_2}^\infty$ is tight in $\mathcal P(\ell^{2}\times\ell^{2})$, hence relatively compact by Prokhorov.
Therefore there exist $n_1,\ldots,n_l\ge N_2$ such that
\begin{equation}\label{eq:5.21}
\bigcup_{n=N_2}^\infty
\mathcal L_{(1-\rho_{m_0})(u_n(\tau),v_n(\tau))}
\subseteq
\bigcup_{j=1}^l
\mathcal B\!\Bigl(\mathcal L_{(1-\rho_{m_0})(u_{n_j}(\tau),v_{n_j}(\tau))},\frac{\varepsilon}{3}\Bigr).
\end{equation}
Consequently, for any $n\ge N_2$, there exists $j_0\in\{1,\ldots,l\}$ such that
\begin{equation}\label{eq:5.22}
d_{\mathcal P(\ell^{2}\times\ell^{2}}
\Bigl(\mathcal L_{(1-\rho_{m_0})(u_n(\tau),v_n(\tau))},
\mathcal L_{(1-\rho_{m_0})(u_{n_{j_0}}(\tau),v_{n_{j_0}}(\tau))}\Bigr)
<\frac{\varepsilon}{3}.
\end{equation}

Now use the definition of the Wasserstein distance and bounded Lipschitz functions. For any $\psi\in\mathrm{Lip}_b(\ell^{2}\times\ell^{2})$ with $\|\psi\|_{\mathrm{Lip}}\le1$, we have
\begin{align*}
	&d_{\mathcal P(\ell^{2}\times\ell^{2})}
	\Bigl(\mathcal L_{(u(\tau),v(\tau))},
	\mathcal L_{(u_{n_{j_0}}(\tau),v_{n_{j_0}}(\tau))}\Bigr) \\[4pt]
	&\ \le
	\mathbb E\Bigl|\psi(u(\tau),v(\tau))
	-\psi\bigl((1-\rho_{m_0})(u(\tau),v(\tau))\bigr)\Bigr|
	\\[4pt]
	&\ \ +\mathbb E\Bigl|\psi\bigl((1-\rho_{m_0})(u(\tau),v(\tau))\bigr)
	-\psi\bigl((1-\rho_{m_0})(u_{n_{j_0}}(\tau),v_{n_{j_0}}(\tau))\bigr)\Bigr|
	\\[4pt]
	&\ \ +\mathbb E\Bigl|\psi\bigl((1-\rho_{m_0})(u_{n_{j_0}}(\tau),v_{n_{j_0}}(\tau))\bigr)
	-\psi(u_{n_{j_0}}(\tau),v_{n_{j_0}}(\tau))\bigr)\Bigr| \\[4pt]
	&\ \le
	\mathbb E\|\rho_{m_0}(u(\tau),v(\tau))\|_{\ell^{2}\times\ell^{2}}
	+\mathbb E\|\rho_{m_0}(u_{n_{j_0}}(\tau),v_{n_{j_0}}(\tau))\|_{\ell^{2}\times\ell^{2})}
	\\[4pt]
	&\ \ +d_{\mathcal P(\ell^{2}\times\ell^{2})}
	\Bigl(\mathcal L_{(1-\rho_{m_0})(u(\tau),v(\tau))},
	\mathcal L_{(1-\rho_{m_0})(u_{n_{j_0}}(\tau),v_{n_{j_0}}(\tau))}\Bigr) \\[4pt]
	&\ \ <\frac{\varepsilon}{3}+\frac{\varepsilon}{3}+\frac{\varepsilon}{3}
	=\varepsilon,
\end{align*}
where we used $|\psi(a)-\psi(b)|\le\|a-b\|_{\ell^{2}\times\ell^{2}}$ (since $\|\psi\|_{\mathrm{Lip}}\le1$)
and \eqref{eq:5.20}, \eqref{eq:5.22}.
Since $\varepsilon>0$ is arbitrary, this shows
\begin{equation}\label{eq:5.23}
\Bigl\{\mathcal L_{(u_n(\tau),v_n(\tau))}\Bigr\}_{n=1}^\infty \ \text{is tight in}\ \mathcal P(\ell^{2}\times\ell^{2}).
\end{equation}
Hence there exists $\varpi\in\mathcal P(\ell^{2}\times\ell^{2})$ such that, up to a subsequence,
\begin{equation}\label{eq:5.24}
\mathcal L_{\Phi(\tau,\tau-t_{n},\Phi_{n})}\rightharpoonup\varpi\qquad(n\to\infty).
\end{equation}

It remains to show that $\varpi\in\mathcal P_4(\ell^{2}\times\ell^{2})$. Let  $K=\{K(\tau):\tau\in\mathbb R\}$ be the closed $\mathcal D$-pullback absorbing set
given by Lemma~\ref*{lem:5.2}. Since $D\in\mathcal D$, there exists $N_3=N_3(\tau,D)\in\mathbb N^+$ such that for all $n\ge N_3$,
\begin{equation}\label{eq:5.25}
\mathcal L_{\Phi(\tau)}\in K(\tau).
\end{equation}
Because $K(\tau)$ is closed in $\mathcal P_4(\ell^{2}\times\ell^{2})$ and the convergence in \eqref{eq:5.24} is weak convergence in $\mathcal P(\ell^{2}\times\ell^{2})$, \eqref{eq:5.25} implies
$\varpi\in K(\tau)\subset\mathcal P_4(\ell^{2}\times\ell^{2})$. Thus $\varpi\in\mathcal P_4(\ell^{2}\times\ell^{2})$, which completes the proof.
\end{proof}
 
 In summary, we have proven the existence and uniqueness of $D$-pullback measure attractors for $\Theta$.
 
\begin{thm}
	\label{th5.4} Assume that (A1)-(A3), \eqref{eq:2.21}--\eqref{eq:2.26} hold. Then, $\Theta$ has a unique $D$-pullback measure attractor $\mathcal{A} = \{ \mathcal{A}(\tau) : \tau \in \mathbb{R} \} \in \mathcal{D}$.
\end{thm}
 \begin{proof} The existence and uniqueness of the measure attractor $\mathcal{A}$ follow from Proposition 2.5 based on Lemmas 5.1, 5.2 and 5.3. 
\end{proof}

\section{Upper semicontinuity of pullback measure attractors}
\indent
\par

In this section, we study the upper semicontinuity of pullback measure attractors
for the non-autonomous distribution-dependent stochastic lattice wave equation \eqref{eq:1.2}
with respect to the M-V perturbation parameter $\varepsilon\to0^{+}$.

To this end, we consider the parameter family equation \eqref{eq:1.2} and the limit equation \eqref{eq:1.3}. Unless otherwise specified, assume that all $f^{\varepsilon}=(f_i^{\varepsilon})_{i\in\mathbb Z}$
and $\sigma^{\varepsilon}=(\sigma_i^{\varepsilon})_{i\in\mathbb Z}$
satisfy conditions {\rm (A1)}--{\rm (A3)} uniformly with respect to $\varepsilon\in(0,1)$,
and that $f=(f_i)_{i\in\mathbb Z}$ and $\sigma=(\sigma_i)_{i\in\mathbb Z}$
satisfy the corresponding conditions without the $\mathcal{L}_{u_i}$ dependence.
Furthermore, assume that \eqref{eq:2.21}--\eqref{eq:2.26} hold.

In addition, we assume that for all $i\in\mathbb Z$, $\varepsilon\in(0,1)$, $t\in\mathbb R$,
$u\in\mathbb R$ and $\mu\in\mathcal P_2(\mathbb R)$,

\begin{equation}\label{eq:6.1}
	|f_i^{\varepsilon}(t,u,\mu)-f_i(t,u)|
	\le\varepsilon\,\tilde\varphi_i(t)\bigl(|u|+\mathbb{W}^{\mathbb{R}}_2(\mu,\delta_0)\bigr),
\end{equation}
\begin{equation}\label{eq:6.2}
	|\sigma_i^{\varepsilon}(t,u,\mu)-\sigma_i(t,u)|
	\le\varepsilon\,\tilde\psi_i(t)\bigl(|u|+\mathbb{W}^{\mathbb{R}}_2(\mu,\delta_0)\bigr),
\end{equation}
where $\tilde\varphi=(\tilde\varphi_i)_{i\in\mathbb Z}$,
$\tilde\psi=(\tilde\psi_i)_{i\in\mathbb Z}\in L^{\infty}(\mathbb R,\ell^2)$.

As before, we rewrite \eqref{eq:1.2} and \eqref{eq:1.3} in the abstract form on the phase space $\ell^2\times\ell^2$.
For $\varepsilon\in(0,1]$,
\begin{equation}\label{eq:6.3}
	\begin{aligned}
	&d\dot{u}^{\varepsilon}(t)+\alpha u^{\varepsilon}(t)\,dt
	-\nu Au^{\varepsilon}(t)\,dt+\lambda u^{\varepsilon}(t)\,dt
	\\&=\bigl(f^{\varepsilon}(t,u^{\varepsilon}(t),\mathcal L_{u^{\varepsilon}(t)})+g(t)\bigr)\,dt
	+\sigma^{\varepsilon}(t,u^{\varepsilon}(t),\mathcal L_{u^{\varepsilon}(t)})\,dW(t),
	\qquad t>\tau,\\
	\text{with}\ \ &v^{\varepsilon}(t)=\dot{u}^{\varepsilon}(t)+\delta u^{\varepsilon}(t),
	\end{aligned}
\end{equation}
and the limit equation (state-dependent, no distribution dependence):
\begin{equation}\label{eq:6.4}
	\begin{aligned}
	&d\dot{u}(t)+\alpha u(t)\,dt
	-\nu Au(t)\,dt+\lambda u(t)\,dt
	=\bigl(f(t,u(t))+g(t)\bigr)\,dt
	+\sigma(t,u(t))\,dW(t),
	\qquad t>\tau,
	\\\text{with}\ \ &v(t)=\dot{u}(t)+\delta u^{\varepsilon}(t).
\end{aligned}
\end{equation}

By \eqref{eq:6.2}, for all $t\in\mathbb R$, $u\in\ell^2$ and $\mu\in\mathcal P_2(\ell^2)$,
\begin{equation}\label{eq:6.5}
	\bigl\|\sigma^{\varepsilon}(t,u,\mu)-\sigma(t,u)\bigr\|_{L_2(\ell^2,\ell^2)}^2
	\le 2\varepsilon^2\|\tilde\psi(t)\|^2
	\bigl(\|u\|^2+\rho(\mu,\hat{\delta_0})^2\bigr).
\end{equation}

We denote by $\Phi^{\varepsilon}(t,\tau,\Phi_{\tau})=\binom{u^{\varepsilon}(t,\tau,u_{\tau})}{v^{\varepsilon}(t,\tau,v_{\tau})}$ the solution of \eqref{eq:6.3} with the initial condition $\Phi^{\varepsilon}_{\tau}=\binom{u_{\tau}}{v_{\tau}}$. The next lemma is about the convergence of solutions of \eqref{eq:6.3} as $\varepsilon\to0^{+}$.
\begin{lem}
	\label{lem:6.1}
Assume that \eqref{eq:6.1}--\eqref{eq:6.2} hold. Then for every  $\tau\in\mathbb R$, $T>0$ and $R>0$, the solutions of \eqref{eq:6.3} and \eqref{eq:6.4} satisfy
\begin{align}\label{eq:6.6}
	\lim_{\varepsilon\to0^{+}}
	\sup_{\|\Phi_{\tau}\|_{L^2(\Omega,\ell^2\times\ell^2)}^2\le R}
	\sup_{\tau\le t\le \tau+T}
	\mathbb E\Bigl(\bigl\|\Phi^{\varepsilon}(t,\tau,\Phi_{\tau})
	-\Phi(t,\tau,\Phi_{\tau})\bigr\|_{\ell^2\times\ell^2}^2\Bigr)=0.
\end{align}	
\end{lem}
\begin{proof}
Let $w(t)=u^{\varepsilon}(t)-u(t)$ and $z(t)=v^{\varepsilon}(t)-v(t)$. From \eqref{eq:6.3} and \eqref{eq:6.4} we obtain that for all $t>\tau$,
\begin{align*}
	d(\Phi^{\varepsilon}(t)-\Phi(t))=[F(\Phi^{\varepsilon}(t))-F(\Phi(t))]dt+[G(\Phi^{\varepsilon}(t))-G(\Phi(t))]dW(t).
\end{align*}
By the It\^{o} formula, taking the expectation, we get that for all $t\ge\tau$,
\begin{equation}
	\begin{aligned}
		\label{eq:6.7}
		\mathbb{E}\bigl\|\Phi^{\varepsilon}(t)
		-\Phi(t)\bigr\|^2
		\le&2\mathbb{E}(\int_{\tau}^{t}(f^{\varepsilon}(s,u^{\varepsilon}(s),\mathcal L_{u^{\varepsilon}(s)})-f(s,u(s)),z(s))ds)\\&+\mathbb{E}(\int_{\tau}^{t}\|\sigma^{\varepsilon}(s,u^{\varepsilon}(s),\mathcal L_{u^{\varepsilon}(s)})-\sigma(s, u(s))\|_{L_2(\ell^2,\ell^2)}^2ds).
	\end{aligned}
\end{equation}	
We estimate each term in the right-hand side of \eqref{eq:6.7}, respectively. For the first term,  we have that
\begin{equation*}
	\begin{aligned}
		&2\mathbb{E}(\int_{\tau}^{t}(f^{\varepsilon}(s,u^{\varepsilon}(s),\mathcal L_{u^{\varepsilon}(s)})-f(s,u(s)),z(s))ds)
		\\=&2\mathbb{E}(\int_{\tau}^{t}(f^{\varepsilon}(s,u^{\varepsilon}(s),\mathcal L_{u^{\varepsilon}(s)})-f^{\varepsilon}(s,u(s),L_{u^{\varepsilon}(s)}),z(s))ds)
		\\&+2\mathbb{E}(\int_{\tau}^{t}(f^{\varepsilon}(s,u(s),L_{u^{\varepsilon}(s)})-f^{\varepsilon}(s,u(s),L_{u(s)}),z(s))ds)
		\\&+2\mathbb{E}(\int_{\tau}^{t}(f^{\varepsilon}(s,u(s),L_{u(s)})-f(s,u(s)),z(s))ds).
	\end{aligned}
\end{equation*}
Along with \eqref{eq:2.2}, \eqref{eq:2.3}, \eqref{eq:2.8} and \eqref{eq:6.1}, we obtain that
\begin{equation}
	\begin{aligned}
		\label{eq:6.8}
		&2\mathbb{E}(\int_{\tau}^{t}(f^{\varepsilon}(s,u^{\varepsilon}(s),\mathcal L_{u^{\varepsilon}(s)})-f(s,u(s)),z(s))ds)
		\\\le&2\mathbb{E}(\int_{\tau}^{t}\|\eta(s)\|\|w(s)\|\|z(s)\|ds)
		+2\mathbb{E}(\int_{\tau}^{t}\sum_{i \in \mathbb{Z}}\varphi_i(s)(\mathbb{E}\left|w_i(s)\right|^2)^\frac{1}{2}\left|z_i(s)\right|ds)
		\\&+2\varepsilon\mathbb{E}(\int_{\tau}^{t}\sum_{i \in \mathbb{Z}}\tilde{\varphi}_i(s)(\left|u_i(s)\right|+(\mathbb{E}\left|w_i(s)\right|^2)^\frac{1}{2})\left|z_i(s)\right|ds)
		\\\le&\|\eta\|_{L^\infty([\tau,T],\ell^2)}\int_{\tau}^{t}\mathbb{E}(\|w(s)\|^2+\|z(s)\|^2)ds
		+2\|\varphi\|_{L^\infty([\tau,T],\ell^2)}\int_{\tau}^{t}\mathbb{E}(\|w(s)\|^2+\|z(s)\|^2)ds
		\\&+2\varepsilon\mathbb{E}(\int_{\tau}^{t}\|\tilde{\varphi}(s)\|(\|u(s)\|^2+\mathbb{E}\|u(s)\|^2)ds)+\varepsilon\mathbb{E}(\int_{\tau}^{t}\|\tilde{\varphi}(s)\|\|z(s)\|^2ds)
		\\\le&(\|\eta\|_{L^\infty([\tau,T],\ell^2)}+\|\varphi\|_{L^\infty([\tau,T],\ell^2)}+\varepsilon\|\tilde{\varphi}\|_{L^\infty([\tau,T],\ell^2)})\mathbb{E}(\int_{\tau}^{t}\bigl\|\Phi^{\varepsilon}(s)
		-\Phi(s)\bigr\|^2ds)
		\\&+\|\tilde{\varphi}\|_{L^\infty([\tau,T],\ell^2)}\int_{\tau}^{t}\mathbb{E}\|u(s)\|^2ds.
	\end{aligned}
\end{equation}
For the second term, on the right-hand side of \eqref{eq:6.7}, we have that
\begin{equation*}
	\begin{aligned}
		&\mathbb{E}(\int_{\tau}^{t}\|\sigma^{\varepsilon}(s,u^{\varepsilon}(s),\mathcal L_{u^{\varepsilon}(s)})-\sigma(s, u(s))\|_{L_2(\ell^2,\ell^2)}^2ds)
		\\\le&\mathbb{E}(\int_{\tau}^{t}\|\sigma^{\varepsilon}(s,u^{\varepsilon}(s),\mathcal L_{u^{\varepsilon}(s)})-\sigma^{\varepsilon}(s, u(s),\mathcal L_{u(s)})\|_{L_2(\ell^2,\ell^2)}^2ds)
		\\&+\mathbb{E}(\int_{\tau}^{t}\|\sigma^{\varepsilon}(s, u(s),\mathcal L_{u(s)})-\sigma(s,u(s))\|_{L_2(\ell^2,\ell^2)}^2ds).
	\end{aligned}
\end{equation*}
By \eqref{eq:2.12}, \eqref{eq:2.8} and \eqref{eq:6.5}, we obtain that
\begin{equation}
	\begin{aligned}
		\label{eq:6.9}
		&\mathbb{E}(\int_{\tau}^{t}\|\sigma^{\varepsilon}(s,u^{\varepsilon}(s),\mathcal L_{u^{\varepsilon}(s)})-\sigma(s, u(s))\|_{L_2(\ell^2,\ell^2)}^2ds)
		\\\le&2\mathbb{E}(\int_{\tau}^{t}\|\kappa(s)\|^2(\|w(s)\|^2+\mathbb{E}\|w(s)\|^2)ds)+2\varepsilon^2\mathbb{E}(\int_{\tau}^{t}\|\tilde\psi(s)\|^2(\|u(s)\|^2+\mathbb{E}\|u(s)\|^2)ds)
		\\\le&4\|\kappa\|^2_{L^\infty(\mathbb R,\ell^2)}\mathbb{E}(\int_{\tau}^{t}\bigl\|\Phi^{\varepsilon}(s)
		-\Phi(s)\bigr\|^2ds)
		+4\varepsilon^2\|\tilde\psi\|^2_{L^\infty([\tau,T],\ell^2)}\int_{\tau}^{t}\mathbb{E}\|u(s)\|^2ds.
	\end{aligned}
\end{equation}
Thus, from \eqref{eq:6.7}-\eqref{eq:6.9}, there exist constants $C_2$ and $C_3$ such that
\begin{equation*}
	\begin{aligned}
		\mathbb{E}\bigl\|\Phi^{\varepsilon}(t)
		-\Phi(t)\bigr\|^2
		\le C_2\mathbb{E}(\int_{\tau}^{t}\bigl\|\Phi^{\varepsilon}(s)
		-\Phi(s)\bigr\|^2ds)+\varepsilon C_3\int_{\tau}^{t}\mathbb{E}\|u(s)\|^2ds.
	\end{aligned}
\end{equation*}
Then by Lemma \ref{lem:4.2}, we know that for all $t\in[\tau,\tau+T]$ and $\varepsilon\in(0,1)$,
\begin{equation*}
	\begin{aligned}
		\mathbb{E}\bigl\|\Phi^{\varepsilon}(t)
		-\Phi(t)\bigr\|^2
		\le C_2\mathbb{E}(\int_{\tau}^{t}\bigl\|\Phi^{\varepsilon}(s)
		-\Phi(s)\bigr\|^2ds)+\varepsilon C_3C_4(1+\|\Phi_{\tau}\|_{L_2(\Omega,\ell^2\times\ell^2)}^2)T,
	\end{aligned}
\end{equation*}
where $C_4=C_4(\tau,T)>0$ is a constant independent of $\varepsilon$. Then by Gronwall's inequality, we have for all $t\in[\tau,\tau+T]$ and $\varepsilon\in(0,1)$ and $\|\Phi_{\tau}\|_{L_2(\Omega,\ell^2\times\ell^2)}\le R$,
\begin{equation*}
	\begin{aligned}
		\mathbb{E}\bigl\|\Phi^{\varepsilon}(t)
		-\Phi(t)\bigr\|^2
		\le \varepsilon C_3C_4(1+R^2)e^{C_2 T} ,
	\end{aligned}
\end{equation*}
which concludes the proof.
\end{proof}

As an immediate consequence of Lemma \ref{lem:6.1}, we obtain the following convergence result.

\begin{Corollary}\label{coro:6.1}
	Assume that \eqref{eq:6.1}--\eqref{eq:6.2} hold.
	Then for every $t\in\mathbb R^+$, $\tau\in\mathbb R$ and $R>0$,
	\begin{align}\label{eq:6.10}
		\lim_{\varepsilon\to0^{+}}
		\sup_{\substack{\mu\in\mathcal P_4(\ell^2\times\ell^2)\\
				\mu(\|\cdot\|^2)\le R}}
		\mathbb W_2\bigl(\Theta^{\varepsilon}(t,\tau)\mu,\Theta(t,\tau)\mu\bigr)=0,
	\end{align}
	and hence
	\begin{align}\label{eq:6.11}
		\lim_{\varepsilon\to0^{+}}
		\sup_{\substack{\mu\in\mathcal P_4(\ell^2\times\ell^2)\\
				\mu(\|\cdot\|^2)\le R}}
		d_{\mathcal P(\ell^2\times\ell^2)}\bigl(\Theta^{\varepsilon}(t,\tau)\mu,
		\Theta(t,\tau)\mu\bigr)=0.
	\end{align}
\end{Corollary}

Denote by $\Theta^{\varepsilon}$ and $\Theta$ the non-autonomous measure dynamical
systems associated with \eqref{eq:6.3} and \eqref{eq:6.4}, and let their $\mathcal D$-pullback
measure attractors be
$\mathcal A^{\varepsilon}=\{\mathcal A^{\varepsilon}(\tau)\}_{\tau\in\mathbb R}$
and $\mathcal A=\{\mathcal A(\tau)\}_{\tau\in\mathbb R}$, respectively.
The main theorem of this chapter proves upper semicontinuity.

\begin{thm}\label{th:6.1}
	Assume that \eqref{eq:6.1}--\eqref{eq:6.2} hold.
	Then for every $\tau\in\mathbb R$,
	\begin{align}\label{eq:6.12}
		\lim_{\varepsilon\to0^{+}}
		d_{\mathcal P_4(\ell^2\times\ell^2)}\!\bigl(\mathcal A^{\varepsilon}(\tau),\,
		\mathcal A(\tau)\bigr)=0,
	\end{align}
	where $d_{\mathcal P_4(\ell^2\times\ell^2)}(\cdot,\cdot)$ denotes the Hausdorff
	semi-distance in $\mathcal P_4(\ell^2\times\ell^2)$.
\end{thm}
\begin{proof}
Since $f^{\varepsilon}, g^{\varepsilon}, \sigma^{\varepsilon}$ and $f, g, \sigma$ satisfy (A1)--(A3) uniformly with respect to $\varepsilon \in (0, 1)$, we have that the family $K = \{ K(\tau) : \tau \in \mathbb{R} \}$ given by Lemma \ref{lem:5.2} is a closed $\mathcal{D}$-pullback absorbing set of $\Theta$ and $\Theta^{\varepsilon}$ for all $\varepsilon \in (0, 1)$. In particular, for every $\tau \in \mathbb{R}$, $\mathcal{A}(\tau)$ and $\mathcal{A}^{\varepsilon}(\tau)$ are subsets of $K(\tau)$ for all $\varepsilon \in (0, 1)$.

Since $\mathcal{A}$ is the $\mathcal{D}$-pullback measure attractor of $\Phi$ in $(\mathcal{P}_4(\ell^2\times\ell^2), d_{\mathcal{P}(\ell^2\times\ell^2)})$, we obtain that for every $\varpi > 0$, there exists $T = T(\varpi, \tau, K) > 0$ such that	
\begin{align}\label{eq:6.13}
	d_{\mathcal P_4(\ell^2\times\ell^2)}\!
	\bigl(\Theta(T,\tau-T)K(\tau-T),\;\mathcal A(\tau)\bigr)<\frac{\varpi}{2}.
\end{align}
By Corollary \ref{coro:6.1}, there exists $\hat \varepsilon\in(0,1)$ such that
for all $\varepsilon\in(0,\hat \varepsilon)$,
\begin{align}\label{eq:6.14}
	\sup_{\mu\in K(\tau-T)}
	d_{\mathcal P_4(\ell^2\times\ell^2)}\!
	\bigl(\Theta^{\varepsilon}(T,\tau-T)\mu,\;\Theta(T,\tau-T)\mu\bigr)
	<\frac{\varpi}{2}.
\end{align}
From \eqref{eq:6.13} and \eqref{eq:6.14} we obtain
\begin{align}\label{eq:6.15}
	\sup_{\mu\in\mathcal A^{\varepsilon}(\tau-T)}
	d_{\mathcal P_4(\ell^2\times\ell^2)}\!
	\bigl(\Theta^{\varepsilon}(T,\tau-T)\mu,\;\mathcal A(\tau)\bigr)<\varpi,
	\qquad\forall\,\varepsilon\in(0,\hat \varepsilon).
\end{align}
 By the invariance of $\mathcal A^{\varepsilon}$,
$\mathcal A^{\varepsilon}(\tau)=\Theta^{\varepsilon}(T,\tau-T)\mathcal A^{\varepsilon}(\tau-T)$,
and \eqref{eq:6.15} imply
\begin{align}\label{eq:6.16}
	\sup_{\mu\in\mathcal A^{\varepsilon}(\tau)}
	d_{\mathcal P_4(\ell^2\times\ell^2)}\!
	\bigl(\mu,\;\mathcal A(\tau)\bigr)<\varpi,
	\qquad\forall\,\varepsilon\in(0,\hat \varepsilon).
\end{align}
Thus $\lim\limits_{\varepsilon\to0^{+}}
d_{\mathcal P_4(\ell^2\times\ell^2)}(\mathcal A^{\varepsilon}(\tau),\mathcal A(\tau))=0$.

Finally, combining Theorem \ref{th5.4} (attractor existence) and
Theorem \ref{th:6.1} (upper semicontinuity), when both attractors are singletons,
we immediately obtain the convergence of invariant measures.

\begin{Corollary}\label{coro:7.2}
Under the hypotheses of Theorems \ref{th5.4} and \ref{th:6.1},
if $f^{\varepsilon}$, $\sigma^{\varepsilon}$, $f$ and $\sigma$ are all time-independent,
then the unique invariant measure $\mu^{\varepsilon}$ of \eqref{eq:1.2} converges to the unique
invariant measure $\mu$ of \eqref{eq:1.3} in $\mathcal P_4(\ell^2\times\ell^2)$:
\begin{align}\label{eq:6.17}
\lim_{\varepsilon\to0^{+}}d_{\mathcal P_4(\ell^2\times\ell^2)}(\mu^{\varepsilon},\mu)=0.
\end{align}
\end{Corollary}
\end{proof}

\section*{Acknowledgement}
\indent
\par
The authors would like to thank the reviewers for their helpful comments. This work was partially supported by the
	National Natural Science Foundation of China(No.12326414), Sichuan Science and Technology Program(No.2023NSFSC0076).



\begin{thebibliography}{99}
	
	\bibitem{1} S. Chen, U. C. Tauber, \emph{Non-equilibrium relaxation in a stochastic lattice Lotka--Volterra model}, Phys. Biol. \textbf{13} (2016), 025005.
	
	\bibitem{2} B. Heiba, S. Chen, U. C. Tauber, \emph{Boundary effects on population dynamics in stochastic lattice Lotka--Volterra models}, Physica A \textbf{491} (2018), 582--590.
	
	\bibitem{3} P. E. Kloeden, X. Han, \emph{Dissipative Lattice Dynamical Systems}, World Scientific, Singapore, 2023.
	
	\bibitem{4} B. Wang, \emph{Dynamics of systems on infinite lattices}, J. Differential Equations \textbf{221} (2006), 224--245.
	
	\bibitem{5} P. W. Bates, H. Lisei, K. Lu, \emph{Attractors for stochastic lattice dynamical systems}, Stoch. Dyn. \textbf{6} (2006), 1--21.
	
	\bibitem{6} X. Han, W. Shen, S. Zhou, \emph{Random attractors for stochastic lattice dynamical systems in weighted spaces}, J. Differential Equations \textbf{250} (2011), 1235--1266.
	
	\bibitem{7} Z. Wang, S. Zhou, \emph{Existence and upper semicontinuity of attractors for non-autonomous stochastic lattice systems with random coupled coefficients}, Commun. Pure Appl. Anal. \textbf{15} (2016), 2221--2245.
	
	\bibitem{8} S. Zhou, L. Wei, \emph{A random attractor for a stochastic second order lattice system with random coupled coefficients}, J. Math. Anal. Appl. \textbf{395} (2012), 42--55.
	
	\bibitem{9} X. Han, \emph{Random attractors for second order stochastic lattice dynamical systems with multiplicative noise in weighted spaces}, Stoch. Dyn. \textbf{12} (2012), 1150024.
	
	\bibitem{10} B. Wang, \emph{Dynamics of stochastic reaction-diffusion lattice system driven by nonlinear noise}, J. Math. Anal. Appl. \textbf{477} (2019), 104--132.
	
	\bibitem{11} R. Wang, B. Wang, \emph{Random dynamics of lattice wave equations driven by infinite-dimensional nonlinear noise}, Discrete Contin. Dyn. Syst. Ser. B \textbf{25} (2020), 2461--2493.
	
	\bibitem{12} R. Wang, Y. Li, \emph{Asymptotic behavior of stochastic discrete wave equations with nonlinear noise and damping}, J. Math. Phys. \textbf{61} (2020), 052701.
	
	\bibitem{13} R. Wang, \emph{Long-time dynamics of stochastic lattice Plate equations with nonlinear noise and damping}, J. Dyn. Differential Equations \textbf{33} (2021), 1--37.
	
	\bibitem{14} S. Zeng, J. Long, R. Wang, \emph{Pullback measure attractors and evolution system of measures for $p$-Laplacian Ginzburg--Landau lattice systems driven by superlinear L\'evy noise}, Adv. Nonlinear Stud. \textbf{26} (2026), 144--195.
	
	\bibitem{15} J.-P. Eckmann, M. Hairer, \emph{Invariant measures for stochastic partial differential equations in unbounded domains}, Nonlinearity \textbf{14} (2001), 133--151.
	
	\bibitem{16} B. Wang, \emph{Attractors for reaction-diffusion equations in unbounded domains}, Phys. D \textbf{128} (1999), 41--52.
	
	\bibitem{17} R. Hu, D. Li, T. Zeng, \emph{Pullback measure attractors for non-autonomous stochastic FitzHugh-Nagumo system with distribution dependence on unbounded domains,} Commun. Nonlinear Sci. Numer.\textbf{151} (2025), 109047.
	
	\bibitem{18} D. Li, B. Wang, X. Wang, \emph{Limiting behavior of invariant measures of stochastic delay lattice systems}, J. Dyn. Differential Equations \textbf{34} (2022), 1453--1487.
	
	\bibitem{19} A. Gu, \emph{Weak pullback mean random attractors for stochastic evolution equations and applications}, Stoch. Dyn. \textbf{22} (2022), 1--16.
	
	\bibitem{20} H. P. McKean, \emph{A class of Markov processes associated with nonlinear parabolic equations}, Proc. Natl. Acad. Sci. \textbf{56} (1966), 1907--1911.
	
	\bibitem{21} A. A. Vlasov, \emph{The vibrational properties of an electron gas}, Sov. Phys. \textbf{10} (1968), 721--733.
	
	\bibitem{22} A. S. Sznitman, \emph{Topics in propagation of chaos}, in: \emph{\'Ecole d'\'Et\'e de Probabilit\'es de Saint-Flour XIX -- 1989}, Lecture Notes in Math. \textbf{1464}, Springer, Berlin, 1991, pp. 165--251.
	
	\bibitem{23} M. R\"ockner, X. Zhang, \emph{Well-posedness of distribution dependent SDEs with singular drifts}, Bernoulli \textbf{27} (2021), 1131--1158.
	
	\bibitem{24} X. Huang, M. R\"ockner, F. Wang, \emph{Nonlinear Fokker--Planck equations for probability measures on path space and path-distribution dependent SDEs}, Discrete Contin. Dyn. Syst. \textbf{39} (2019), 3017--3035.
	
	\bibitem{25} D. Ba\~nos, \emph{The Bismut--Elworthy--Li formula for mean-field stochastic differential equations}, Ann. Inst. Henri Poincar\'e Probab. Stat. \textbf{54} (2018), 220--233.
	
	\bibitem{26} W. Hong, S. Li, W. Liu, \emph{Large deviation principle for McKean--Vlasov quasilinear stochastic evolution equations}, Appl. Math. Optim. \textbf{84} (2021), 1119--1147.
	
	\bibitem{27} W. Liu, Y. Song, J. Zhai, T. Zhang, \emph{Large and moderate deviation principles for McKean--Vlasov SDEs with jumps}, Potential Analysis \textbf{59} (2023), 1141--1190.
	
	\bibitem{28} J. Bao, M. Scheutzow, C. Yuan, \emph{Existence of invariant probability measures for functional McKean--Vlasov SDEs}, Electron. J. Probab. \textbf{27} (2022), 1--14.
	
	\bibitem{29} M. Liang, M. B. Majka, J. Wang, \emph{Exponential ergodicity for SDEs and McKean--Vlasov processes with L\'evy noise}, Ann. Inst. H. Poincar\'e Probab. Statist. \textbf{57} (2021), 1665--1701.
	
	\bibitem{30} F. Wang, \emph{Exponential ergodicity for non-dissipative McKean--Vlasov SDEs}, Bernoulli \textbf{29} (2023), 1035--1062.
	
	\bibitem{31} L. Shi, J. Shen, B. Wang, \emph{Long-time dynamics of McKean-Vlasov stochastic reaction-diffusion equations on R\^n}, J. Differential Equations \textbf{453} (2026), 113855.
	
	\bibitem{32} L. Shi, J. Shen, K. Lu, \emph{Pullback measure attractors and limiting behaviors of McKean--Vlasov stochastic delay lattice systems}, Sci. China Math.‌ (2026), https://doi.org/10.1007/s11425-025-2542-5.
	
	\bibitem{33} F. Bai, Z. Chen, X. Sun, \emph{Existence and approximation of measure attractors and invariant measures for McKean--Vlasov stochastic lattice system with L\'evy noise}, J. Differential Equations \textbf{453} (2026), 113784.
	
	\bibitem{34} B. Schmalfuss, \emph{Long-time behaviour of the stochastic Navier--Stokes equation}, Math. Nachr. \textbf{152} (1991), 7--20.
	
	\bibitem{35} H. Crauel, \emph{Measure attractors and Markov attractors}, Dyn. Syst. \textbf{23} (2008), 75--107.
	
	\bibitem{36} B. Schmalfuss, \emph{Measure attractors and random attractors for stochastic partial differential equations}, Stoch. Anal. Appl. \textbf{17} (1999), 1075--1101.
	
	\bibitem{37} C. Marek, N. J. Cutland, \emph{Measure attractors for stochastic Navier--Stokes equations}, Electron. J. Probab. \textbf{8} (1998), 1--15.
	
	\bibitem{38} H. Morimoto, \emph{Attractors of probability measures for semilinear stochastic evolution equations}, Stoch. Anal. Appl. \textbf{10} (1992), 205--212.
	
	\bibitem{39} B. Schmalfuss, \emph{Qualitative properties for the stochastic Navier--Stokes equation}, Nonlinear Anal. \textbf{28} (1997), 1545--1563.
	
	\bibitem{40} C. Qiao, Y. Li, X. Tang, \emph{Enlarged measure attractors of superlinearly stochastic fractional discrete $p$-Laplace equations on Banach spaces}, Banach J. Math. Anal. \textbf{20} (2026), 32.
	
	\bibitem{41} D. Li, B. Wang, \emph{Pullback measure attractors for non-autonomous stochastic reaction-diffusion equations on thin domains}, J. Differential Equations \textbf{397} (2024), 232--261.
	
	\bibitem{42} R. Li, S. Mi, D. Li, \emph{Pullback measure attractors for non-autonomous stochastic 3D globally modified Navier--Stokes equations}, Qual. Theory Dyn. Syst. \textbf{23} (2024), 246.
	
	\bibitem{43} S. Mi, R. Li, D. Li, \emph{Pullback measure attractors for non-autonomous fractional stochastic reaction-diffusion equations on unbounded domains}, Appl. Math. Optim. \textbf{90} (2024), 55.
	
	\bibitem{44} S. Mi, D. Li, T. Zeng, \emph{Pullback measure attractors for non-autonomous stochastic lattice systems}, Proc. Roy. Soc. Edinburgh Sect. A (2024), 1--20.
	
	\bibitem{45} S. Weng, S. Mi, D. Li, \emph{Measure attractors of stochastic fractional lattice systems}, Fractal Fract. \textbf{8} (2024), 440.
	
	\bibitem{46} D. Li, R. Li, T. Zeng, \emph{Long-time behaviour of the non-autonomous stochastic FitzHugh--Nagumo systems on thin domains}, Dyn. Sys. \textbf{41} (2026), 252--284.
	
	\bibitem{47} S. Mi, R. Li, D. Li, \emph{Measure attractors of stochastic Schr\"odinger lattice systems driven by nonlinear noise}, Bull. Sci. Math. (2025), 103672.
	
	\bibitem{48} D. Li, X. Yang, G. Li, Y. Wu, \emph{Pullback measure attractors and evolution systems of measures for non-autonomous stochastic porous media lattice systems}, Commun. Nonlinear Sci. Numer. Simul. \textbf{152} (2026), 109227.
	
	\bibitem{49} T. Jiao, J. Leng, R. Wang, S. Zeng, \emph{Measure attractors and measure evolution systems in high-order moment metric space for fractional stochastic Ginzburg--Landau equations on unbounded domains}, Chaos Solitons Fractals \textbf{200} (2025), 116945.
	
	\bibitem{50} J. Zhang, X. Wang, D. Huang, \emph{Pullback measure attractors for non-autonomous stochastic lattice reversible Selkov systems}, Stoch. Dyn. (2025), 2540003.
	
	\bibitem{51} X. Mao, \emph{Stochastic Differential Equations and Applications}, 2nd ed., Woodhead Publishing, Cambridge, 2011.
	
	\bibitem{52} B. Wang, \emph{Sufficient and necessary criteria for existence of pullback attractors for non-compact random dynamical systems}, J. Differential Equations \textbf{253} (2012), 1544--1583.
	
	\bibitem{53} P. E. Kloeden, T. Lorenz, \emph{Mean-square random dynamical systems}, J. Differential Equations \textbf{253} (2012), 1422--1438.
	
	\bibitem{54} Y. Li, X. Tang, C. Qiao, \emph{Minimality and convexity of pullback measure attractors for fractional FitzHugh--Nagumo equations with superlinear noise}, Chaos \textbf{36} (2026), 033135.
	
	\bibitem{55} Z. Chen, D. Yang, S. Zhong, \emph{Large deviation principle for stochastic FitzHugh--Nagumo lattice systems}, Commun. Nonlinear Sci. Numer. Simul. \textbf{136} (2024), 108070.
	
	\bibitem{56} Y. Zhang, X. Li, \emph{Large deviation principles of invariant measures of stochastic discrete wave equations with multiplicative noise and nonlinear damping}, Bull. Sci. Math. \textbf{204} (2025), 103670.
\end{thebibliography}
\end{document}